\documentclass[10pt]{amsart}
\usepackage{graphics}
\usepackage[english]{babel}
\usepackage{amssymb,amsmath,amsfonts}
\usepackage{datetime}
\usepackage{color}
\usepackage{ulem}
\usepackage{lipsum}
\usepackage{amsfonts}
\usepackage{enumerate}
\usepackage{cite}
\usepackage{tikz}
\usetikzlibrary{matrix}
\usepackage[pdftex]{hyperref}
\RequirePackage{amscd}
\RequirePackage{epic}
\RequirePackage[all]{xy}
\RequirePackage{url}
\RequirePackage[shortlabels]{enumitem}
\usepackage{empheq}
\usepackage{epsfig}
\usepackage{lineno} 
 \usepackage{perpage}
 \MakePerPage{linenumbers}
\usepackage[english]{babel}

\usepackage[utf8]{inputenc}
\usepackage{cite}
\RequirePackage{amscd}
\RequirePackage{epic}
\RequirePackage{eepic}
\RequirePackage[all]{xy}
\RequirePackage{url}
\RequirePackage[shortlabels]{enumitem}
\usepackage{empheq}
\usepackage{nomencl}
\usepackage{epsfig}
\usepackage{graphicx}

\newcommand{\red}[1]{#1}
\newcommand{\comment}[1]{}
   
    \newcommand{\set}[1]{{\left\{#1\right\}}}
\newcommand{\pa}[1]{{\left(#1\right)}}
\newcommand{\sq}[1]{{\left[#1\right]}}

\newcommand{\abs}[1]{{\left|#1\right|}}
\newcommand{\norm}[1]{{\left |#1\right |}}

\newcommand{\T}{\mathbb{T}}
\newcommand{\Z}{\mathbb{Z}}
\newcommand{\R}{\mathbb{R}}
\newcommand{\C}{\mathbb{C}}
\newcommand{\teta}{\theta}
\newcommand{\dg}{{\mathtt{D}_\g}}
\newcommand{\dgp}{{\mathtt{D}_{\g}}}

\newcommand{\eps}{\varepsilon}

\renewcommand{\Re}{\operatorname{Re}}

\newcommand{\II}{I}
\newcommand{\pan}{{\mathcal Q}}

\newcommand{\co}[1]{\textit{#1}}
\newcommand{\gr}[1]{\textbf{#1}}

\newcommand{\upp}[1]{\uppercase{#1}}
\newcommand{\id}{\operatorname{Id}}
\newcommand{\nar}[1]{\norma{#1}_{r,s, \eta}}

\newcommand{\ad}{\operatorname{ad}}

\usepackage{amsthm}

\RequirePackage{url}
\newtheorem{prop}{Proposition}[section]
    \newtheorem{thm}{Theorem}
      \newtheorem{them}{Theorem}[]
    \newtheorem*{thm*}{Theorem}
    \newtheorem*{cor*}{Corollary}
    \newtheorem*{teo normal}{"Twisted Conjugacy" Theorem}

    \newtheorem{cor}{Corollary}
    \newtheorem{lemma}{Lemma}
    \theoremstyle{remark}
     \newtheorem{ex}{Example}
\newtheorem{rmk}{Remark}[section]
\theoremstyle{definition}
\newtheorem{defn}{Definition}

\numberwithin{equation}{section}
\numberwithin{thm}{section}
\numberwithin{defn}{section}
\numberwithin{prop}{section}
\numberwithin{cor}{section}
\numberwithin{lemma}{section}
\numberwithin{rmk}{section}

\newcommand{\g}{\gamma}

\newcommand{\s}{{\sigma}}

\def\wc{ {}}

\newcommand{\N}{{\mathbb N}}

\newcommand{\cA}{{\mathcal A}}

\newcommand{\cC}{{\mathcal C}}

\newcommand{\cH}{{\mathcal H}}

\newcommand{\cK}{{\mathcal K}}
\newcommand{\cL}{{\mathcal L}}

\newcommand{\cN}{{\mathcal N}}
\newcommand{\cO}{{\mathcal O}}

\newcommand{\cQ}{{\mathcal Q}}
\newcommand{\cR}{{\mathcal R}}
\newcommand{\cS}{{\mathcal S}}
\newcommand{\cT}{{\mathcal T}}

\newcommand{\fA}{{\bf A}}

\newcommand{\ta}{{\mathtt{a}}}

\newcommand{\tH}{{\mathtt{H}}}
\newcommand{\tK}{{\mathtt{K}}}

\newcommand{\be}{{\bf e}}

\usepackage{bm}

\newcommand{\al}{{\alpha}}
\newcommand{\bt}{{\beta}}

\renewcommand{\d}{\partial}

\newcommand\norma[1]{\left\lVert#1\right\rVert}

\newcommand{\im}{{\rm i}}
\newcommand{\jap}[1]{\langle #1 \rangle}
\newcommand{\und}[1]{\underline{#1}}

\newcommand{\e}{{\varepsilon}}

\newcommand{\meas}{{\rm meas}}

\newcommand{\tw}{{\mathtt{w}}}
\newcommand{\nnorm}[1]{{\left\vert\kern-0.25ex\left\vert\kern-0.25ex\left\vert #1 
    \right\vert\kern-0.25ex\right\vert\kern-0.25ex\right\vert}}
\newcommand{\coef}[1]{#1_{m,\alpha,\beta}}    

\newcommand{\bcoef}[1]{#1_{\bal,\bbt}}

\newcommand{\aluno}{\alpha^{1}}
\newcommand{\aldue}{\alpha^{2}}
\newcommand{\btuno}{\beta^{1}}
\newcommand{\btdue}{\beta^{2}}

\newcommand{\uuu}{u^{\aluno}\bar{u}^{\btuno}}
\newcommand{\uud}{u^{\aldue}\bar{u}^{\btdue}}
\newcommand{\buu}{{u^\bal}{\bar{u}^\bbt}}

\newcommand{\pon}{{{\Pi^{0,\mathcal{K}}}}} 
\newcommand{\por}{{\Pi^{0,\mathcal{R}}}}
\newcommand{\pd}{{\Pi^{-2}}}

\newcommand{\modi}[1]{\abs{u_{#1}}^2}

\newcommand{\es}{e^{\set{S,\cdot}}}
\newcommand{\bal}{{\bm \al}}
\newcommand{\bbt}{{\bm \bt}}

\newcommand{\Heta}{\cH_{r,s, \eta}}
\newcommand{\peseta}{e^{\eta\abs{\pi(\al - \bt)}}} 
\newcommand{\nore}[1]{\abs{#1}_{r,s,\eta}}
\newcommand{\noreq}[1]{\abs{#1}_{r,s,\eta} } 
\newcommand{\divisor}[1]{#1\cdot\pa{\bal - \bbt}}
\newcommand{\zero}[1]{#1^{(0)}}
\newcommand{\zeroR}[1]{#1^{(0,\cR)}}
\newcommand{\zeroK}[1]{#1^{(0,\cK)}}
\newcommand{\due}[1]{#1^{(-2)}}
\newcommand{\buon}[1]{#1^{\ge 2}}
\newcommand{\call}{\mathcal{L}}

\newcommand{\Calg}{\mathtt{C}_{\mathtt{alg}}}
\usepackage{framed,enumitem} 

\newcommand{\dueR}[1]{#1^{(-2,\cR)}}

\newcommand{\dueK}[1]{#1^{(-2,\cK)}}

\newcommand{\croc}{{\mathtt C}}
\newcommand{\crac}{{\mathtt C'}}
\newcommand{\ws}{\tw_s}
\newcommand{\wsap}{\tw^\infty_{p,s,a}}

\renewcommand{\o}{{\omega}}

\newcommand{\betta}{{\omega}}

\newcommand{\ub}{v}
\newcommand{\zb}{z}
\newcommand{\Vb}{{U}}
\newcommand{\Wb}{{ W}}
\newcommand{\Ib}{{J}}

\newcommand{\bo}{{\alpha}}
\newcommand{\bO}{{ \Omega}}
\newcommand{\bl}{{\nu}}
\newcommand{\bmu}{{\mu}}

\begin{document}
\author{Luca Biasco, Jessica Elisa Massetti \& Michela Procesi}
 
 \title{Almost periodic invariant tori for the NLS on the circle}
\begin{abstract}
In this paper we study the existence and linear stability of almost periodic solutions for a NLS equation on the circle with external parameters. Starting from the seminal result of Bourgain in \cite{Bourgain:2005} on the quintic NLS, we propose a novel approach allowing to prove in a unified framework  the persistence of finite and infinite dimensional invariant tori, which are the support of the desired solutions. The persistence result is given through a rather abstract "counter-term theorem" \`a la Herman, directly in the original elliptic variables without passing to action-angle ones. Our framework allows us to find "many more" almost periodic solutions with respect to the existing literature and consider also non-translation invariant PDEs.

\end{abstract} 
\maketitle
\setcounter{tocdepth}{1} 
\tableofcontents

\section{Introduction and main results}

A   classical result in KAM theory for  PDEs is the existence and stability of quasi-periodic solutions for semilinear Hamiltonian PDEs on the circle with an elliptic fixed point. The first pioneering works on this topic (see \cite{Kuksin-Poschel:1996},\cite{W})  were direct generalizations of the corresponding results on elliptic tori for finite dimensional Hamiltonian systems and dealt with the existence and linear stability of finite dimensional invariant tori which are the support of a quasi-periodic solution.
In order to avoid resonances and simplify the inherent small divisor problems, one can work with an $n$-parameter family of PDEs and show that for most values of the parameters there exist invariant tori of dimension $n$. Of course one would prefer to have information on a fixed PDE. Following the strategy of finite dimensional dynamical systems, this is typically done by showing that, due to the presence of the non-linearity, the initial data modulate the frequency and hence can be used as parameters. 
It must be noted that quasi-periodic solutions by construction are not typical (w.r.t. any reasonable measure on the phase space), this is true already in finite dimension where lower dimensional tori  clearly have measure zero. 
\\
A natural step forward is to look for  
 {\it almost-periodic solutions}, i.e. solutions which are limit (in the uniform topology in time) of quasi-periodic solutions. 
A very na\"if approach would be to construct the desired solution by just constructing quasi-periodic solutions supported on invariant tori of dimension $n$ and then taking the limit $n\to \infty$. Unfortunately the classical KAM procedure (of say \cite{Poschel:1996},\cite{Kuksin-Poschel:1996}) is not uniform in the dimension $n$,  and by taking the limit one just falls on the elliptic fixed point. 

A refined version of this very natural idea was in fact used by P\"oschel in \cite{Poschel:2002}, to prove  the existence of almost-periodic solutions with very high regularity, i.e. with Fourier coefficients $u_j$ which decay in a super-exponential way as $j\to \infty$. 
His idea was to construct a sequence of invariant tori of growing dimension using at each step the invariant torus of the previous one as an unperturbed solution: the KAM procedure being not uniform in the dimension $n$, the $n+1$'th and $n$'th tori are extremely close, this leading to very regular solutions.

The model of P\"oschel \cite{Poschel:2002} is an NLS equation on the circle with a multiplicative potential and smoothing non-linearity. The potential is not fixed a priori, but gives an infinite set of {\it external parameters}, which are used to tune the frequencies and control the small divisors. Of course this means that the existence of almost-periodic solutions is proved for {\it most} choices of the potential.  
The fact such potential can be used to modulate the frequencies, comes from spectral results and from the fact that the non-linearity is smoothing.\\
  Recently Geng and Xu in \cite{GX13} proved existence of analytic almost periodic solutions for the NLS equation with external parameters
  given by a Fourier multipliers without assuming any smoothing
  assumptions on the nonlinearity. Their approach generalizes  the one of \cite{Poschel:2002} by applying the ideas of T\"oplitz-Lipschitz functions which give a better control on the asymptotics of the frequencies.

A  different approach was proposed by Bourgain in \cite{Bourgain:2005} to study  a translation invariant  NLS equation with a Fourier multiplier providing external parameters in $\ell^\infty$. The main result is to prove -for most values of the parameters-  existence and linear stability of almost-periodic solutions with {\it Gevrey} regularity. The idea is to construct  a converging sequence of
$\infty$-dimensional approximately invariant manifolds and prove that the limit is the support of the desired almost-periodic solution.
The fact that one does not restrict to  neighborhoods of  finite dimensional tori allows a better control of the small-divisors and hence the construction of more general {\it i.e. less regular}  solutions. 
\\
Bourgain points out that a weak point of the costruction through finite dimensional tori comes from using the action-angle coordinates, i.e. the symplectic coordinates adapted to the finite dimensional approximately invariant  tori. This is due to the fact that action-angle coordinates in the $\infty$-dimensional context are not, in general, well defined; then the idea is to work directly in the  Fourier basis and exploit the properties of functions analytic in a neighborhood of zero. 
\\
We mention also \cite{Yuan_et_al:2017}, where the authors discuss non-linear stability of the invariant manifolds studied in \cite{Bourgain:2005}.

In the present paper 
we take the 
 same point of view as Bourgain.
 On the other hand our analysis extends 
 Bourgain's result in many ways.
 \\
 We introduce a new functional setting which allows us 
to construct both  finite and infinite dimensional tori  
by a KAM algorithm  which is {\it uniform} in the dimension.
In particular our only restriction on the actions
is that they live in a ball around the origin, whereas
those of  \cite{Bourgain:2005} belong to a set of zero measure.
Moreover we discuss whether (and with respect to which topology) 
our invariant manifolds
 are embedding of an $\infty$-torus and also whether one can define action angle variables close to it.
Finally we are able to consider also non translation invariant
NLS equations.

Before describing  further our results, let us briefly discuss our model.

Following Bourgain in \cite{Bourgain:2005}
we consider families of NLS equations on the circle with external parameters of the form:
\begin{equation}\label{NLS}
\im u_t + u_{xx} - V\ast u +
f(x,|u|^2)u=0\,.
\end{equation}
Here $\im=\sqrt{-1}$,
$u=\sum_{j\in\Z}  u_j e^{\im j x},$
$V\ast$ is a Fourier multiplier 
\begin{equation}
\label{vu}
V\ast u = \sum_{j\in\Z} V_j u_j e^{\im j x}\,,\quad \pa{V_j}_{j\in\Z}\in 
\ell^\infty
\end{equation}
and $f(x,y)$ is $2\pi$ periodic and real analytic in $x$ and  is real analytic in $y$ in a neighborhood of $y=0$. We shall assume that $f(x,0)=0$. By  analyticity, for some $\mathtt a,R>0$ we have
\begin{equation}\label{analitico}
f(x,y)= 
\sum_{d=1}^\infty f^{(d)}(x) y^d\,,\quad
|f|_{\ta,R}:=\sum_{d=1}^\infty|f^{(d)}|_{\T_{\mathtt a}}R^d <\infty \,,
\end{equation}
where,  given a real analytic function  $g(x)=\displaystyle \sum_{j\in \Z}g_j e^{\im j x},$ we 
set
$ |g|_{\T_{\ta}}:=\displaystyle\sup_{j\in\Z}|g_j| e^{\mathtt a|j|}\,.$
It is well known that \eqref{NLS}
is a Hamiltonian system  with Hamiltonian \begin{eqnarray}\label{Hamilto0}
&&H_{\rm NLS}(u):=
\int_\T  |u_x|^2 + V\ast |u|^2 + F(x,|u(x)|^2) dx \,,\quad F(x,y):=\int_0^y f(x,s) ds\,,
\nonumber
\\
&&
\end{eqnarray}
w.r.t. the symplectic form $\Omega(u,v)= $ Im $\int_\T u \bar v$ induced by the Hermitian product on $L^2(\T,\C)$.
\\
Note that Equation \eqref{NLS} with $f(x,|u|^2)=|u|^4$ is the model considered in \cite{Bourgain:2005}. 
\\
Passing to the Fourier side, i.e. setting
\[
u(x) = \sum_{j\in \Z} u_j e^{\im j x} \,,\quad (u_j)_{j\in\Z} \in \ell^2(\C)\,,
\]
$H_{\rm NLS}$  in \eqref{Hamilto0} is an infinite dimensional Hamiltonian System
consisting in an infinite chain of harmonic oscillators, with linear frequencies $\lambda_j = j^2 + V_j$, coupled by a non-linearity.

If we ignore the non-linearity all the bounded solutions are of the form
\[
u_{\rm Lin}(x,t) = \sum_{j\in \Z} u_j(0)e^{\im (jx + \lambda_j t)} \,,\quad \lambda_j = j^2 + V_j
\]
hence --for most values of $V$-- they are periodic, quasi-periodic or almost-periodic\footnote{
	We recall that an almost-periodic function is the uniform limit of 
	quasi-periodic ones.} 
accordingly whether  one, finitely many or infinitely many modes 
$u_j(0)$
are excited.

It is natural to ask if these solutions persist when the nonlinearity 
is plugged in. In this direction,  for the translation
invariant NLS,  
 Bourgain
in \cite{Bourgain:2005} proves 
that\footnote{Actually Bourgain prove \eqref{giulio}
with $\theta=1/2$ the extension for $0<\theta<1$
was given later in \cite{Yuan_et_al:2017}.}:

\medskip

\noindent
{\it
	Given any $s>0,$ $0<\theta<1$, for "most choices" of $V=(V_j)_{j\in\Z}\in \ell^\infty$ there exist almost periodic solutions $u(x,t)$ such that:
	\begin{equation}\label{giulio}
	\frac{r}{2}{\jap{j}}^{-2}e^{-s\jap{j}^\theta} \le |u_j(t)| \le r {\jap{j}}^{-2}e^{-s\jap{j}^\theta}  \quad \forall j\in \Z\,,\qquad \jap{j}:=\max\{ |j|,1\}\,,
	\end{equation}
	for some small enough $r>0$ and for all times.
}

Informally speaking in this paper we remove the 
lower bound in \eqref{giulio} and generalize the result
to non translation invariant equation \eqref{NLS}.

In order to formulate our result in a more precise way, let us introduce some functional setting.
We start by passing to the Fourier side and identifying $u(x)$ with the sequence of Fourier coefficients $(u_j)_{j\in \Z}$. We work on $\ell^2(\C)$ with the standard real symplectic structure coming from the Hermitian product\footnote{We recall that given a  complex Hilbert space $H$ with a Hermitian product $(\cdot,\cdot)$, its realification is a real symplectic Hilbert space with scalar product  and symplectic form given by
	\[
	\langle u,v\rangle = 2{\rm Re}(u,v)\,,\quad  \omega(u,v)= 2{\rm Im}(u,v)\,.
	\] }.
Then  \eqref{Hamilto0} becomes
\begin{eqnarray}\label{hamNLS}
& H_{\rm NLS}(u):=
\sum_{j\in\Z} (j^2+V_j) |u_j|^2+P \,,\quad P:= \int_\T F(x,|\sum_j u_j e^{\im j x}|^2) dx.
\end{eqnarray}
Fixing once and for all $0 < \teta < 1$, for $s>0$, $a\ge 0$, $p>1$, we consider the following Banach spaces of sequences on $\C$
\begin{equation}\label{pecoreccio}
\tw^\infty_{p,s,a}
:= 
\set{v:= \pa{v_j}_{j\in\Z}\in\ell^2(\C)\; : \quad \abs{v}_{p,s,a}:= \sup_{j\in\Z}\abs{v_j} \jap{j}^{ p}e^{a\abs{j}+s \jap{j}^\teta}< \infty}\,,
\end{equation}
We endow $\tw^\infty_{p,s,a}\subset \ell^2$ with the symplectic structure inherited from $\ell^2$.
\\
Our aim is to prove the existence of a symplectic change of variables, well defined and analytic in some open ball 
${\bar B}_r(\tw^\infty_{p,s,a})\subset \tw^\infty_{p,s,a}$ centered at the origin and with radius $r>0$, which conjugates $H_{\rm NLS}$ to a  Hamiltonian which has an invariant torus which supports an almost-periodic solution of Diophantine frequency. 
Following \cite{Bourgain:2005}, we fix the hypercube 
\begin{equation}\label{Omega}
\pan:=
\set{\omega=\pa{\omega_j}_{j\in \Z}\in \R^\Z,\quad \sup_j|\omega_j-j^2| \leq 1/2 }.
\end{equation}
\begin{defn}\label{fantino} Given $0<\gamma<1$, we  denote by $\dgp$ the set  of \sl{$\gamma$-Diophantine} frequencies 
	\begin{equation}\label{diofantinozero}
	\dgp:=\set{\omega\in \pan\,:\;	|\omega\cdot \ell|> \gamma \prod_{n\in \Z}\frac{1}{(1+|\ell_n|^2 \jap{n}^{2})}\,,\quad \forall \ell\in \Z^\Z: \ell\ne 0\,,\; |\ell|:=\sum_i|\ell_i|<\infty}.
	\end{equation} 
\end{defn}

Given a sequence
\begin{equation}\label{polenta}
I = \pa{I_j }_{j\in\Z}\,,\qquad 
I_j\geq 0\,,
\qquad 
\sqrt{I}:= \pa{\sqrt{I_j} }_{j\in\Z}\in \tw^\infty_{p,s,a}\,,
\end{equation}
we consider the torus
\begin{equation}\label{bove}
\cT_I:=\{ u \in \tw_{p,s,a}^\infty: |u_j|^2= I_j\,,\quad \forall j\in \Z \}.
\end{equation}
{	We say that $\cT_I$ is a KAM torus of frequency $\omega$ for the  Hamiltonian $N$ if
	$$
	N=\sum_{j\in \Z}\omega_j |u_j|^2+P\,,
	$$
	where the  Hamiltonian vector field $X_P$ vanishes
	on the torus $\cT_I.$
Indeed under the hypotheses above  $\cT_I$  is an invariant torus  for the dynamics of $N$. Moreover the restricted dynamics  is linear with frequency $\o$, namely
\begin{equation}
\label{governoladro}
u_j(t) = u_j(0)e^{\im \omega_j t}\,, \qquad
|u_j(0)|^2=I_j\,,\ \ j\in\Z\,.
\end{equation}
We are now ready to state our result.
\begin{them}\label{torello}
	For any $p>1$, $s>0$, $0\leq a<\ta$,  $\g>0$ there exists $\e_*=\e_*(p,s,\ta-a)>0$ such that for all $r>0$
	satisfying 
	\begin{equation}\label{cornettone}
\frac{| f|_{\ta, R}}{\g R} r^2\leq \e_*
\end{equation}
	the following holds.
	For all $\omega\in \dgp$ and  
	$\sqrt I\in {\bar B}_{r}(\tw^\infty_{p,s,a})$ as in \eqref{polenta}
	there exists a  potential 
	$V\in\ell^\infty$
	and a symplectic analytic change of variables 
	$\Phi: {\bar B}_{2r}(\tw^\infty_{p,s,a}) \to {\bar B}_{4r}(\tw^\infty_{p,s,a})$ such that $\cT_I$ is a KAM torus of frequency $\omega$ for
	$
	H_{\rm NLS} \circ \Phi$.
	Finally $V$  depends on $\omega$
	in a Lipschitz way. 
\end{them}}

{Of course  the change of variables $\Phi$ is invertible in the sense that
	there exists
		$\Psi: {\bar B}_{2r}(\tw^\infty_{p,s,a})\mapsto {\bar B}_{4r}(\tw^\infty_{p,s,a})$  such that 
	$	\Psi\circ\Phi u=\Phi\circ\Psi u= u\,,$ for all $u$  in some smaller ball. We denote as is habitual $\Phi^{-1}:=\Psi$.
If $I_j>0$ for infinitely many $j\in\Z$ then  $\Phi(\cT_I)$
		supports truly  almost-periodic solutions.\footnote{{ Indeed, the map $t \mapsto u(t)$ defined in \eqref{governoladro} is almost periodic from $\R$ to the phase space $\tw^\infty_{p',s,a}$ with $p'< p$. }}
	\vskip10pt
	
		If $\rho:=\inf_{j\in\Z}I_j  \jap{j}^{2p}e^{2a\abs{j}+2s \jap{j}^\teta}>0$ 
		then we are in the framework \eqref{giulio}.
		Moreover, as we show in Appendix \ref{aa}, in this setting we can introduce action-angle variables around the torus
		$\cT_I$. The action-angle map is a diffeomorphism from a neighborhood 
		of $\cT_I$ into $\{|J-I|_{2p,2s,2a}<\rho/2\}\times \T^\Z,$
		where $\T:=\R/2\pi\Z,$ and $\T^\Z$ is a differential manifold modeled on 
		$\ell^\infty$. { In this setting $\Phi(\cT_I)$ is an embedded invariant torus.}
	
		\vskip10pt
		
		  In order to compare our result with the one of \cite{Bourgain:2005} we remark that the condition \eqref{giulio} places strong restriction on the action space. Indeed   the set of $u\in {\bar B}_r(\tw_{p,s,a}^\infty)$ such that $|u_j|^2$ satisfies \eqref{giulio} has zero measure 
		with respect to the natural product measure
		in a weighted-$\ell^\infty$ space. 
		We remark (see Appendix \ref{aa}) that the set
		\begin{equation}
				\label{perlunointro}
			\set{ u\in\tw^{\infty}_{p,s,a}\, : \, \inf_j \jap{j}^{ p}e^{a\abs{j}+s\jap{j}^\theta}| u_j|>0 }\,
		\end{equation}
		is topologically generic in the sense of Baire, but still has measure zero.	Our result instead holds for all actions $\sqrt I\in {\bar B}_{r}(\tw^\infty_{p,s,a})$. 
		
		Another  interesting advantage of this result is that it is \co{independent} from the dimension
		of the invariant torus (i.e. how many actions are non-zero), and one can prove within \co{the same} unified scheme both the persistence of quasi-periodic (if only a finite number of $I_j \neq 0$) and almost-periodic solutions which is independent of the dimension. Of course if we follow Theorem \ref{torello} directly, it seems that one needs to modulate infinitely many parameters and require an infinite dimensional diophantine condition even to construct finite dimensional tori. In section \ref{elliptic}, we show that in fact this is not true: as expected we only need to modulate as many parameters as the non-zero actions. Moreover in this case the small divisor conditions \eqref{diofantinozero} reduce to the classical Melnikov ones

Let us fix a non empty set $\cS\subseteq\Z$ and, as in \eqref{Omega} we set
\begin{equation}\label{OmegaS}
\pan_\cS:=
\set{\omega=\pa{\omega_j}_{j\in \cS}\in \R^\cS,\quad \sup_j|\omega_j-j^2| \leq 1/2 }.
\end{equation}
 In order to keep the proof as simple as possible we shall avoid technical issues related to double eigenvalues by assuming that the
 NLS Hamiltonian \eqref{hamNLS}
 is translation invariant,  namely that $f(x,|u|^2)$ 
  does not depend  on $x$.
We refer the reader to \cite{EK10}, \cite{Feo15}, \cite{MP19} for a discussion of double eigenvalues.

\begin{them}\label{torelloellitticointro}
Consider a translation invariant NLS Hamiltonian 
as in \eqref{hamNLS}.
For any $p>1$, $s>0$, $0\leq a<\ta$, $\g>0$  there exists $\e_*=\e_*(p,s,\ta-a)>0$ such that 
for all $r>0$
	satisfying 
	\eqref{cornettone},
for every 
	$\sqrt I\in {\bar B}_{r}(\tw^\infty_{p,s,a})$  with 
  $I_j=0$ for $j\in\cS^c$ and 
for any    $|W_j|\leq 1/4$ with $j\in\cS^c$,  such  that if $0\in\cS^c$ then  $W_0\ne 0$,
  the following holds.
\\
There exist
  a positive measure Cantor-like   set
  $\cC\subset \pan_\cS$,  such that for all $\omega\in \cC$
  there exists a potential  $V\in \ell^\infty$ 
  and a change of variables  $\Phi:{\bar B}_{2r}(\tw^\infty_{p,s,a})\to {\bar B}_{4r}(\tw^\infty_{p,s,a})$  such that $\cT_I$ is an elliptic  KAM torus of frequency $\omega$ for
	$
	H_{\rm NLS} \circ \Phi$ and $V_j
	=W_j$ for $j\in\cS^c$. 
	Finally $V$  depends on $\omega$
	in a Lipschitz way. 
\end{them}

The general strategy of our paper is to rephrase  Theorems \ref{torello}, \ref{torelloellitticointro} as a counterterm problem and look for the change of variables $\Phi$ by performing an iterative KAM scheme. 

\subsection{ A normal form theorem \`a la Herman}
In finite dimension, the first generalization of KAM theory to degenerate (and not necessarily conservative) systems is due to J. Moser in 1967 who established a normal form proving, in the framework of analytic vector fields, that the persistence of a reducible, Diophantine invariant quasi-periodic torus is a phenomenon of finite co-dimension. This is due to the {{introduction}} in the perturbed system of some {\it extra parameters} (or {\it counter-terms}), in order to compensate its eventual \textit{degeneracies}, such as absence of twist properties or Hamiltonian symmetries that usually make the general KAM-scheme work.\\ 
In the course of the 80-90's, Herman and R\"ussmann at first exploited this idea to derive KAM-type results through a technique now known as "elimination of parameters", the power of which is enlightened in many further works, see at instance \cite{Chenciner:1985, Broer-Huitema-Takens:1990, Sevryuk:1999, Fejoz:2004,  EFK:2013, Massetti:APDE, CGP} and references therein. \\
It is then natural to extend this approach also to the infinite dimensional case
and derive the existence of an almost periodic torus in two separated steps:
\begin{itemize}[leftmargin=*]
	\item prove a normal form which does not rely on any non-degeneracy condition (but containing the hard analysis)
	\item show that the counter-terms can be \textit{eliminated} by using internal or external parameters and convenient non-degeneracy assumptions (twist condition, symmetries...) satisfied by the system under analysis, through the application of the usual implicit function theorem: if the extra corrections vanish, then the perturbed systems under normal form possesses an invariant almost-periodic torus.
\end{itemize}\smallskip
Theorems \ref{torello}, \ref{torelloellitticointro} will then be proved through a direct application to the NLS Hamiltonian of  normal form theorems in the spirit of "Herman's twist theorem" for degenerate Hamiltonians \cite{Herman:1971}, and the elimination of the counter terms through the $(V_j)$.\\
While in finite dimension, this kind of results can be stated in a very sinthetic and somehow more conceptual way (see for instance, \cite{ Fejoz:2004, Fayad-Krikorian:2009, Massetti:ETDS} and references therein), in infinite dimensions one must be more precise regarding   quantitative aspects of the functional spaces involved.

\smallskip
In order to motivate our strategy let us briefly describe 
{ the  classical approach in the finite dimensional} case to prove the persistence of an invariant torus $\cT_I$.
 As is to be expected, the simplest scenario is the { Lagrangian} case, i.e. $I_j>0$ for all $j=1,\dots,d$ (where $d$ is the dimension of the configuration space).
In this context, the simplest approach is to write $H$ in action-angle variables centered at $I$, i.e. set
\[
(y,\theta)\to  u(I,y,\theta)\,,\quad u_j= \sqrt{I_j+y_j}e^{\im \theta_j}\,,\quad |y_j|<I_j\,,
\]
so that  $H(y,\theta)$ is analytic in $y$ and we can Taylor expand it at $y=0$.
Then  the KAM scheme corresponds to looking for a change of variables - analytic for $|y|$ small enough - which cancels  the  homogeneous terms up to the order one in $y$ in the Hamiltonian (namely those terms, depending on $\teta$, which prevent $\cT_I$ to be invariant and quasi-periodic). 
\\
Then it is very natural to make the ansatz that the change of variables above should be affine in $y$, see \cite{Poschel:2001} for example.
Since the action angle coordinates introduce a singluarity at $y=-I$, the KAM schemes in action-angle variables produce a change of variables defined and analytic in some neighborhood of $\cT_I$ (essentially an annulus).
In the case of {lower dimensional tori}, one adapts the scheme above see for instance \cite{Poschel:1989}. Note that in this context also the dynamics normal to the torus comes into play. To bound the resulting small divisors,  one requires the so called "Melnikov conditions" which control  interactions between tangential and normal modes.

\smallskip

 It is worthwhile to mention that, in the case of a dynamical system with an elliptic fixed point at zero, one can also work directly in the natural elliptic variables $(p,q)$ (or in complex notation $u= p+\im q$), see \cite{CGP}, \cite{Llave} and conjugate  a  Hamiltonian with external parameters to normal form by a change of variables which is analytic in a neighborhood of zero, just like in Theorem \ref{torello}. 
 \\
 We stress that in this scheme, by its very nature, there is no need to specify whether $\cT_I$ is maximal or not.

\medskip
Trying to reproduce the approach described above in the { infinite dimensional setting} {is not} at all sreightforward, except in the case of finite dimensional elliptic tori.  \\
First of all, introducing { infinite dimensional action-angles variables} is much more delicate: whether the map $(\teta, y) \mapsto u$ defines a diffeomorphisms in the neighborhood of $\T^\infty\times \R^\infty$ is strongly related to the chosen topology endowing the spaces involved.  In Appendix \ref{aa} we show that, if we work in a weighted $\ell^\infty$ space, then they can be defined at least if $I$ satisfies some appropriate conditions. Nonetheless, even in this simplest setting \emph{we are not able} to perform the KAM scheme in action-angle variables. The main reason relying on the fact that we are not able to prove the analyticity of the solution of the {\it homological equation} $L_\omega F = G$, if $G$ is defined only in an annulus around $\cT_I$. 
\smallskip\\
It is then preferable to work directly in the { elliptic variables $u$}, where we can control the solution of the homological equation, provided that $G$ is analytic in a ball at $u=0$, see Lemma \ref{Lieder}. \\
The whole problem then amounts to understanding which terms need to be canceled in order to have an invariant almost-periodic torus. The key difference with respect to the action-angle approach is that these terms must be analytic in a ball centred at the origin. In order to guarantee these two conditions at the same time, we perform a decomposition of the Hamiltonian following the idea of Bourgain in \cite{Bourgain:2005}, but constructing a  more natural functional setting, that we are going to describe.

We shall consider Hamiltonians of the form
\[
\sum_{j\in \Z}\omega_j|u_j|^2  + P \,, \qquad \omega=\pa{\omega_j}_{j\in\Z}\in \dgp
\]
where $P$ is a {\bf regular Hamiltonian}, namely 
the Cauchy majorant of $P$ is an analytic function  from some ball ${\bar B}_r(\tw^\infty_{p,s,a}) $ to $ \R$, whose  Hamiltonian vector field $X_P$ is again a bounded analytic function from ${\bar B}_r(\tw^\infty_{p,s,a}) $ to $ \tw^\infty_{p,s,a}$, (see definition \ref{gianfranco}).
 We  denote the space above by $\cH_{r,s,\eta}$ and we endow it with the norm $|\cdot|_{r,s,\eta}$  (here the parameter $\eta$ is technical, used to control terms in the Hamiltonian which do not preserve momentum). In order to control the dependence of a Hamiltonian on $\omega\in\dgp$ we  introduce  weighted Lipschitz norms which we denote by $\|\cdot\|_{r,s,\eta}$ (see \eqref{normag} ).

By the very definition of a KAM torus, we wish to decompose regular Hamiltonians as a sum of regular terms with an increasing "order of zero" at $\cT_I$, $\sqrt I\in {\bar B}_{r'}(\tw^\infty_{p,s,a})$ with $0<r'<r$.
To this purpose we consider the 
{\it direct sum decomposition}\footnote{Actually 
	\eqref{sommadiretta}
	holds at any even order $d\geq-2$, namely
	$\cH_{r,s, \eta}=
	\cH_{r,s, \eta}^{(-2)} \oplus \cdots
	\oplus \cH_{r,s, \eta}^{(d)}\oplus \cH_{r,s, \eta}^{(\ge d+2)}$,
	see Proposition \ref{proiettotutto}.}
\begin{equation}
\label{sommadiretta}
\cH_{r,s, \eta} =\cH_{r,s, \eta}^{(-2)} \oplus \cH_{r,s, \eta}^{(0)}\oplus \cH_{r,s, \eta}^{(\ge 2)}\,,\qquad
H=
H^{(-2)}+H^{(0)}+H^{(\geq 2)}
\,,
\end{equation}
so that  $H\in \cH_{r,s, \eta}^{(0)}$ vanishes at $\cT_I$ (however $X_H$ is tangent but not necessarily null), while $\cH_{r,s, \eta}^{(\ge 2)}$ has a zero of order at least $ 2$ at $\cT_I$ (this means that the corresponding vector field vanishes at $\cT_I$).
A crucial property for the convergence of the KAM scheme
is the behaviour of such decomposition
with respect to Poisson brackets, namely
$\{F,G^{(\geq 2)}\}^{(-2)}=0$
and, if $F^{(-2)}=0$, also  
$\{F,G^{(\geq 2)}\}^{(0)}=0$
(see Lemma \ref{gradi}).
\\
Let us roughly discuss the decomposition 	
\eqref{sommadiretta}.
The main idea is  to make a power series expansion centered at $I$ without introducing a singularity.
Start from a regular Hamiltonian $H(u)$ expanded
in Taylor series at $u=0$ and rewrite
every monomial $u^\bal\bar u^\bbt$  as $|u|^{2m} u^\al \bar u^\bt$ with $\al,\bt$ with distinct support.
Then define 
an {\sl auxiliary Hamiltonian} $\tH(u,w)$ 
(here $w=(w_j)_{ j\in \Z}$ are auxiliary ``action'' variable)
by  the substitution 
$|u|^{2m} u^\al \bar u^\bt \rightsquigarrow w^m u^\al \bar u^\bt$ 
(see \eqref{coppiette} below).
Since we are considering functions on a product space, it turns out that $\tH(u,w)$ is  analytic in both $u$ and $w$. In particular
we can Taylor expand with respect to $w$ at the point
$w=I$, being $I$  in the domain of analyticity.
Then we set
$H^{(-2)}(u) := \tH(u,I) $, $ H^{(0)}(u) := D_w \tH(u,I)[|u|^2-I] $ and $H^{(\geq 2)}(u)$ is what is left.
As an example the Hamiltonian 
\[
H= |u_1|^2 |u_2|^4 \Re (u_1 \bar u_3)
=
(|u_1|^2- I_1+I_1)(|u_2|^2 - I_2+I_2)^2 \Re (u_1 \bar u_3)\,,
\] 
has auxiliary Hamiltonian
$\tH(u,w)=w_1 w_2^2 \Re (u_1 \bar u_3)$
and decomposes as 
\begin{align*}
& H^{(-2)}:= I_1I_2^2 \Re (u_1 \bar u_3)\,,
\qquad
H^{(0)}:=\sq{I_2^2(|u_1|^2-I_1)   + 2 I_1I_2 (|u_2|^2-I_2)}  \Re (u_1 \bar u_3) \\
&H^{(\geq 2)}:=\pa{|u_1|^2(|u_2|^2-I_2) + 2 I_2 (|u_1|^2-I_1) } (|u_2|^2 - I_2)\Re (u_1 \bar u_3)\,.
\end{align*}
The above decomposition is, {\it at a formal level},
the same introduced by Bourgain in \cite{Bourgain:2005};
the main novelty here is that we introduce a suitable
Banach space, namely $\cH_{r,s, \eta}$, 
such that {\it the decomposition holds in the same space
with no loss of regularity},
see \eqref{sommadiretta}, while 
the analogous is not true for the norm used in \cite{Bourgain:2005},
where  a loss of regularity in the Gevrey 
parameter $s$ is necessary.
An important point is that all our construction works independently of the "dimension" of $\cT_I$, namely it never requires conditions of the form $I_j \neq 0$.

In view of the above decomposition we give the following

\begin{defn}[Normal forms]\label{NCO}
 We say that a Hamiltonian $N\in D_\omega + \Heta$ is in normal form at $\cT_I$ if $N-D_\omega\in \Heta^{(\ge 2)}$ and denote the (affine) subspace of normal forms by $\cN_{r, s ,\eta}(\omega,I)\equiv \cN_{r, s ,\eta}$.
\end{defn}
Of course if  $N$ is in normal form at $\cT_I$ then the torus is invariant and the $N$-flow
is linear on $\cT_I$,  with frequency $\omega.$
In  typical KAM schemes  one can assume that $N-D_\omega$ is negligible, since it has  a zero of order at least two at $\cT_I$, by restricting to  a sufficiently small neighborhood of $\cT_I$. Of course in our setting this is not true and  $N-D_\omega$   is not necessarily small in a ball centered at zero (and containing $\cT_I$).

Now we fix  parameters 
\begin{equation}\label{newton}
\mbox{$r_0,s_0,\eta_0>0$\  \ \ and take}\ \ \ 
0<\rho <\frac{r_0}{2}\,,\quad 0<r \le \frac{r_0}{2\sqrt{2}}\,,
\quad 0<\s<\min\{\frac{\eta_0}{2},1\}\,.
\end{equation}
 We then have the following normal form theorem.

\begin{them}[Twisted conjugacy à la Herman]
	\label{allaMoserbis} 
	Consider
	$r_0,s_0,\eta_0, \rho, r,\s$ as in \eqref{newton}.
There exists $\bar{\epsilon},\bar C>0$, depending only on
$
\rho/r_0$
and
$\s$ 
such that the following holds.
 Let $\sqrt{I}\in {\bar B}_{r}(\tw_{p,s_0+\s,a}^\infty)$,  $N_0\in\cN_{r_0, s_0 ,\eta_0}(I,\omega)$ and    $H\in D_\omega +\cH_{r_0,s_0,\eta_0}$.   If
\begin{equation}\label{maipendanti}
	(1+\Theta)^3 \epsilon \le \bar{\epsilon} \,,\qquad
	\mbox{where}\qquad
	\epsilon:=\g^{-1}\norma{H- N_0}_{r_0 ,s_0 ,\eta_0 }\,,\quad \Theta = \g^{-1}\norma{D_\omega- N_0}_{r_0 ,s_0 ,\eta_0 }
\end{equation}
then there exist a symplectic diffeomorphism $\Psi: {\bar B}_{r_0-\rho}\pa{\tw_{p,s_0+\s,a}^\infty} \to {\bar B}_{r_0}\pa{\tw_{p,s_0+\s,a}^\infty} $, close to the identity, a unique correction (counter term) 
	$\Lambda = \sum_{j}\lambda_j \pa{\modi{j} - I_j}$,
	Lipschitz depending on $\o\in\dg$, with:
	\[
	\norma{\lambda}_\infty \le \bar C \g   (1+\Theta)\epsilon
	\]
	and a Hamiltonian $N\in\cN_{r_0-\rho,s_0+\sigma,\eta_0-\sigma}(I,\omega)$,  such that 
	\begin{equation}\label{coniugio}
	\pa{\Lambda + H}\circ \Psi = N.  
	\end{equation}
	\end{them}
  
\begin{rmk}
 The quantities $\bar\epsilon$ and $\bar C$
 can be explicitly evaluated; see \eqref{zampadecane} below.
\end{rmk}

The name "twisted conjugacy",  comes from the fact that there exists a conjugacy between $H$ and $N$ which is "twisted" by the correction in frequencies $\Lambda$  (geometrically, the Hamiltonian action on $H$ is "twisted" by the presence of the counter-term: $H = N\circ \Psi^{-1} - \Lambda$, cf. equation \eqref{coniugio}) (see \cite{Fejoz:oberwolfach}).

The proof of Theorem
\ref{allaMoserbis} 
is based on an iterative scheme that 
  kills out the obstructing terms, namely  terms belonging to 
  $ \cH_{r,s, \eta}^{(-2)}$ and $ \cH_{r,s, \eta}^{(0)}$,
  by solving homological equations of the form 
\[
L_\omega F^{(d)} = G^{(d)},\qquad G^{(d)}\in \cH_{r,s, \eta}^{(d)},\quad d= -2, 0.
\]
The control of the solution $L_\omega^{-1} G^{(d)}$ relies strongly on the fact that $G^{(d)}$ is analytic in a ball around $0$, see Lemma \ref{Lieder}.
\subsection*{Stability of tori}
 In \cite{Yuan_et_al:2017} the authors prove, with respect to Bourgain's work, the long time stability of the almost periodic tori. As can be expected, this result can be recovered also in the slightly more general setting of Theorem \ref{allaMoserbis}, namely without requiring any smalleness condition on the non linear part of the normal form $N_0$.
In section \ref{seziostabile} we consider the Hamiltonian flow of a normal form $N$ and we show that initial data $\delta$-close to the invariant torus, stay $2\delta$-close to the torus for polynomially long times.


\section{Functional setting}

Consider the scale of Banach spaces
$\tw^\infty_{p,s,a}$
defined in \eqref{pecoreccio}.
For any $ p \le p', a \le a', s \le s' $ we have 
\begin{equation}
\mathtt{w}_{p',s',a'}^\infty \subset \mathtt{w}_{p,s,a}^{\infty} 
\end{equation}
and 
\begin{equation*}
\abs{v}_{p,s,a} \le \abs{v}_{p',s',a'},\quad \forall v\in \mathtt{w}_{p',s',a'}^\infty.
\end{equation*}
Note that
\begin{equation}\label{caciocavallo}
u,v\in\ell^\infty(\C)\,,\ \ |u_j|\leq |v_j|\ \ \ \forall j\in\Z\qquad
\Longrightarrow\qquad
|u|_{p,s,a}\leq |v|_{p,s,a}\,.
\end{equation}

In general, given $r>0$ and a Banach space $E$
we denote by ${\bar B}_r(E)\subset E$ the {\sl closed}  ball 
of radius $r$ centered at the origin.

We endow $\tw^\infty_{p,s,a}$ with the standard real symplectic structure coming from the Hermitian product on $\ell^2(\C)$
For convenience and to keep track of the complex structure, we write everything  in complex notation, that is
\[
 \im \sum_j d u_j\wedge \d \bar u_j \,,\quad X_H^{(j)}  = \im \frac{\partial}{\partial \bar u_j} H\,
\]
%
\subsection{Spaces of Hamiltonians}

\begin{rmk}
  Here and in the following we shall always assume that $p>1$ and $s>0$.
\end{rmk}
\smallskip

\noindent
\gr{Multi-index notation}.
  In the following we denote, with abuse of notation, by $\N^\Z$ the set of 
   multi-indexes $\bal,\bbt$ etc. such that
  $|\bal|:=\sum_{j\in\Z}\bal_j$ is finite.
  As usual $\bal!:=\prod_{j\in\Z,\, \bal_j\neq 0}\bal_j$.
  Moreover $\bal\preceq\bbt$ means $\bal_j\leq\bbt_j$
  for every $j\in\Z$, then $\binom{\bbt}{\bal}:=\frac{\bbt!}{\bal!(\bbt-\bal)!}.$
   Finally take $j_1<j_2<\ldots<j_n$
  such that $\bal_j\neq 0$ if and only if $j=j_i $ for some $1\leq i\leq n$,
  as usual we set $\partial^{\bal} f:=\partial^{\bal_{j_1}}_{u_{j_1}}
  \ldots \partial^{\bal_{j_n}}_{u_{j_n}} f
 \,;$
  analogously for $\partial_{\bar u}^\bbt f.$
  
\smallskip
\begin{defn}[majorant analytic Hamiltonians]\label{Hr}
	For $ r>0$,  let
	$\mathcal{A}_r(\wsap)$
	be the space of  
	Hamiltonians 
	$$
	H : {\bar B}_r(\wsap) \to \R
	$$ 
	such that there exists a pointwise  absolutely convergent power series expansion\footnote{As usual given a vector $k\in \Z^\Z$, 
		$|k|:=\sum_{j\in\Z}|k_j|$.}
	\begin{equation}\label{mergellina}
H(u)  = \sum_{\bal,\bbt\in\N^\Z }H_{\bal,\bbt}u^\bal \bar u^\bbt\,,
	\qquad
	u^\bal:=\prod_{j\in\Z}u_j^{\bal_j}
\end{equation}
	with the following properties: 
	\begin{enumerate}[(i)]
		\item Reality condition:
		\begin{equation}\label{real}
		H_{\bal,\bbt}= \overline{ H}_{\bbt,\bal}\,;
		\end{equation}
		\item Mass conservation:
		\begin{equation}\label{cecio}
		H_{\bal,\bbt}= 0 \quad\mbox{if}\;\, |\bal|\neq |\bbt| \,,
		\end{equation}namely the Hamiltonian Poisson commutes with the { \sl mass} $\sum_{j\in \Z}|u_j|^2$.
		\end{enumerate}
	Finally, given $H$ as above, we define its majorant
	$\und H:  {\bar B}_r(\wsap) \to \R$  as
	\begin{equation}\label{betta}
	\und H(u)  := \sum_{\bal,\bbt\in\N^\Z }|H_{\bal,\bbt}|u^\bal \bar u^\bbt
	\end{equation}
	and, for $\eta\ge 0$ its $\eta$-majorant:
	 \begin{equation}\label{etamag}
	\und { H}_\eta (u)
	:=
	 \sum_{\bal,\bbt\in\N^\Z} \abs{{H}_{\bal,\bbt}}e^{\eta|\pi(\bal-\bbt)|}\buu
	 \in
	 \mathcal{A}_r(\wsap)\,,
	\end{equation} 
	where 
	given $\bal\in \N^\Z$
\begin{equation}\label{momento}
\pi(\bal):= \sum_{j\in \Z} j \bal_j\,.
\end{equation}
\end{defn}

Note that $\pi(\bal-\bbt)$ is the eigenvalue of the adjoint action of the momentum Hamilonian $P= \sum_{j\in \Z}j|u_j|^2$ on the monomial $\buu$. The exponential weight $e^{\eta|\pi(\bal-\bbt)|}$ is added in order to ensure that the monomials which do not preserve momentum have an exponentially small coefficient.

\medskip

\begin{defn}[$\eta$-regular Hamiltonians]\label{gianfranco}
	We  say that a Hamiltonian 
	$H\in \cA_r(\tw_{p,s,a}^\infty)$
	 is $\eta$-regular\footnote{When $\eta=0$ we simply say that $H$ is regular.} 
	 if 
	  ${\underline H}_\eta(u)\in \cA_r(\tw_{p,s,a}^\infty)$ and  its Hamiltonian vector field is bounded, i.e. 
	  $$
	  \abs{H}_{r,s,\eta}^{(a,p)}=
	  \frac1r \pa{\sup_{\norm{u}_{p,s,a}\leq r} \norm{{X}_{{\underline H}_\eta}}_{p,s,a} } <\infty\,.
	  $$
	  We denote such space by
	  $
	  \Heta =\Heta^{(a,p)}.
	  $
	\end{defn}

	Note that 
$\abs{\cdot}_{r,s,\eta}$ is a seminorm
on  $\Heta$ and a norm on its subspace
\begin{equation}\label{lenticchia}
 \Heta^0:=
 \{\
		H\in\Heta\ \ {\rm with}\ \ 
H(0)=0\ \}\,,
\end{equation}
endowing $\Heta^0$ with a Banach space structure.
Analogously we set
\begin{equation}
\label{ao}	
	\mathcal{A}_r^0(\wsap):=\{\
		H\in\mathcal{A}_r(\wsap)\ \ {\rm with}\ \ 
H(0)=0\ \}\,
\end{equation}

\begin{lemma}\label{elisabetta}
For $H\in \cA_r(\tw_{p,s,a}^\infty)$, we have that 
 	\begin{equation}\label{norma1}
 \abs{H}_{r,s,\eta}^{(a,p)} 
 =
 \sup_j 
 \sum_{\bal,\bbt\in\N^\Z} \abs{H_{\bal,\bbt}}\bbt_j
 u_0^{\bal + \bbt - 2e_j}e^{\eta|\pi(\bal-\bbt)|}
 \end{equation}
 where $u_0=u_0(r)$ is defined as
 \begin{equation}\label{giancarlo}
 u_{0,j}(r):= r  \jap{j}^{-p} e^{- a\abs{j}- s\jap{j}^\theta} \,.
 \end{equation}
 Note that
 \begin{equation}\label{burrata}
 |u|_{p,s,a}\leq r \qquad
 \Longrightarrow
 \qquad
 |u_j|\leq u_{0,j}\,,\ \ \forall j\in\Z\,.
 \end{equation}
 Then   a formal power series as in \eqref{etamag} such that \eqref{norma1} is finite, 
 totally converges in the closed ball 
 ${\bar B}_r(\tw_{p,s,a}^\infty)$ with estimate
 \begin{equation}\label{abacab}
 \sum_{\bal,\bbt} \sup_{|u|_s\leq r}
 \abs{{H}_{\bal,\bbt}}e^{\eta|\pi(\bal-\bbt)|}\abs{\buu}
 \leq |H_{00}|+
 \sum_{|\bal|=|\bbt|>0} \abs{{H}_{\bal,\bbt}}e^{\eta|\pi(\bal-\bbt)|} u_0^{\bal+\bbt}
 \leq
 |H_{00}|+ r^2 \abs{H}_{r,s,\eta}^{(a,p)}\,.
 \end{equation}
 Therefore $H(u)$ and ${\underline H}_\eta(u)$
 are analytic in the open ball\footnote{In particular by \eqref{abacab}
 $X_{{\underline H}_\eta}$ exists on the open ball
 $B_r(\tw_{p,s,a}^\infty)$ and can be continuously extended
 on the closed ball $\bar{B}_r(\tw_{p,s,a}^\infty)$, so that the 
 expression ${\underline H}_\eta(u_0(r))$ makes sense.
 } 
 $B_r(\tw_{p,s,a}^\infty)$ and
$$
\abs{H}_{r,s,\eta}^{(a,p)}=
\frac1r \pa{\sup_{\norm{u}_{p,s,a}\leq r} \norm{{X}_{{\underline H}_\eta}}_{p,s,a} } 
 =
\frac1r
\norm{{X}_{{\underline H}_\eta}(u_0(r))}_{p,s,a}\,,
$$
where ${X}_{{\underline H}_\eta}$ is the vector field associated to the $\eta$-majorant Hamiltonian defined in \eqref{etamag}.
\end{lemma}

\noindent
The proof is postponed to  Appendix \ref{appendice tecnica}.

\medskip
\begin{rmk}
 The norm defined in \eqref{norma1} in $\Heta^0$
 is weaker than the one introduced by Bourgain in \cite{Bourgain:2005} (see formula (1.14)),
 which is a weighted $L^\infty$-norm on the coefficients of  the Taylor expansion of $H$ at the origin.
 Note that both norms are {\sl not intrinsic} since they depend on such  Taylor expansion.
\end{rmk}

\medskip

\noindent
By the reality 
condition\footnote{Indeed 
	\[
	\sum \abs{H_{\bal,\bbt}}\bbt_j {u_0}^{\bal + \bbt - e_j} e^{\eta \abs{\pi(\bal-\bbt)}}= \sum \abs{H_{\bal,\bbt}}\bal_j {u_0}^{\bal + \bbt - e_j} e^{\eta \abs{\pi(\bal-\bbt)}}\,.
	\]} 
	\ref{real} we get
\begin{eqnarray}
\abs{H}_{r,s,\eta}^{(a,p)}
&=&
 \sup_j  
\sum_{\bal,\bbt\in\N^\Z} \abs{H_{\bal,\bbt}}\bal_j
u_0^{\bal + \bbt - 2e_j}e^{\eta|\pi(\bal-\bbt)|}
\nonumber
\\
&=&
\frac12  \sup_j  
\sum_{\bal,\bbt\in\N^\Z} \abs{H_{\bal,\bbt}}\pa{ \bal_j + \bbt_j}u_0^{\bal + \bbt - 2e_j}e^{\eta|\pi(\bal-\bbt)|}
\nonumber
\\
&=&
\frac12  \sup_j  
\sum_{\bal,\bbt\in\N^\Z} \abs{H_{\bal,\bbt}}\pa{ \bal_j + \bbt_j}
c^{(j)}_{r,s, \eta}(\bal,\bbt)\,,
\label{normatris}
\end{eqnarray}
where
\begin{equation}
c^{(j)}_{r,s, \eta}(\bal,\bbt)
	:=
	u_0^{\bal + \bbt - 2e_j}e^{\eta|\pi(\bal-\bbt)|}
\,.
	\label{persico}
	\end{equation}

\noindent
\gr{Notations}:
Since we will always keep the parameters $a,p$  fixed, we shall drop them from our notations. Hence we will set 
\begin{equation}\label{matera}
\tw_s := \tw^\infty_{p,s,a},\qquad \abs{\cdot}_s:=\abs{\cdot}_{p,s,a} ,\qquad
 \abs{\cdot}_{r,s,\eta}:=\abs{\cdot}_{r,s,\eta}^{(a,p)} .
\end{equation}

\begin{rmk}
We note that if $H$ preserves momentum, i.e.
\[
H_{\bal,\bbt}=0\quad \mbox{if}\quad \pi(\bal-\bbt)\neq 0
\]
then $\noreq{H}= \abs{H}_{r,s,0}$, namely it does not depend on $\eta$. 
\end{rmk}

\begin{prop}\label{neminchione}
Let $f, R, \ta$ as in \eqref{analitico} and $P$
as in \eqref{hamNLS}.
Let $p> 1$, $r>0,$ $a,s,\eta\ge 0$  with $a+\eta<\ta$.
Set 
\begin{equation}\label{pangoccioli}
\Calg=\Calg(p):=2^{p+1}
\Big(1+2\sum_{n\geq 1}n^{-p}\Big)\,,\qquad
C(p,s,t):=\sup_{j} e^{-t |j|+ s \jap{j}^\teta}\jap{j}^{p}
\end{equation}
and assume that 
	$(\Calg r)^2\le  R$.
	Then 
	\begin{equation}
	\label{stimazero}
	| P|_{r,s,\eta} 
	\leq C(p,s,\ta- a -\eta)\frac{(\Calg r)^2}{R}|f|_{\ta,R}< \infty\,.
	\end{equation} 
\end{prop}

\subsection{Inhomogenous weighted Lipschitz norm}
In the following, we shall keep track of the Lipschitz dependence of the Hamiltonians on the frequency $\omega$. The frequencies will live in the set	\begin{equation}
	\pan:=
	\set{\omega=\pa{\omega_j}_{j\in \Z}\in \R^\Z,\quad \sup_j|\omega_j-j^2| \leq 1/2 }.
	\end{equation}
which is isomorphic to $[-1/2,1/2]^\Z$ (endowed with the sup-norm)
 via the map 
	 \begin{equation}\label{omegaxi}
	 \omega\ :\ \xi\mapsto \omega(\xi)\,,\quad \mbox{where}\quad \omega_j(\xi) = j^2 +\xi_j\,
	 \end{equation}
$\pan$ is  endow with the probability measure
	 $\mu$
	 induced\footnote{Denoting  by $\mu$ the measure in $\pan$ and by $\nu$ the product measure on $[-1/2,1/2]^\Z$, then $\mu(A)= \nu(\omega^{(-1)}(A))$ for all sets $A\subset\pan$ such that  $\omega^{(-1)}(A)$ is $\nu$-measurable. }  by the product measure on $[-1/2,1/2]^\Z$. 
	 
	 \medskip
	 
	 Let $\cO\subset\pan$ be a closed bounded set of positive Lebesgue measure and assume that $H = H(\omega)\in\Heta$ is $\eta$-regular uniformly with respect to $\omega\in\cO$, we define its Lipschitz semi-norm as 
\begin{equation}\label{nessundorma}
\nore{H}^{{\rm Lip},\cO}
= 
\sup_{\substack{\omega,\omega'\in\cO \\ \omega\neq \omega'}} \frac{\nore{H(\omega) - H(\omega')}^\wc}{\abs{\omega - \omega'}_{\infty}}
= 
\sup_{\substack{\omega,\omega'\in\cO\\ \omega\neq \omega'}} 
 \nore{\Delta_{\omega,\omega'}H}\,,
\end{equation}
where, as usual
$
|v|_\infty:=\sup_{j\in\Z} |v_j|
$
and
\begin{equation}
\label{delta}
\Delta_{\omega,\omega'}H := \frac{H(\omega) - H(\omega')}{\abs{\omega - \omega'}_{\infty}}\,.
\end{equation}
Set
$$
\Heta^\cO:=\Big\{
H(\omega)\in\Heta,\ \ \omega\in\cO\,,\ \ \ {\rm with}\ \ \sup_{\omega\in\cO} \nore{H(\omega)}<\infty\,,
\ \ \nore{H}^{{\rm Lip},\cO}<\infty
\Big\}
$$
and
$$
\Heta^{\cO,0}:=\left\{
H\in\Heta^\cO 
\ \ {\rm with}\ \ 
H(0)=0
\right\}\,.
$$
For any $\mu\geq 0$ 
\begin{eqnarray}
\norma{H}_{r,s,\eta}^{\mu,\cO} 
&:=& 
\sup_{\omega\in\cO} \nore{H(\omega)} +\mu 
\nore{H}^{{\rm Lip},\cO}
 \nonumber
\\
&=&
\sup_{\omega\in\cO} \nore{H(\omega)} +\mu 
\sup_{\substack{\omega,\omega'\in\cO\\ \omega\neq \omega'}} 
 \nore{\Delta_{\omega,\omega'}H}
 <\infty\,.
\label{normag}
\end{eqnarray} 
is a weighted Lipschitz semi-norm on $\Heta^\cO$ 
and a norm on $\Heta^{\cO,0}.$
It is immediate to verify that 
$\Heta^{\cO,0}$ is a Banach space
endowed with the above norm.

\begin{rmk}
In the following, for brevity, we will often write
 $\norma{\cdot}_{r,s,\eta}$
 instead of $\norma{\cdot}_{r,s,\eta}^{\mu,\cO}.$
\end{rmk}
\begin{defn}\label{proiettoindici}
	Given any subset $\mathcal S\subset \N^\Z\times \N^\Z$
let us define
$$
\Pi^{\mathcal S} H:=
\sum_{(\bal,\bbt)\in\mathcal S} H_{\bal,\bbt}u^\bal\bar{u}^\bbt\,.
$$
\end{defn}
Then
\begin{equation}\label{numeretti}
\|\Pi^{\mathcal S} H\|_{r,s,\eta}\leq \|H\|_{r,s,\eta}\,.
\end{equation}

\begin{lemma}
Given $0<r_1<r$ and $N\in\N$
\begin{equation}\label{cippi}
\|\Pi^{\{|\bal|=|\bbt|>N\}} H\|_{r_1,s,\eta}
\leq (r_1/r)^{2N}
\|\Pi^{\{|\bal|=|\bbt|>N\}} H\|_{r_1,s,\eta}
\leq (r_1/r)^{2N}
\|H\|_{r_1,s,\eta}\,.
\end{equation}
\begin{proof}
 By \eqref{normatris} and noting that
 $u_0(r_1)=(r_1/r)u_0(r)$
 (recall \eqref{giancarlo}).
\end{proof}
\end{lemma}

We  set
\begin{equation}
\label{ragno}
\Pi^\cR H := \sum_{\bal,\bbt\in\N^\Z\,: \bal\ne \bbt}H_{\bal,\bbt}u^\bal \bar u^\bbt\,,\qquad \Pi^\cK H := H- \Pi^\cR H= \sum_{\bal\in\N^\Z}H_{\bal,\bal}|u|^{2\bal}\,,
\end{equation}
and correspondigly we define the following subspaces of $\Heta$:
\begin{equation}\label{dolcenera}
\Heta^\cR:=\{ H\in \Heta\,:\quad \Pi^\cR H= H\}\,,\qquad \Heta^\cK:=\{ H\in \Heta\,:\quad \Pi^\cK H= H\}\,.
\end{equation}
By \eqref{numeretti} we get
\begin{equation}\label{numeroni}
\|\Pi^\cR H\|_{r,s,\eta}\,,\ 
\|\Pi^\cK H\|_{r,s,\eta}
\ \leq\ 
\|H\|_{r,s,\eta}\,.
\end{equation}

\subsection{Poisson structure and Hamiltonian flows}

	\begin{prop}\label{fan}
For any $F, G \in\cH_{r + \rho, s, \eta}$, with $\rho>0$, we have
$\{F,G\}\in\Heta^0$ and
\begin{equation}\label{oreo}
|\{F,G\}|_{r,s,\eta}
	\le 
	8\max\set{1, \frac{r}{\rho} }
	|F|_{r+\rho,s,\eta} |G|_{r+\rho,s,\eta}\,.
\end{equation}
Moreover the Leibniz formula holds:
\begin{equation}\label{filistei}
\partial^\bal_u\partial^\bbt_{\bar u} (f\cdot g)=
\sum_{\gamma\preceq\bal,\, \delta\preceq\bbt}
\binom{\bal}{\gamma}
\binom{\bbt}{\delta}
\partial^{\bal-\gamma}_u\partial^{\bbt-\delta}_{\bar u}f \,\cdot\,
\partial^{\gamma}_u\partial^{\delta}_{\bar u}g\qquad
\qquad
\mbox{on}\ \ \ {\bar B}_r
\end{equation}
 and
 \begin{equation}\label{ammoniti}
 \partial^\bal_u\partial^\bbt_{\bar u} \big\{f, g\big\}=
\sum_{\gamma\preceq\bal,\, \delta\preceq\bbt}
\binom{\bal}{\gamma}
\binom{\bbt}{\delta}
\Big\{
\partial^{\bal-\gamma}_u\partial^{\bbt-\delta}_{\bar u}f \,,\,
\partial^{\gamma}_u\partial^{\delta}_{\bar u}g
\Big\}
\qquad
\qquad
\mbox{on}\ \ \ {\bar B}_r\,.
\end{equation}
Analogously for any $F, G \in\cH_{r + \rho, s, \eta}^\cO$
\begin{equation}\label{commXHK}
	\norma{\{F,G\}}_{r,s,\eta}
	\le 
	8\max\set{1, \frac{r}{\rho} }
	\norma{F}_{r+\rho,s,\eta}\norma{G}_{r+\rho,s,\eta}\,.
\end{equation}
\end{prop}

The proof is given in  Appendix \ref{appendice tecnica}.

\begin{prop}[Hamiltonian flow]\label{ham flow}
	Let $S\in\cH_{r+\rho, s, \eta}^\cO$ with 
	\begin{equation}\label{stima generatrice}
	\norma{S}_{r+\rho,s,\eta} \leq\delta:= \frac{\rho}{16 e\pa{r+\rho}}. 
	\end{equation} 
	Then, for every $\o\in\cO$ the time $1$-Hamiltonian flow 
	$\Phi^1_{S(\o)}: {\bar B}_r(\ws)\to
	{\bar B}_{r + \rho}(\ws)$  is well defined, analytic in $B_r(\ws)$, symplectic with
	\begin{equation}
	\label{pollon}
	\sup_{u\in  {\bar B}_r(\ws)} 	\norm{\Phi^1_{S(\o)}(u)-u}_{s}
	\le
	(r+\rho)  \norma{S}_{r+\rho, s, \eta}
	\leq
	\frac{\rho}{16 e}.
	\end{equation}
	For any $H\in \cH_{r+\rho, s, \eta}^\cO$
	we have that
	$H\circ\Phi^1_S= e^{\set{S,\cdot}} H\in\cH_{r, s, \eta}^\cO$,
	$e^{\set{S,\cdot}} H-H \in\cH_{r, s, \eta}^{\cO,0}$
	 and
	\begin{align}
	\label{tizio}
	\norma{\es H}_{r, s, \eta}& \le 2 \norma{H}_{r+\rho, s, \eta}\,,
	\\
	\label{caio}
	\norma{\pa{\es - \id}H}_{r, s, \eta}
	&\le  \delta^{-1}
	\norma{S}_{r+\rho, s, \eta}
	\norma{H}_{r+\rho, s, \eta}\,,
	\\
	\label{sempronio}
	\norma{\pa{\es - \id - \set{S,\cdot}}H}_{r, s, \eta} &\le 
	\frac12 \delta^{-2}
	\pa{\norma{S}_{r+\rho, s, \eta}}^2
	\norma{H}_{r+\rho, s, \eta}	\end{align}
	More generally for any $h\in\N$ and any sequence  $(c_k)_{k\in\N}$ with $| c_k|\leq 1/k!$, we have 
	\begin{equation}\label{brubeck}
	\norma{\sum_{k\geq h} c_k \ad^k_S\pa{H}}_{r, s, \eta } \le 
	2 \norma{H}_{r+\rho, s, \eta} \big(\norma{S}_{r+\rho, s, \eta}/2\delta\big)^h
	\,,
	\end{equation}
	where  $\ad_S\pa{\cdot}:= \set{S,\cdot}$.
\end{prop}
The proof is completely analogous  to the one of Lemma 2.1 of \cite{BMP1:2018}
and it is based on \eqref{commXHK} and on the Lie series expansion for 
$\es.$

\subsection{Monotonicity}
The following properties of monotonicity are fundamental in bounding solutions of the linearized problem.
\begin{prop}\label{crescenza}
	The following inequalities hold:
	\begin{enumerate}
		\item { \sl Monotonicity.}
		The norm $\norma{\cdot}_{r,s,\eta}$ is  increasing\footnote{Not strictly.}  in $r,\mu,\eta$ and\footnote{Namely if $\cO'\subseteq\cO$ then 
		$\norma{\cdot}_{r,s,\eta}^{\mu,\cO'}\le 
		\norma{\cdot}_{r,s,\eta}^{\mu,\cO}$.}
		$\cO$.
				
		\item  { \sl Variation w.r.t. the parameter $s$.} 
		For any $0<\sigma< \eta$ 
		and $\Heta^\cO$
		we have
		\begin{equation}\label{emiliaparanoica}
		\norma{H}_{r,s+\s, \eta-\s} \le \norma{H}_{r,s,\eta}.
		\end{equation}
	\end{enumerate}
\end{prop}

\noindent
Note that item {\it (2)} correspond to monotonicity with respect to $s$ whenever $H$ preserves momentum. 

\begin{proof}
In order to prove 
Proposition \ref{crescenza} we first need the following Lemma, which we shall use
also in the proof of Lemma \ref{Lieder} below.
Its proof directly follows 
from \eqref{normatris}-\eqref{persico}
and \eqref{normag}.

\begin{lemma}\label{stantuffo}	  
Let 
	$
	H^{(1)}\in \cH_{r_1,s_1,\eta_1}^\cO$ and $H^{(2)}\in \cH_{r_2,s_2,\eta_2}^\cO\,,
$
	be such that,  for all $\bal,\bbt\in \N^\Z$ and  $j\in \Z$ with
	$|\bal|=|\bbt|$ and
	 $\bal_j+\bbt_j\neq 0,$ 
	one has for all $\omega\in\cO$
\begin{equation}\label{stantuffo normale}
	|H^{(1)}_{\bal,\bbt}(\omega)| c^{(j)}_{r_1,s_1,\eta_1}(\bal,\bbt)  \le |H^{(2)}_{\bal,\bbt}(\omega)| c^{(j)}_{r_2,s_2,\eta_2}(\bal,\bbt)
\end{equation}
	and for all $\omega\neq \omega'\in\cO$
	\begin{equation}\label{stantuffo delta}
	|\Delta_{\omega,\omega'}H^{(1)}_{\bal,\bbt}| c^{(j)}_{r_1,s_1,\eta_1}(\bal,\bbt)  \le |\Delta_{\omega,\omega'}H^{(2)}_{\bal,\bbt}| c^{(j)}_{r_2,s_2,\eta_2}(\bal,\bbt)\,,
	\end{equation}
	where the coefficients $c^{(j)}_{r,s,\eta}(\bal,\bbt)$ are defined in \eqref{persico}.Then
	\[
	\norma{H^{(1)}}_{r_1,s_1,\eta_1}\le \norma{H^{(2)}}_{r_2,s_2,\eta_2}\,.
	\]
\end{lemma}

\noindent
Let us come back to the proof of Proposition \ref{crescenza}.
While point {\it (1)} is immediate
from \eqref{normatris}-\eqref{persico}, \eqref{normag}
and Lemma \ref{stantuffo} (being $(\bal_j+\bbt_j)c^{(j)}_{r,s, \eta}(\bal,\bbt)$  not decreasing with respect to
$r$ and $\eta$ 
since \footnote{Note that in the case $\bal=\bbt=0$ we have $(\bal_j+\bbt_j)c^{(j)}_{r,s, \eta}(\bal,\bbt)=0$.
Otherwise, since $|\bal|=|\bbt|$,
the exponent ${\bal + \bbt - 2e_j}$ in the definition of $c^{(j)}_{r,s, \eta}(\bal,\bbt)$ in \eqref{persico}
is non negative.
 }
 $|\bal|=|\bbt|$), point {\it (2)} 
is more delicate.
We need to show that, for all $\bal,\bbt\in \N^\Z$ and  $j\in \Z$ with
	$|\bal|=|\bbt|$ and
	 $\bal_j+\bbt_j\neq 0,$
	\begin{equation}
	\label{sonno}
	c^{(j)}_{r,p,s+\s,a,\eta-\s}(\bal,\bbt) \le  c^{(j)}_{r,p,s,a,\eta }(\bal,\bbt)\,,
	\end{equation}
	namely that
$$
	\exp\Big[-\s \Big(\sum_i \jap{i}^\theta (\bal_i+\bbt_i) -2\jap{j}^\theta+ |\pi(\bal-\bbt)|\Big)\Big] 
	\le 1
$$
or, equivalently, that
	\begin{equation}\label{stima1}
	\sum_i \jap{i}^\theta (\bal_i+\bbt_i) -2\jap{j}^\theta+ |\pi(\bal-\bbt)|
	 \ge  
	 0\,.
	\end{equation}
	 The proof of the non trivial estimate \eqref{stima1} is contained in 
\cite{BMP1:2018} (see formula (3.20))
and it is based on an idea proposed by  Bourgain in 
\cite{Bourgain:2005} (see also \cite{Yuan_et_al:2017}).
\end{proof}

\section{Small divisors and Homological equation}

\subsection{Small divisors} \label{Homo}
We star by recalling that the set of 
Diophantine frequencies, of Definition \ref{fantino} are {\sl typical } in $[-1/2,1/2]^\Z$, namely
there exists a positive constant $C$ such that
\begin{equation}\label{misura}
\meas\big([-1/2,1/2]^\Z\setminus \dgp\big)
\leq C\g\,,
\end{equation}
where $\meas$ is the  product measure on $[-1/2,1/2]^\Z.$
The proof of \eqref{misura} is contained in \cite{BMP1:2018} (see Lemma  4.1).

\subsection{Homological equation}
The proof of the following classical Lemma (core of any small divisors problem), relies on some notation and results introduced by Bourgain in \cite{Bourgain:2005} and extended later on by Cong-Li-Shi-Yuan in \cite{Yuan_et_al:2017} (see Lemma \ref{constance 2 gen} below). We shall send the reader to \cite{BMP1:2018} for the detailed proof of such results.


Recalling the definitions of
$\Pi^\cR$ and $\Heta^\cR$ given in \eqref{ragno} and \eqref{dolcenera},  
we introduce the following operator on the space of formal power series:
\begin{equation}
L_\omega G: = \sum \im\pa{\omega \cdot(\bal-\bbt)}G_{\bal,\bbt}\buu.
\end{equation}
The operator $L_\omega$ is nothing but the action of the Poisson bracket $\set{\sum_j\omega_j\abs{u_j}^2, \cdot}$ and it is  invertible on the subspace of  formal power series such that  $F=\Pi^\cR F$ with inverse
\begin{equation}
\label{lala}
L_\omega^{-1} F = G:= \sum \frac{F_{\bal,\bbt}}{\im\pa{\omega \cdot(\bal-\bbt})}\buu
\end{equation}
\begin{lemma}[Straightening the torus]\label{Lieder}Let $0< \sigma <\min\{\eta,1\}$,
 $r>0$ and 
	let $ \dgp\ni\omega \mapsto F(\omega) \in\Heta^{\cR}$ be a Lipschitz family of Hamiltonians.
Then, defining $G$ as in \eqref{lala},  the following bound holds
	\begin{equation*}
	\norma{G}_{r, s+\sigma, \eta - \sigma}^{\g,\dgp}
	\le 
	\gamma^{-1}
	e^{\croc{\sigma^{-\frac{3}{\teta}}}}
	\norma{F}_{r , s , \eta}^{\g,\dgp}
	\end{equation*}
	for suitable $ \croc\geq 1$.
\end{lemma}

\begin{proof}
	We first claim that
	\begin{equation}\label{pontina}
	\sup_{\o\in\dgp}|G|_{r,s+\s,\eta-\s}\leq \frac{1}{3\g} 
	e^{\croc{\sigma^{-\frac{3}{\teta}}}}
	\sup_{\o\in\dgp}|F|_{r , s , \eta}
	\end{equation}
	Note that by Lemma \ref{elisabetta} estimate \eqref{pontina} 
	ensures that the formal power series $G$ is actually totally convergent, together   with 
	$\und G_{\eta-\s}.$
	
	Let us prove \eqref{pontina}.	In order to apply Lemma \ref{stantuffo} and get the stated bound, we shall start by proving that, for all $j\in\Z$ and $\bal,\bbt$ with $|\bal|=|\bbt|$, $\bal\ne \bbt$ and $\bal_j+\bbt_j\neq 0$ one has 
	\[
	\frac{|\bcoef{F}(\omega)|}{|\omega\cdot\pa{\bal - \bbt}|} c^{(j)}_{r,s+\s,\eta-\s}(\bal,\bbt) \le \frac{1}{3\g}
	e^{\croc\s^{-\frac{3}{\teta}}} |\bcoef{F}(\omega)|c^{(j)}_{r,s,\eta}(\bal,\bbt)
	\]
	for a suitable $\croc\geq 1$ large enough.
	This is equivalent to proving
	\begin{equation}\label{virgilio}
	\frac{e^{-\s\pa{\sum_i\jap{i}^\theta (\bal_i+\bbt_i) -2\jap{j}^\theta+|\pi|}}}{\abs{\omega\cdot{\pa{\bal - \bbt}}}}
	\leq
	\frac{1}{3\g} e^{\croc\s^{-\frac{3}{\teta}}}\,.
	\end{equation}
	In order to prove \eqref{virgilio} we consider two cases. \\
	The first case is when
	\begin{equation}\label{onbroadway}
	\abs{\sum_i{\pa{\bal_i-\bbt_i}i^2}}\ge 10 \abs{\bal-\bbt}\,.
	\end{equation}
	Then, denoting $\omega_j = j^2 + \xi_j $ with $\abs{\xi_j}\le \frac{1}{2}$, 
	\begin{equation}\label{thelamb}
	\abs{\omega\cdot\pa{\bal-\bbt}} \ge 10\abs{\bal-\bbt}
	- \frac{1}{2}\abs{\bal-\bbt}\ge 9\abs{\bal-\bbt}\geq 9\,.
	\end{equation}
	Then by \eqref{sonno}
	the 
	left hand side of \eqref{virgilio} is bounded by $1/9$
	and \eqref{virgilio} follows (since $\g < 1$ and $\croc>0.$)
	\\
	Otherwise, 
	in order to control small divisors, we shall make use of the following result 
	(see \cite{Bourgain:2005} and Lemma 4.2 of \cite{BMP1:2018}).
	
	\begin{lemma}\label{constance 2 gen}
		Consider $\bal,\bbt\in\N^\Z$ with $1\leq |\bal|=|\bbt|$ and $\bal\ne \bbt$.
		If
		\begin{equation}\label{divisor}
		\abs{\sum_i{\pa{\bal_i-\bbt_i}i^2}}\le 10 \sum_i\abs{\bal_i-\bbt_i}
		\end{equation}
		then
		for all $j$ such that $\bal_j+\bbt_j\neq 0$ one has
		\begin{equation}\label{adele}
		\sum_i\abs{\bal_i-\bbt_i}\jap{i}^{\theta/2} 
		\le 
		\frac{13}{1-\teta}\pa{\sum_i \pa{\bal_i+\bbt_i}\jap{i}^\teta- 2\jap{j}^\teta  + \abs{\pi}},
		\end{equation}
		where $\pi = \sum_i i\pa{\bal_i - \bbt_i}$ is defined in \ref{momento}.
	\end{lemma}

	\noindent
	Then if \eqref{divisor} holds,
	by Lemma \ref{constance 2 gen} and using that 
	$\omega\in \dgp$, the following chain of inequalities holds
	\begin{align}	
	&\frac{e^{-\s\pa{\sum_i\jap{i}^\theta (\bal_i+\bbt_i) -2\jap{j}^\theta+|\pi|}}}{\abs{\omega\cdot{\pa{\bal - \bbt}}}}  
	\le 
	\g^{-1} e^{-\frac{\s}{C_*(\teta)}\sum_i\abs{\bal_i-\bbt_i}\jap{i}^{\frac{\teta}{2}}}
	\prod_i\pa{1+|\bal_i-\bbt_i|^2\jap{i}^2}
	\nonumber
	\\
	& 
	\le 
	\g^{-1}\exp{\sum_i\sq{-\frac{\s(1-\teta)}{13} \abs{\bal_i - \bbt_i}\jap{i}^{\frac{\teta}{2}}
			+ \ln{\pa{1 + |\bal_i - \bbt_i|^2\jap{i}^2}}}} 
	\nonumber
	\\
	& 
	= \g^{-1}\exp{\sum_i f_i(\abs{\bal_i-\bbt_i},\s)} 
	\label{tachipirina}
	\end{align}
	where, for $0<\s\leq 1$, $i\in \Z$ and $x\geq 0$, we set
	\begin{equation}\label{irlanda}
	f_i(x,\s) := -\frac{\s(1-\teta)}{13} x\jap{i}^{\frac{\teta}{2}} + \ln{\pa{1 + x^2\jap{i}^2}}\,.
	\end{equation}
	Now we exploit the following estimate whose proof is given \cite{BMP1:2018}
	(see Lemma 4.5)
	\begin{equation}\label{scozia}
	\sum_i f_i(|\ell_i|,\s)\leq 
	21
	i_\sharp(\s) \ln i_\sharp(\s)\,,
	\qquad
	{\rm with}\ \ \ 
	i_\sharp(\s):=
	\left(
	\frac{312}{\s\theta(1-\teta)}
	\ln
	\frac{156}{\s\theta(1-\teta)}
	\right)^{\frac2\theta}
	\end{equation}
	for every $\ell\in\Z^\Z$ with $|\ell|<\infty$.
	By \eqref{tachipirina} we get that \eqref{virgilio}
	holds for $\croc$ large enough;
	then \eqref{pontina} follows.
	
	\smallskip
	
	Let us now estimate the Lipschitz semi-norm.
	By Leibniz's rule we have (recall \eqref{delta})
	\begin{eqnarray}\label{delta inversa}
	\Delta_{\omega,\omega'} L^{-1}_\omega F
	&=& 
	\sum\frac{\Delta_{\omega,\omega'} \bcoef{F}}{\im\omega\cdot\pa{\bal - \bbt}}\buu 
	+
	\sum \bcoef{F}(\o')\Delta_{\omega,\omega'} \pa{\frac{1}{\im\omega\cdot\pa{\bal - \bbt}}}\buu	
	\nonumber
	\\
	&=:&G_1+G_2
	\end{eqnarray}
	Arguing as in the estimate of $G$ in \eqref{pontina}
	we get
	\begin{equation}\label{pontina2}
	\sup_{\substack{\o,\o'\in\dgp\\ \o\neq\o'}}
	|G_1|_{r,s+\s,\eta-\s}
	\leq 
	\frac{1}{3\g}
	e^{\croc{\sigma^{-\frac{3}{\teta}}}}
	\sup_{\substack{\o,\o'\in\dgp\\ \o\neq\o'}}
	|\Delta_{\omega,\omega'} F|_{r , s , \eta}\,.
	\end{equation}
	
	Regarding the term $G_2$ we claim that
	\begin{equation}\label{pontina3}
	\sup_{\substack{\o,\o'\in\dgp\\ \o\neq\o'}}
	|G_2|_{r,s+\s,\eta-\s}
	\leq \frac{1}{3\g^2} 
	e^{\croc{\sigma^{-\frac{3}{\teta}}}}
	\sup_{\substack{\o,\o'\in\dgp\\ \o\neq\o'}}
	| F|_{r , s , \eta}\,,
	\end{equation}
	taking $\croc\geq 1$ large enough.
	Set for brevity
	$$
	\Delta_{\o,\o'}(\bal,\bbt):=\abs{\Delta_{\omega,\omega'} \pa{\frac{1}{\im\omega\cdot\pa{\bal - \bbt}}} }\,.
	$$
	By \eqref{delta inversa} 
	the claim \eqref{pontina3} follows by
	Lemma \ref{stantuffo} if we prove that, for all $j\in\Z$ and $\bal,\bbt$ with $|\bal|=|\bbt|$ and $\bal_j+\bbt_j\neq 0$  
	$$
	\Delta_{\o,\o'}(\bal,\bbt)
	c^{(j)}_{r,s+\s,\eta-\s}(\bal,\bbt) 
	\le 
	\frac{1}{3\g^2} 
	e^{\croc{\sigma^{-\frac{3}{\teta}}}} c^{(j)}_{r,s,\eta}(\bal,\bbt)
	$$
	or, equivalently,  taking the logarithm 
	\begin{equation}\label{virgilio2}
	-\s\pa{\sum_i\jap{i}^\theta (\bal_i+\bbt_i) -2\jap{j}^\theta+|\pi|}
	+\ln\Delta_{\o,\o'}(\bal,\bbt)
	\leq \ln \frac{1}{3\g^2} 
	e^{\croc{\sigma^{-\frac{3}{\teta}}}}\,.
	\end{equation}
	We have that
	\begin{eqnarray}\label{antibiotico}
	\Delta_{\o,\o'}(\bal,\bbt)
	&=& 
	\frac{\abs{\frac{1}{\im\divisor{\omega}} - \frac{1}{\im\divisor{\omega'}}}}{\abs{\omega - \omega'}_{\infty}} 
	= 
	\frac{\abs{\pa{\omega' - \omega}\cdot\pa{\bal - \bbt}}}{|\divisor{\omega}|\,|\divisor{\omega'}|}{\abs{\omega - \omega'}_{\infty}^{-1}}
	\nonumber
	\\
	&\leq&
	\frac{\abs{\bal - \bbt}}{|\divisor{\omega}|\,|\divisor{\omega'}|}
	\,.
	\end{eqnarray}
	As above we have two cases.
	In the first case, namely when \eqref{onbroadway} holds, by \eqref{thelamb} and \eqref{antibiotico}
	we get
	$$
	\Delta_{\o,\o'}(\bal,\bbt)\leq 1\,.
	$$
	Then, by \eqref{stima1},  \eqref{virgilio2} holds in this first case since
	its left hand side is negative while its right hand side is positive.
	\\
	In the second case, when \eqref{divisor} holds, by \eqref{adele} and \eqref{antibiotico}
	we get,
	for $\omega,\o'\in \dgp$,
	\begin{eqnarray}\label{stima delta}
	\Delta_{\o,\o'}(\bal,\bbt)
	& \le&
	\abs{\bal - \bbt} \gamma^{-2} \prod_{i\in \Z}{(1+ |\bal_i - \bbt_i|^2 \jap{i}^2)}^2
	\nonumber
	\\
	& \le&
	\gamma^{-2} \prod_{i\in \Z}{(1+ |\bal_i - \bbt_i|^2 \jap{i}^2)}^3\,.
	\end{eqnarray}
	Then by \eqref{adele}
	\begin{eqnarray*}
		&&
		-\s\pa{\sum_i\jap{i}^\theta (\bal_i+\bbt_i) -2\jap{j}^\theta+|\pi|}
		+\ln\Delta_{\o,\o'}(\bal,\bbt)
		\\
		&\leq&
		\ln \g^{-2}+3
		\sum_i\sq{-\frac{\s(1-\teta)}{39} \abs{\bal_i - \bbt_i}\jap{i}^{\frac{\teta}{2}}
			+ \ln{\pa{1 + |\bal_i - \bbt_i|^2\jap{i}^2}}}
		\\
		&\stackrel{\eqref{irlanda}}=&
		\ln \g^{-2}+3
		\sum_i
		f_i(|\bal_i-\bbt_i|,\s/3)
		\\
		&\stackrel{\eqref{scozia}}\leq&
		\ln \g^{-2}+
		63
		i_\sharp(\s/3) \ln i_\sharp(\s/3)\,.
	\end{eqnarray*}
	Then \eqref{virgilio2} and, hence,
  \eqref{pontina3} follow also in the second case
	taking $\croc$ large enough.

	Recollecting, recalling \eqref{normag}, \eqref{delta inversa} and using 
	\eqref{pontina}, 
	\eqref{pontina2} and \eqref{pontina3}, we get
	\begin{eqnarray*}
		&&\|G\|_{r,s+\s,\eta-\s}^{\g,\dgp}
		\\
		&&\leq\sup_{\o\in\dgp}|G|_{r,s+\s,\eta-\s}
		+\g
		\sup_{\substack{\o,\o'\in\dgp\\ \o\neq\o'}}
		|G_1|_{r,s+\s,\eta-\s}
		+\g
		\sup_{\substack{\o,\o'\in\dgp\\ \o\neq\o'}}
		|G_2|_{r,s+\s,\eta-\s}
		\\
		&&
		\leq 
		\frac{1}{\g} 
		e^{ \croc{\sigma^{-\frac{3}{\teta}}}}
		\|F\|_{r,s,\eta}^{\g,\dgp}\,.
	\end{eqnarray*}
	The proof is completed. 
\end{proof}

\section{Projections on the torus}\label{provola}

We now fix a {\it torus} $\cT_I$ associated to an action $I$, as explained in the introduction, then given a regular Hamiltonian we define a {\it degree decomposition} with increasing order of zero at $\cT_I$. The general ideas of this section were sketched  in \cite{Bourgain:2005}, here we adapt them to our norm and give detailed explanations of the projections and their properties.
\\
We fix  $r>0$ and (recall \eqref{polenta})
\begin{equation}\label{polenta2}
\pa{I_j}_{j\in\Z}\,,\quad
I_j\geq 0\,,\quad \forall j\in\Z\,,\qquad {\rm with}\quad
\quad |\sqrt{I}|_{s}\leq \kappa r \,, \quad \kappa<1\,.
\end{equation}
Then we define the torus
$$
\cT_I=\big\{ u \in \tw_{p,s,a}^\infty: |u|^2= I \big\}\, 
$$
where for brevity we set
\begin{equation}\label{giammai}
\big(|u|^2\big)_j:= |u_j|^2\,,\qquad j\in\Z\,.
\end{equation}
Note that by construction $\cT_I$ 
is contained in the interior of ${\bar B}_r(\tw_{s'})$ for all $s'\le s$.
In the following we will use the quantity
	\begin{equation}\label{tegolino}
	c_{\kappa}
	:=
	\left\{
	\begin{array}{ll}
\displaystyle	\frac{1}{\ln\kappa^{-2}}
	&    {\rm if} \ \frac12<\kappa^2<1 \,, \\
	& \\
\displaystyle	2\kappa^2&       {\rm if} \  0<\kappa^2\leq\frac12\,.
	\end{array}
	\right.
	\end{equation}

The main result of the section is condensed in the following

\begin{prop}\label{proiettotutto}
	Let $I$ be fixed as in \eqref{polenta2}.
	For every $d=2q-2$, $q\in\N$
	and $I$  fixed as in \eqref{polenta2},
	 there exist linear continuous ``projection'' operators
	$\Pi^d: \Heta^\cO\to \Heta^\cO$ such that the following holds:
	
\begin{itemize}
	\item[(i)] 
	$\Pi^d \Pi^d = \Pi^d $ and
	$\Pi^{d'} \Pi^d=\Pi^d \Pi^{d'} =0$
	for every even $d'\neq d$,  $d'\geq -2.$

  \item[(ii)] 		
	For every $\kappa_*\leq 1$ 
	\begin{equation}\label{cacioepepe}
\|\Pi^{2q-2} H \|_{\kappa_* r,s,\eta}
\leq 
\kappa_*^{-2}
\pa{1+\frac{\kappa_*^2}{\kappa^2}}^{{q}}
c_\kappa^q
\|H\|_{r,s,\eta}\,,
\end{equation}
with $c_\kappa$ defined in \eqref{tegolino}.

\item[(iii)] For any  $\mathtt q\ge 0$,
\begin{equation}\label{persiani}
\Pi^{2q-2} H=0\,,\;\forall q \le \mathtt   q
\qquad
\Longrightarrow
\qquad
\partial_u^{\bal}\partial_{\bar u}^{\bbt}\, H\equiv 0
\qquad\mbox{on}\quad \cT_I\,,\qquad
\mbox{for}\quad 0\leq|\bal|+|\bbt|\leq \mathtt   q\,.
\end{equation}

\item[(iv)]
Setting for brevity 
$$
\Pi^{<d}:=\sum_{\substack{j =-2\\  j\in 2\Z}}^{d-2}  \Pi^j \,,
$$
	if $\Pi^{< d_1}F = \Pi^{< d_2} G=0$, then 
	 \begin{equation}\label{cicladi}
	 \Pi^{< d_1+d_2}\set{F,G}=0\,.
\end{equation}
\end{itemize}
\end{prop}
\begin{rmk}
(i) The ``projection'' operators
	$\Pi^d$ are explicitly defined below in formula \eqref{accad}.
	\\
(ii)	
For $u$ such that\footnote{$|u|^2$ is defined in \eqref{giammai}. }	
 $ |u|^2$ belongs to the interior of
$ I +{\bar B}_{(1-\kappa^2)r^2}(\tw_{2p,2s,2a})$ we have
\begin{equation}
\label{accad3}
H(u)= \sum_{q\in \N} \Pi^{2q-2} H(u)\,, 
\end{equation}
Moreover for $\kappa^2<1/2$
	\eqref{accad3} actually holds in the whole ball ${\bar B}_r.$
\\
(iii)	The
	viceversa of \eqref{persiani} does not hold, in general, if $I_j=0$ for some $j\in\Z$. On the other hand if all the $I_j$ are positive the
	viceversa of \eqref{persiani} holds.
\end{rmk}


\noindent{\bf Notation}:\;
for any even integer $d\geq -2$  we define 
$$
\Pi^{\geq d}:=\id-\Pi^{<d}\,.
$$
moreover for $H\in\Heta^\cO$
we set
\begin{equation}\label{norma}
H^{(d)}:=\Pi^{d} H \,,\quad
H^{\le d}:=\Pi^{\le d} H := \sum_{\substack{j =-2\\  j\in 2\Z}}^{d} H^{(j)} \,,
\quad
H^{\ge d}:=\Pi^{\ge d} H \,.
\end{equation}		
and denote
\begin{equation}
\Heta^{(d)} = \set{H\in\Heta : \Pi^d H = H}.
\end{equation}

In the proof of Proposition \ref{proiettotutto}, in particular points {\it (iii)-(iv)},
we will use the following result, which is interesting in itself.

\begin{prop}[Bourgain's representation]\label{burger}
Let $I$ and $\kappa$ be as in \eqref{polenta2}.
Let $H\in \Heta^{\mathcal O}$ and  $q\geq 0$.
Then 
 the following representation formula holds
\begin{equation}\label{rappre}
\Pi^{\ge 2q-2} H(u) = \sum_{|\delta|=q} (|u|^2-I)^\delta \check{H}_\delta(u)
\qquad\mbox{for}\ \ |u|_s< r\,,
\end{equation}
where $\check{H}_\delta(u)$ are analytic  in ${\bar B}_r(\tw_s)$ and can be written in totally convergent power series in every ball
$|u|_s\leq \kappa_* r$ with $\kappa_*< 1$.
\\
Moreover
\begin{equation}\label{sparta}
\|\Pi^{\ge 2q-2} H\|_{\kappa_* r,s,\eta}
\le 
 \frac 1{\kappa_*^2} \pa{\frac{\kappa^2+\kappa_*^2}{\kappa_*^2}c_{\kappa_*}}^{q}
 \| H\|_{ r,s,\eta}\,,
\end{equation}
where
 $c_{\kappa_*}$ was defined in \eqref{tegolino}.
\end{prop}

We prove first Proposition \ref{proiettotutto} and then
Proposition \ref{burger}.

\medskip		

\begin{proof}[Proof of Proposition \ref{proiettotutto}]

We start constructing the projections $\Pi^d.$
\\
Assume that a formal power series
$$
H(u) = \sum_{\abs{\bal} = \abs{\bbt}}\widehat H_{\bal,\bbt}u^{\bal}\bar{u}^{\bbt}
$$
satisfies $|H|_{r,s,\eta}<\infty.$
Since by Lemma \ref{elisabetta} the above series totally converges
in $|u|_s\leq r,$
we can  rearrange its terms in the unique way
\begin{equation}\label{H disgiunta}
H(u) = \sum^\ast H_{m,\al,\bt}\abs{u}^{2m}u^\al\bar{u}^\bt\,,\qquad
{\rm with}\qquad 
H_{m,\al,\bt}=  \widehat H_{m+\al,m+\bt}\,,
\ \ 
m, \alpha, \beta\in\N^{\Z}
\end{equation}
and 
\begin{equation}\label{drago}
\sum^\ast:=
\sum_{
\substack{m, \alpha, \beta\in\N^{\Z}
\\  
\abs{\alpha} = \abs{\beta},\ \al_j
\bt_j=0,\, \forall\, j\in\Z
}
}\,.
\end{equation}
Note that the
map
$$
(m,\al,\bt)\to(\bal,\bbt)\,,\qquad
\bal:=m+\al\,,\ \ \bbt:=m+\bt
$$
with inverse 
$$
\al:=\bal-m\,,\quad 
\bt:=\bbt-m\,,\qquad
m_j:=\min\{\bal_j,\bbt_j\}\,,\quad
\forall j\in\Z\,,
$$
is a bijection between the
 sets of indexes
$$
\big\{ m,\al,\bt\in\N^\Z\ \ :\ \  |\al|=|\bt|\,,\ \al_j\bt_j=0\,,\ \ \forall j\in\Z  \big\}
\ \ {\rm and}\ \ 
\big\{ \bal,\bbt\in\N^\Z\ \ :\ \  |\bal|=|\bbt|\big\}\,.
$$

Note that  by \eqref{normatris} we get
\begin{equation}
\label{normatris2}
\abs{H}_{r,s,\eta} =\frac12  \sup_j \sum^\ast \abs{\coef{H}}
\pa{2m_j + \al_j + \bt_j}u_0^{2m + \al + \bt - 2e_j}\peseta.
\end{equation}
where $u_0=u_0(r)$ is defined in
\eqref{giancarlo}.

Given 
\begin{equation}\label{polenta2a}
w = \pa{w_j }_{j\in\Z}\in \tw^\infty_{2p,2s,2a}\,,\quad
|w|_{2p,2s,2a}\leq r^2\,,
\end{equation}
and noting that $|w_j|\leq u_{0,j}^2$ for every $j\in\Z,$
we
define  (recall \eqref{drago})
\begin{equation}\label{coppiette}
\tH(u,w) := \sum^\ast H_{m,\al,\bt} w^m u^\al\bar u^\beta\,.
\end{equation} 
Reasoning as in \eqref{abacab} we get
that the series in \eqref{coppiette}
{\it totally converges} in $\{|w|_{2p,2s,2a}\leq r^2\}
\times \{  |u|_s\leq r\}$
with estimate
\begin{eqnarray}
\sum^\ast
\sup_{|w|_{2p,2s,2a}\leq r^2} \sup_{ |u|_s\leq r}
 \abs{H_{m,\al,\bt}} |w^m| |u^\al| |\bar u^\beta|
&=&
\sum^\ast \abs{H_{m,\al,\bt}}
u_0^{2m+\al+\bt}
\nonumber
\\
&\leq&
|H_{0,0,0}|+ r^2 |H|_{r,s,\eta}\,.
 \label{vesta}
\end{eqnarray}
Then $\tH(u,w)$ is analytic in the interior of
\begin{equation}\label{gianicolo}
{\bar B}_{r^2}(\tw^\infty_{2p,2s,2a})\times {\bar B}_r(\tw_s)\ \ni\ (w,u)\,.
\end{equation}
Note that 
\begin{equation}\label{setteemezzo}
H(u) = \tH(u,\abs{u}^2)\,.
\end{equation}

\medskip

	Recalling \eqref{polenta2}, and  defining the multi-index $\kappa^2 u_0^2$,
with $(\kappa^2u_0^2)_j:=\kappa^2u_{0,j}^2$ for every $j\in\Z,$ we have
\begin{equation}\label{cerere}
I\preceq \kappa^2 u_0^2\,.
\end{equation}
Let $H\in\Heta.$
For every $u$ in the interior of ${\bar B}_r (\tw_s)$, by analyticity we can 
Taylor expand the function $w\mapsto \tH(u,w)$ at $w=I$
obtaining
\[
\tH(u,w)= \sum_{q\in \N} \frac{1}{q!}D_w^q \tH(u,I) [\underbrace{w- I,\dots,w-I}_{q\ {\rm times}}] \,,
\] 
for $w-I$ in the interior of
${\bar B}_{(1-\kappa^2)r^2}(\tw_{2p,2s,2a})$.
For every $u$ in the interior of ${\bar B}_r (\tw_s)$,
 $$
 D_w^q \tH(u,I):
 \underbrace{\tw_{2p,2s,2a}\times 
 \ldots\times \tw_{2p,2s,2a}}_{q\ {\rm times}}\ 
 \to \mathbb C
 $$ is the bounded $q$-linear symmetric operator  
of the $q$'th derivative evaluated 
at\footnote{Note that
\begin{equation}\label{sublimeporta}
\frac{1}{q!}D_w^q \tH(u,I) [\underbrace{\xi,\dots,\xi}_q]
=\sum_{|\delta|=q}\frac{1}{\delta !} \partial^\delta_w \tH(u,I)\xi^\delta. 
\end{equation}
} $w=I$.
For
$q\in\N^+$ and $u$ in the interior of
${\bar B}_r(\tw_s)$
  we set 
\begin{equation}
\label{accad}
\Pi^{-2}H(u)=H^{(-2)}(u):=\tH(u,I)
\qquad
{\rm and}
\qquad
\Pi^{2q-2}H(u)=
H^{(2q-2)}(u):=
\frac{1}{q!}D_w^q \tH(u,I) [\underbrace{|u|^2- I,\dots,|u|^2-I}_{q\ {\rm times}}] 
\,.
\end{equation}
Then,
recalling \eqref{setteemezzo},
we get
\begin{equation}\label{cannolo}
H(u)=\tH(u,|u|^2)= \sum_{q\in \N} H^{(2q-2)}(u)
\end{equation}
for
$
|u|^2$
in the interior of
$ I +{\bar B}_{(1-\kappa^2)r^2}(\tw_{2p,2s,2a})$.

We now rewrite $H^{(2q-2)}$ in order to prove that
it is not only analytic but also regular.
Since the norm in \eqref{norma1}, equivalently
\eqref{normatris2}, is not intrinsic but is defined in terms
of the coefficients of the Taylor expansion at zero,
we now expand $H^{(2q-2)}(u)$ at $u=0.$
Since for any $u$ in the interior of  $ {\bar B}_r(\tw_s)$
$$
\frac{1}{q!}D_w^q \tH(u,I)[\underbrace{\xi,\dots,\xi}_{q\ {\rm times}}] 
=
\sum_{|\delta|=q}
\frac{1}{\delta!}\partial_w^\delta \tH(u,I) \xi^\delta
=
\sum_{|\delta|=q}
\sum^\ast_{m\succeq\delta,\al,\bt}
\binom{m}{\delta} H_{m,\al,\bt}
I^{m-\delta} u^\al \bar u^\bt \xi^\delta
$$
by the definition in \eqref{accad} we get 
\begin{equation}\label{cippalippalippa}
H^{(2q-2)}(u)
=
\sum_{|\delta|=q}
\pa{\abs{u}^2 - \II}^\delta
\sum^\ast_{m\succeq\delta,\al,\bt}
\binom{m}{\delta} H_{m,\al,\bt}
I^{m-\delta} u^\al \bar u^\bt\,.
\end{equation}
Since
$$
\pa{\abs{u}^2 - \II}^\delta
=
\sum_{\gamma\preceq \delta}
 \binom{\delta}{\gamma}
(-1)^{|\delta-\gamma|}
I^{\delta-\gamma}|u|^{2\gamma}
$$
we get
\begin{eqnarray}\label{cippalippa}
H^{(2q-2)}(u)
=
\sum_{|\delta|=q}\ 
\sum_{\gamma\preceq \delta}\
\sum^\ast_{m\succeq\delta}
\binom{m}{\delta}
 \binom{\delta}{\gamma}
(-1)^{|\delta-\gamma|}
I^{m-\gamma}
 H_{m,\al,\bt}
 |u|^{2\gamma}
 u^\al \bar u^\bt 
\,.
\end{eqnarray}
First we show that
the series in \eqref{cippalippa}
totally converges in the closed ball
$|u|_{p,s,a}\leq r$.
Indeed for $\kappa_*\leq 1$ we get
\begin{eqnarray}
&&
\sum_{|\delta|=q}\ 
\sum_{\gamma\preceq \delta}\
\sum^\ast_{m\succeq\delta}
\sup_{|u|_{p,s,a}\leq \kappa_* r}
\binom{m}{\delta}
 \binom{\delta}{\gamma}
I^{m-\gamma}
 |H_{m,\al,\bt}|
 |u|^{2\gamma+\al+\bt}
 \nonumber
 \\
 &\stackrel{\eqref{giancarlo}}=&
\sum_{|\delta|=q}\ 
\sum_{\gamma\preceq \delta}\
\sum^\ast_{m\succeq\delta}
\binom{m}{\delta}
 \binom{\delta}{\gamma}
I^{m-\gamma}
 |H_{m,\al,\bt}|
 (\kappa_* u_0)^{2\gamma+\al+\bt}
 \nonumber
 \\
 &\leq&
 \frac12 \sup_j 
\sum_{|\delta|=q}\ 
\sum_{\gamma\preceq \delta}\
\sum^\ast_{m\succeq\delta}
\binom{m}{\delta}
 \binom{\delta}{\gamma}
I^{m-\gamma}
 |H_{m,\al,\bt}|
 (2m_j+\al_j+\bt_j)
 (\kappa_* u_0)^{2\gamma+\al+\bt}
 \nonumber
 \\
 &\leq&
 (\kappa_*r)^2
 \frac12 \sup_j 
\sum_{|\delta|=q}\ 
\sum_{\gamma\preceq \delta}\
\sum^\ast_{m\succeq\delta}
\binom{m}{\delta}
 \binom{\delta}{\gamma}
I^{m-\gamma}
 |H_{m,\al,\bt}|
 (2m_j+\al_j+\bt_j)
 (\kappa_* u_0)^{2\gamma+\al+\bt-2e_j}
 \peseta
 \nonumber
 \\
 &=:& C_{q,\kappa_*} 
 \label{cippalippa2}
 \end{eqnarray}
 Noting that
$$
	I^{m-\gamma}
	\stackrel{\eqref{cerere}}\leq 
	\kappa^{2|m-\gamma|} u_0^{2m-2\gamma}
	\leq \kappa^{2|m|-2|\gamma|} u_0^{2m-2\gamma}
	$$
we have
\begin{eqnarray*}
C_{q,\kappa_*} 
&\leq&
 \frac{r^2}2 \sup_j 
 \sum_{|\delta|=q}\ 
\sum^\ast_{m\succeq\delta}\ 
\sum_{\gamma\preceq \delta}
\binom{m}{\delta}
 \binom{\delta}{\gamma}
\left(\frac{\kappa_*}{\kappa}\right)^{2|\gamma|}
  \kappa^{2|m|}
 |H_{m,\al,\bt}|
 (2m_j+\al_j+\bt_j)
 u_0^{2m+\al+\bt-2e_j}
 \peseta
 \nonumber
 \\
 &\leq&
 \pa{1+\frac{\kappa_*^2}{\kappa^2}}^{{q}}
 \frac{r^2}2 \sup_j 
 \sum_{|\delta|=q}\ 
\sum^\ast_{m\succeq\delta}
\binom{m}{\delta}
  \kappa^{2|m|}
 |H_{m,\al,\bt}|
 (2m_j+\al_j+\bt_j)
 u_0^{2m+\al+\bt-2e_j}
 \peseta
 \end{eqnarray*}
noting that for $|\delta|=q$
\begin{equation}\label{ponto}
\sum_{\gamma\preceq\delta} 
	\binom{\delta}{\gamma}
	\pa{\frac{\kappa_*}{\kappa}}^{2|\gamma|}
	=
	\pa{1+\frac{\kappa_*^2}{\kappa^2}}^{|\delta|}=
	\pa{1+\frac{\kappa_*^2}{\kappa^2}}^{{q}}\,.
\end{equation}

Now we need the following technical result, whose proof is postponed to the Appendix.	
\begin{lemma}\label{appendicite}
 Let $0<\kappa< 1$. 	
	For  $q\in\N$  and $|m|\geq q$ we 
	get\footnote{We set $\binom{m}{\delta}:=
	\prod_{i\in\Z} 
	\binom{m_i}{\delta_i}\,.
	$
	}
	\begin{equation}\label{fornaio}
\kappa^{2|m|}
\sum_{\substack{ |\delta|={q}\\ \delta\preceq m}} 
 \binom{m}{\delta}
 \leq
 c_\kappa^q \,,
\end{equation}
where $c_{\kappa}$ was defined in \eqref{tegolino}.
\end{lemma}

Then by \eqref{fornaio}
and exchanging the order of summation
between $\delta$ and $m$
we get
\begin{eqnarray}
C_{q,\kappa_*} 
 &\leq&
 \pa{1+\frac{\kappa_*^2}{\kappa^2}}^{{q}}
 c_\kappa^q
 \frac{r^2}2 \sup_j 
\sum^\ast
 |H_{m,\al,\bt}|
 (2m_j+\al_j+\bt_j)
 u_0^{2m+\al+\bt-2e_j}
 \peseta
 \nonumber
 \\
 &=&
 \pa{1+\frac{\kappa_*^2}{\kappa^2}}^{{q}}
 c_\kappa^q
 r^2 |H|_{r,s,\eta}\,.
 \label{falcone}
 \end{eqnarray}
This implies that the series in \eqref{cippalippa}
totally converges in the closed ball
$|u|_{p,s,a}\leq r$ (taking $\kappa_*=1$).
\\
Then
we can exchange the order of summation
obtaining the representation formula
\begin{equation}\label{cippalippa3}
H^{(2q-2)}(u)=
\sum^\ast_{|\gamma|\leq q,\al,\bt}\ 
\left[
\sum_{\delta\succeq \gamma, |\delta|=q}\ 
\sum_{m\succeq \delta}
\binom{m}{\delta}
 \binom{\delta}{\gamma}
(-1)^{|\delta-\gamma|}
I^{m-\gamma}
 H_{m,\al,\bt}
 \right]
 |u|^{2\gamma}
 u^\al \bar u^\bt \,,
 \qquad
 \forall\, 
 |u|_{p,s,a}\leq r
\,,
\end{equation}
which is in the form \eqref{H disgiunta}.
Then, by \eqref{normatris2}, we can evaluate
the norm of $H^{(2q-2)}$ obtaining
\begin{eqnarray*}
&&|H^{(2q-2)} |_{\kappa_* r,s,\eta}
\\
&\leq&
\frac12 \sup_j\sum^\ast_{|\gamma|\leq q,\al,\bt}\ 
\sum_{\delta\succeq \gamma, |\delta|=q}\ 
\sum_{m\succeq \delta}
\binom{m}{\delta}
 \binom{\delta}{\gamma}
I^{m-\gamma}
| H_{m,\al,\bt}|
 (2\gamma_j+\al_j+\bt_j)
(\kappa_* u_0)^{2\gamma+\al+\bt-2e_j}
 \peseta
\\
&\leq&
 (\kappa_*r)^{-2} C_{q,\kappa_*}\,,
 \end{eqnarray*}
 using that $m\succeq \gamma$
 and \eqref{cippalippa2}.
 Then by \eqref{falcone}
$$
	|H^{(2q-2)} |_{\kappa_* r,s,\eta}
	\leq 
	\kappa_*^{-2}
	\pa{1+\frac{\kappa_*^2}{\kappa^2}}^{{q}}
	c_\kappa^q
	|H|_{r,s,\eta}\,.
$$
Since the Lipschitz estimate is analogous
we get \eqref{cacioepepe}.

We now prove that 
$\Pi^d \Pi^d H= \Pi^d H$ and
	$\Pi^{d'} \Pi^d H=\Pi^d \Pi^{d'} H=0$.
Setting $\tilde H(u):=\Pi^{2q-2} H(u)=H^{(2q-2)}(u)$
and recalling \eqref{coppiette},
by \eqref{cippalippalippa} we get as in \eqref{coppiette}
with $H\rightsquigarrow \tH$,
$$
\tilde \tH(u,w)
=
\sum_{|\delta|=q}
\pa{w - \II}^\delta
\sum^\ast_{m\succeq\delta,\al,\bt}
\binom{m}{\delta} H_{m,\al,\bt}
I^{m-\delta} u^\al \bar u^\bt\,.
$$
Noting that
$
\frac{1}{\delta'!} \partial_w^{\delta'} \sum_{|\delta|=q}
\pa{w - \II}^\delta
$
evaluated at $w=I$ is equal to 1 if $|\delta'|=q$
and vanishes if $|\delta'|\neq q$ we conclude by
the definition of the projections in \eqref{accad}
(recall also \eqref{sublimeporta}).

\medskip

Now we note that \eqref{persiani} directly follows by the representation formula
\eqref{rappre} with $q=1+\mathtt q$.

\medskip

Finally we use again  \eqref{rappre} in order to prove \eqref{cicladi}.
	By linearity it suffices to work on monomials.  
Namely let us assume that $F= (\abs{u}^{2} - \II)^{\delta^{(1)}}\abs{u}^{2k^{(1)}}\uuu$ and 
$G = (\abs{u}^{2} - \II)^{\delta^{(2)}}\abs{u}^{2k^{(2)}}\uud$
with 
$|\delta^{(i)}|\geq 1+\frac12 d_i$, for $i=1,2.$
	Then by Leibniz rule we have
	\begin{align*}
	H:=\set{F, G} = &\set{\pa{\abs{u}^{2} - \II}^{\delta^{(1)}}, \abs{u}^{2 k^{(2)}}\uud }\pa{\abs{u}^{2} - \II}^{\delta^{(2)}}\abs{u}^{2k^{(1)}}\uuu +\\
	& \set{\abs{u}^{2k^{(1)}}\uuu, \pa{\abs{u}^2 - \II}^{\delta^{(2)}}}\pa{\abs{u}^2 - \II}^{\delta^{(1)}}\abs{u}^{2 k^{(2)}}\uud +\\
	& \set{\abs{u}^{2 k^{(1)}}\uuu,\abs{u}^{2 k^{(2)}}\uud}\pa{\abs{u}^2 - \II}^{\delta^{(1)}}\pa{\abs{u}^2 - \II}^{\delta^{(2)}}.
	\end{align*}
	We can write $H$ in the form \eqref{rappresaglia}
	with $q=|\delta^{(1)}| + |\delta^{(2)}|-1\geq 1+\frac12 (d_1+d_2).$
	Then $H^{< d_1+d_2}=0$, proving \eqref{cicladi}.
	
	\smallskip
	The proof of Proposition \ref{proiettotutto} is now completed.
	\end{proof}



\begin{proof}[Proof of Proposition \ref{burger}]
First we note that the case $q=0$ is trivial since $\Pi^{\geq -2}=\id$ and formula \eqref{rappre}
reduces to \eqref{H disgiunta}.
Then we consider now $q\geq 1.$
\\
	 Let us introduce the auxiliary function, defined for
	 $(u,w)$ in the interior of
	 $
	  {\bar B}_{r^2}(\tw^\infty_{2p,2s,2a})\times {\bar B}_r(\tw_s), 
	 $
	$$
	F(u,w):=\tH(u,w)-\sum_{i=0}^{q-1}
	\frac{1}{i!}D_w^i \tH(u,I) [\underbrace{w- I,\dots,w-I}_i]
	=
	\int_0^1
	\frac{(1-t)^{q-1}}{(q-1)!}
	D^{q}_w \tH\big(u,I+t(|u|^2-I)\big)[\underbrace{w- I,\dots,w-I}_{q}]\, dt
	$$
	by Taylor's formula, where $I_w(t):=I+t(w-I)$.
	Note that by \eqref{accad} 
	$$
	F(u,|u|^2)=\tH(u,|u|^2)- \sum_{i=0}^{q-1}H^{(2i-2)}(u)
	= \Pi^{\geq 2q-2}H(u)
	\,, $$
	So, for $q\geq 1$, we get the representation formula 
	$$
	H^{\geq 2q-2}(u)
	=
	\int_0^1
	\frac{(1-t)^{q-1}}{(q-1)!}
	D^{q}_w \tH\big(u,I+t(|u|^2-I)\big)
	[\underbrace{|u|^2- I,\dots,|u|^2-I}_{q}]\, dt\,.
	$$
	now
	\[
	\frac{1}{(q-1)!}	
	D^{q}_w \tH(u,w)[h,\dots,h]
	=
	 q\sum_{|\delta|=q} \frac{1}{\delta!}\partial_w^\delta \tH(u,w)  h^\delta
	 \stackrel{\eqref{H disgiunta}}=
	  q\sum_{|\delta|=q} h^\delta
\sum^\ast_{\al,\bt, m\succeq\delta} 
\prod_i \binom{m_i}{\delta_i} H_{m,\al,\bt} w^{m-\delta}u^\al \bar u^\bt	\]
so in conclusion
\begin{align}
H^{\geq 2q-2}(u) 
&= q\sum_{|\delta|=q} (|u|^2-I)^\delta
\sum^\ast_{\al,\bt, m\succeq\delta} \prod_i \binom{m_i}{\delta_i} 
H_{m,\al,\bt} u^\al \bar u^\bt \int_0^1(1-t)^{q-1}(t|u|^2 +(1-t)I)^{m-\delta} dt
\nonumber
\\
& = q\sum_{|\delta|=q} (|u|^2-I)^\delta
\sum^\ast_{\al,\bt, m\succeq\delta} \sum_{k\preceq m-\delta}\prod_i \binom{m_i}{\delta_i}  \binom{m_i-\delta_i}{k_i}  H_{m,\al,\bt} u^\al \bar u^\bt |u|^{2k}I^{m-\delta-k} \int_0^1(1-t)^{|m|-|k|-1}t^{|k| }dt
\nonumber
\\
& = q\sum_{|\delta|=q} (|u|^2-I)^\delta
\sum^\ast_{k,\al,\bt} \sum_{m\succeq\delta+k }
\prod_i \binom{m_i}{\delta_i}  \binom{m_i-\delta_i}{k_i}  H_{m,\al,\bt} u^\al \bar u^\bt |u|^{2k}I^{m-\delta-k} 
\frac{|k|! (|m|-|k|-1)!}{|m|!}\,,
\label{engaddi}
\end{align}
where we have exchanged the order of summation in $m$ and $k$
and used that
\begin{equation}\label{commagene}
\int_0^1(1-t)^{|m|-|k|-1}t^{|k| }dt = \frac{|k|! (|m|-|k|-1)!}{|m|!}
\end{equation}
  The capability  to exchange the order of summation in \eqref{engaddi}
  and in the rest of the proof
   follows by the absolute convergence of the series, which we will
   prove below.
   \\
We now set for  $\delta,k,\al,\bt\in\N^\Z$
\begin{equation}\label{nefertiti}
\check{H}_{\delta,k,\al,\bt}
:=
q \sum_{m\succeq\delta+k }
\binom{m}{\delta}  \binom{m-\delta}{k}
H_{m,\al,\bt} I^{m-(\delta+k)} \frac{|k|! (|m|-|k|-1)!}{|m|!}
\end{equation}
and we claim that
the above series absolutely converges. 
We also claim that
\begin{equation}\label{aida}
\check{H}_\delta(u):= \sum^\ast_{k,\al,\bt} \check{H}_{\delta,k,\al,\bt} |u|^{2k} u^\al \bar u^\bt
\qquad
{\rm  totally\ converges\ for}\ \ 
|u|_s\leq \kappa_* r
\end{equation}
for every $\kappa_*<1.$
 Then \eqref{rappre} follows by \eqref{nefertiti}, \eqref{aida} and \eqref{engaddi}.
Developing 
$(|u|^2-I)^\delta=\sum_{\g\preceq\delta}\binom{\delta}{\gamma}|u|^{2\g}(-I)^{\delta-\g}$
and exchanging again the order of summation
we get
\begin{eqnarray}
H^{\geq 2q-2} (u)
&=&  \sum_{|\delta|=q} 
\sum_{\g\preceq\delta}\binom{\delta}{\gamma}|u|^{2\g}(-I)^{\delta-\g}
\sum^\ast_{k,\al,\bt} \check{H}_{\delta,k,\al,\bt} |u|^{2k} u^\al \bar u^\bt
\nonumber
\\
&=&
\sum^\ast_{\mathtt m,\al,\bt} 
\bigg(
\sum_{|\delta|=q} \sum_{\substack{\g\preceq\delta\\ \g\preceq \mathtt m}}
\binom{\delta}{\gamma}
(-I)^{\delta-\g}\check{H}_{\delta,\mathtt m-\g,\al,\bt}
\bigg)
 |u|^{2\mathtt m}u^\al \bar u^\bt
 \label{tirinto}
 \\
 &=&
 \sum^\ast_{k,\al,\bt} 
\bigg(
\sum_{|\delta|=q} \sum_{\g\preceq\delta}
\binom{\delta}{\gamma}
(-I)^{\delta-\g}\check{H}_{\delta,k,\al,\bt}
\bigg)
 |u|^{2\g+2k}u^\al \bar u^\bt
 \,.
 \nonumber
\end{eqnarray}
We now claim that the last series totally  converges on every ball  $|u|_s\leq\kappa_* r< r,$
in particular that
\begin{eqnarray}
&&\sum^\ast_{k,\al,\bt} 
\bigg(
\sum_{|\delta|=q} \sum_{\g\preceq\delta}
\binom{\delta}{\gamma}
I^{\delta-\g}
 |\check{H}_{\delta,k,\al,\bt}|
\bigg)
(\kappa_* u_0)^{2\g+2k + \al + \bt}
\stackrel{\eqref{nefertiti}}\leq
\nonumber
 \\
&&
\sum^\ast_{k,\al,\bt} 
\bigg(
\sum_{|\delta|=q} \sum_{\g\preceq\delta}
\binom{\delta}{\gamma}
q \sum_{m\succeq\delta+k }
 \binom{m}{\delta}  \binom{m-\delta}{k}
|H_{m,\al,\bt}| I^{m-\g-k} \frac{|k|! (|m|-|k|-1)!}{|m|!}
\bigg)
(\kappa_* u_0)^{2\g+2k + \al + \bt}
\stackrel{\eqref{cerere}}\leq
\nonumber
 \\
 &&q\sum^\ast_{m,\al,\bt} |H_{m,\al,\bt}|
 u_0^{2m+\al+\bt} \kappa^{2|m|}\kappa_*^{\al+\bt}
\sum_{|\delta|=q} \binom{m}{\delta} 
 \sum_{\g\preceq\delta}
\binom{\delta}{\gamma}
\pa{\frac{\kappa_*^2}{\kappa^2}}^{|\g|}
 \sum_{k\preceq m-\delta }
 \binom{m-\delta}{k}
 \pa{\frac{\kappa_*^2}{\kappa^2}}^{|k|}
 \frac{|k|! (|m|-|k|-1)!}{|m|!}
 \nonumber
 \\
 &&
<\infty\,.
\nonumber
\\
\label{ctesifonte}
&&
\end{eqnarray}
Before proving  \eqref{ctesifonte}
we stress that it shows that the series in \eqref{engaddi} and \eqref{tirinto}
absolutely converge and that $\check{H}_{\delta,k,\al,\bt}$ 
in \eqref{nefertiti} is well defined since the series absolutely converges.
Moreover since the first line of \eqref{ctesifonte} is bounded we get \eqref{aida}.

 Note that the first two inequalities in \eqref{ctesifonte} are 
 straightforward  and we have to prove only the last one, which follows by the auxiliary estimate
 \begin{eqnarray}
A_{q,\kappa,\kappa_*}
&:=&
\frac q{2} \sup_j \sum^\ast_{m,\al,\bt} e^{\eta|\pi(\al-\bt)|}|H_{m,\al,\bt}| (2m_j+\al_j+\bt_j) u_0^{2m +\al+\bt-2e_j}
\kappa^{2|m|}
\nonumber\\
&&\qquad\qquad
\sum_{|\delta|=q} \binom{m}{\delta} 
 \sum_{\g\preceq\delta}
\binom{\delta}{\gamma}
\pa{\frac{\kappa_*^2}{\kappa^2}}^{|\g|} 
\sum_{k\preceq m-\delta} 
\binom{m-\delta}{k} 
\pa{\frac{\kappa_*}{\kappa}}^{2|k|}  \frac{|k|! (|m|-|k|-1)!}{|m|!}
\nonumber
\\
&\le&
\pa{\frac{\kappa^2+\kappa_*^2}{\kappa_*^2}c_{\kappa_*}}^{q}
| H|_{ r,s,\eta}\,.
 \label{alessandria}
\end{eqnarray}
Indeed in the series in  \eqref{ctesifonte} the index $m$
cannot be zero, since $|m|\geq |\delta|=q\geq 1,$
and, therefore, the term $2m_j+\al_j+\bt_j$ in \eqref{alessandria}
is always greater or equal than 2 for some $j$.
Then   \eqref{alessandria} implies \eqref{ctesifonte}.
\\
Let us now prove \eqref{alessandria}. 
 For $|\delta|=q$ we have
\begin{eqnarray*}
&&\sum_{k\preceq m-\delta} 
\binom{m-\delta}{k} 
\pa{\frac{\kappa_*}{\kappa}}^{2|k|}  \frac{|k|! (|m|-|k|-1)!}{|m|!}
\stackrel{\eqref{commagene}}=
\int_0^1   
\left(
\sum_{k\preceq m-\delta} 
\binom{m-\delta}{k} 
\pa{ \frac{t\kappa_*^2}{\kappa^2}}^{k}(1-t)^{m-\delta-k} \right)
(1-t)^{|\delta|-1} \ dt
\\
&&=
\int_0^1   \left(\frac{ t \kappa_*^2}{\kappa^2} +(1-t) \right)^{|m|-|\delta|}
(1-t)^{|\delta|-1} \ dt
=
\kappa^{2q-2|m|}
\int_0^1   (\kappa^2 + t (\kappa_*^2 -\kappa^2) )^{|m|-q}
(1-t)^{q-1} \ dt
\end{eqnarray*}
Moreover integrating by parts we get
\begin{eqnarray*}
 0
 &\leq&
 \int_0^1   (\kappa^2 + t (\kappa_*^2 -\kappa^2) )^{|m|-q}
(1-t)^{q-1} \ dt
\\
&=&
\frac{1}{q} \kappa^{2|m|-2q}
+
\frac{(\kappa_*^2 -\kappa^2)(|m|-q)}{q}
\int_0^1   (\kappa^2 + t (\kappa_*^2 -\kappa^2) )^{|m|-q-1}
(1-t)^{q} \ dt\,.
\end{eqnarray*}
Since for $|m|\geq q$ 
\begin{eqnarray*}
0
&\leq&
(|m|-q)
\int_0^1   (\kappa^2 + t (\kappa_*^2 -\kappa^2) )^{|m|-q-1}
(1-t)^{q} \ dt
\\
&\leq&
(|m|-q)
\int_0^1   (\kappa^2 + t (\kappa_*^2 -\kappa^2) )^{|m|-q-1}
 \ dt
 \\
 &=&
 \frac{\kappa_*^{2|m|-2q}-\kappa^{2|m|-2q}}{\kappa_*^2 -\kappa^2}
 \end{eqnarray*}
 we get
 $$
 \int_0^1   (\kappa^2 + t (\kappa_*^2 -\kappa^2) )^{|m|-q}
(1-t)^{q-1} \ dt
\leq 
\frac{1}{q} \kappa_*^{2|m|-2q}\,.
 $$
Then in conclusion we get
\begin{equation}
\label{serve}
\sum_{k\preceq m-\delta} 
\binom{m-\delta}{k} 
\pa{\frac{\kappa_*}{\kappa}}^{2|k|}  \frac{|k|! (|m|-|k|-1)!}{|m|!}
\leq 
\frac{1}{q} \pa{\frac{\kappa_*}{\kappa}}^{2|m|-2q}\,.
\end{equation}
Recalling the definition of $A_{q,\kappa,\kappa_*}$ we get
 \eqref{alessandria}
 \begin{eqnarray}
A_{q,\kappa,\kappa_*}
&\leq&
\frac 1{2} 
\pa{\frac{\kappa}{\kappa_*}}^{2q}
\sup_j \sum^\ast_{m,\al,\bt} e^{\eta|\pi(\al-\bt)|}|H_{m,\al,\bt}| (2m_j+\al_j+\bt_j) u_0^{2m +\al+\bt-2e_j}
\kappa_*^{2|m|}
\nonumber\\
&&\qquad\qquad
\sum_{\substack{|\delta|=q\\ \delta\preceq m}} \binom{m}{\delta} 
 \sum_{\g\preceq\delta}
\binom{\delta}{\gamma}
\pa{\frac{\kappa_*^2}{\kappa^2}}^{|\g|} 
\nonumber
\\
&\stackrel{\eqref{ponto}}=&
\frac 1{2} 
\pa{1+\frac{\kappa^2}{\kappa_*^2}}^{q}
\sup_j \sum^\ast_{m,\al,\bt} e^{\eta|\pi(\al-\bt)|}|H_{m,\al,\bt}| (2m_j+\al_j+\bt_j) u_0^{2m +\al+\bt-2e_j}
\kappa_*^{2|m|}
\sum_{\substack{|\delta|=q\\ \delta\preceq m}} \binom{m}{\delta} \,.
\end{eqnarray}
Then \eqref{alessandria} follows  by \eqref{fornaio} and \eqref{normatris2}.

 Let us now prove \eqref{sparta}.
 By \eqref{normatris2} and \eqref{tirinto} we get 
\begin{eqnarray*}
&&\abs{H^{\ge 2q-2}}_{\kappa_* r,s,\eta} 
\\
&&
\leq
\frac12  \sup_j 
\sum^\ast_{\mathtt m,\al,\bt} 
\bigg(
\sum_{|\delta|=q} \sum_{\substack{\g\preceq\delta\\ \g\preceq \mathtt m}}
\binom{\delta}{\gamma}
I^{\delta-\g}
 |\check{H}_{\delta,\mathtt m-\g,\al,\bt}|
\bigg)
\pa{2\mathtt m_j + \al_j + \bt_j}
(\kappa_* u_0)^{2\mathtt m + \al + \bt - 2e_j}\peseta
\\
&&
=
\frac12  \sup_j 
\sum^\ast_{k,\al,\bt} 
\sum_{|\delta|=q} \sum_{\g\preceq\delta}\binom{\delta}{\gamma}
I^{\delta-\g}\ |\check{H}_{\delta,k,\al,\bt}|
\pa{2 k_j+ 2 \g_j + \al_j + \bt_j}
(\kappa_* u_0)^{2(k+\g) + \al + \bt - 2e_j}\peseta
\\
&&
\stackrel{\eqref{nefertiti}}\leq
\frac{q}2  \sup_j 
\sum^\ast_{k,\al,\bt} 
\sum_{|\delta|=q} \sum_{\g\preceq\delta}
\sum_{m\succeq\delta+k }
\pa{2 k_j+ 2 \g_j + \al_j + \bt_j}
(\kappa_* u_0)^{2(k+\g) + \al + \bt - 2e_j}\peseta
\\
&&
\quad
 \binom{m}{\delta}  \binom{m-\delta}{k}
\binom{\delta}{\gamma}
|H_{m,\al,\bt}| I^{m-\g-k} \frac{|k|! (|m|-|k|-1)!}{|m|!}
\\
&&
\stackrel{\eqref{cerere}}\leq
\frac{q}2  \sup_j 
\sum^\ast_{k,\al,\bt} 
\sum_{|\delta|=q} \sum_{\g\preceq\delta}
\sum_{m\succeq\delta+k }
\pa{2 k_j+ 2 \g_j + \al_j + \bt_j}
\pa{\frac{\kappa_*}{\kappa}}^{2(|k|+|\g|)}
\kappa_*^{|\al|+|\bt|-2}
\kappa^{2|m|}
 u_0^{2m + \al + \bt - 2e_j}\peseta
\\
&&
\quad
 \binom{m}{\delta}  \binom{m-\delta}{k}
\binom{\delta}{\gamma}
|H_{m,\al,\bt}| \frac{|k|! (|m|-|k|-1)!}{|m|!}\,.
\end{eqnarray*}
Since all the term are positive 
we can exchange the order of summation
(of $k$ and $m$) obtaining
\begin{eqnarray}
&&\abs{H^{\ge 2q-2}}_{\kappa_* r,s,\eta} 
\nonumber
\\
&&
\leq
\frac{q}2  \sup_j 
\sum^\ast_{m,\al,\bt} 
\sum_{|\delta|=q} \sum_{\g\preceq\delta}
\sum_{k\preceq m-\delta }
\pa{2 k_j+ 2 \g_j + \al_j + \bt_j}
\pa{\frac{\kappa_*}{\kappa}}^{2(|k|+|\g|)}
\kappa_*^{|\al|+|\bt|-2}
\kappa^{2|m|}
 u_0^{2m + \al + \bt - 2e_j}\peseta
\nonumber\\
&&
\qquad\qquad
 \binom{m}{\delta}  \binom{m-\delta}{k}
\binom{\delta}{\gamma}
|H_{m,\al,\bt}|  \frac{|k|! (|m|-|k|-1)!}{|m|!}
 \nonumber
 \\
 &&
 \leq
  \frac 1{\kappa_*^2}
 A_{q,\kappa,\kappa_*}
 \stackrel{\eqref{alessandria}}\leq
 \frac 1{\kappa_*^2} \pa{\frac{\kappa^2+\kappa_*^2}{\kappa_*^2}c_{\kappa_*}}^{q}
| H|_{ r,s,\eta}\,,
\end{eqnarray}
using $\kappa_*^{|\al|+|\bt|-2}\leq \kappa_*^{-2}.$
Since the Lipschiz estimate is analogous we get \eqref{sparta}.
\end{proof}

Proposition \ref{burger} has a partial viceversa.
\begin{lemma}\label{tenedo}
Assume that  $H\in\Heta$ can be written as a totally convergent series for $|u|_s\leq r'$
with $\kappa r<r'<r$ ($\kappa$ introduced in \eqref{polenta2}), of the form
 \begin{equation}\label{rappresaglia}
H(u) = \sum_{|\delta|=q} (|u|^2-I)^\delta\sum^\ast_{k,\al,\bt} H_{\delta,k,\al,\bt} |u|^{2k} u^\al \bar u^\bt\,,\qquad
\mbox{for}\ \ |u|_s\leq r'\,,
\end{equation}
then
$$
H^{\leq 2q-4}=0\,.
$$
\end{lemma}
\begin{proof}
 Recalling \eqref{coppiette} we have
 $$
 \tH(u,w) = \sum_{|\delta|=q} (w-I)^\delta\sum^\ast_{k,\al,\bt} 
 H_{\delta,k,\al,\bt} w^{k} u^\al \bar u^\bt\,.
 $$
 Then
 $$
 D^{q'}\tH(u,I)=0\,,\qquad
 \forall\, q'<q\,.
 $$
 We conclude by \eqref{accad}.
\end{proof}

\medskip
\begin{rmk} Proposition \ref{burger} shows a connection between the  order at the origin  $u=0$ and on the torus
$\cT_I$. Indeed if we consider a polynomial function
\[
H(u)= \sum_{|\bal|=|\bbt|\le N} H_{\bal,\bbt}\buu
\]
then one has 
\[
\Pi^{\ge 2N} H=0\,.
\]
Conversely if we consider a Hamiltonian of the form
\[
H(u)= \sum_{|\bal|=|\bbt|> N} H_{\bal,\bbt}\buu\,,
\]
then, given a parameter $\kappa_0$ with
$$
\kappa_*<\kappa_0<\min\{ 1, \kappa_*/\kappa\}\,,
$$
by \eqref{sparta} (applied with $\kappa_*\to\kappa_0$,
$\kappa\to\kappa_1:=\kappa \kappa_0/\kappa_*$
and $r\to r_1:=\kappa_* r/\kappa_0$ so that
\eqref{polenta2} reads
 $|\sqrt{I}|_{p,s,a}\leq \kappa r=\kappa_1r_1
 $
) and
by \eqref{cippi} 
we get
\begin{equation}\label{spartaco}
\|\Pi^{\ge 2q-2} H\|_{\kappa_* r,s,\eta}
\le 
\frac{1}{\kappa_0^2}
\pa{\frac{\kappa^2+\kappa_*^2}{\kappa_*^2}c_{\kappa_0}}^{q}
\pa{\frac{\kappa_*}{\kappa_0}}^{2N}
 \| H\|_{ r,s,\eta}\,.
\end{equation}
\end{rmk}

One can easily verify that the projections $\Pi^d$ commute with the projections $\Pi^\cK$ and $\Pi^\cR$ defined in  \eqref{ragno} so that we may define
\begin{equation}\label{poppea}
\cH^{d,\cR}_{r,s,\eta}:= \set{H\in \cH_{r,s, \eta}\,:\quad \Pi^{d,\cR}H:=\Pi^d\Pi^\cR H = H}\,.
\end{equation}
Similarly for $\cH^{\geq d,\cR}_{r,s,\eta},$ $\cH^{\geq d,\cK}_{r,s,\eta}$, etc.
Note that $\Heta^{d}=\Heta^{d, \cK}\oplus\Heta^{d, \cR}.$

A particularly important space is $\Heta^{0,\cK}$ wich, as can be seen in the following Lemma, is identified isometrically with $\ell^\infty$.

Given $\lambda=(\lambda_i)_{i\in\Z}\in\ell^\infty$
possibly depending in a Lipschitz way on the parameter
$\omega\in\cO$ we set\footnote{Recall \eqref{delta}.}

\begin{equation}\label{norma lambda}
\norma{\lambda}_\infty^{\cO,\mu}:= 
\sup_{\cO}
|\lambda|_{\ell^\infty}+
\mu 
\sup_{\o\neq \o'}
|\Delta_{\omega,\omega'} \lambda|_{\ell^\infty}\,.
\end{equation}
and as usual set $\norma{\lambda}_\infty= \norma{\lambda}_\infty^{\dgp,\g}$.
\begin{lemma}[Norms in $\Heta^{0, \cK}$]
	\label{lemma contro}
	Every $H\in \Heta^{0, \cK}$  has the form 
	\begin{equation}\label{silvestre}
H = \sum_{j\in \Z} \lambda_j( |u_j|^2 -I_j)	\,, \quad \lambda=\pa{\lambda_j}\in \ell_\infty :\quad \nar{H}=  \norma{\lambda}_\infty.
	\end{equation}
\end{lemma}
\begin{proof}
	We start by noting that  for $H$ as in \eqref{H disgiunta}
	\[
	\pon H = \sum_{m}
	H_{m,0,0}\sum_{i\in\Z} m_i \II^{m - e_i}\pa{\abs{u_i}^2 - \II_i}\,.
	\]
	so $ \pon H =H$ means that 
	$H_{m,\al,\bt}=0$ if $(m,\al,\bt) \ne (0,0,0)$ or $(e_j,0,0)$;  moreover all the $H_{e_j,0,0}$ are free real parameters and finally 
	$H_{0,0,0} = -\sum_{j}
	H_{e_j,0,0}  \II_j$. So we set $\lambda_j= H_{e_j,0,0}$ and obtain the representation in \eqref{silvestre}.
	Regarding the estimate we have
	\begin{align*}
	\abs{H}_{r,s,\eta} 
	&= \frac12 \sup_j 
	\sum_\ast \abs{\coef{H}}
	\pa{2m_j + \al_j + \bt_j}u_0^{2m + \al + \bt -2 e_j} \peseta 
	\\
	& = 
	\frac12 \sup_j 
	\sum_i \abs{\lambda_i} 2 \delta\pa{i,j} u_0^{2 e_i - 2 e_j} 
	=   \sup_j \abs{\lambda_j}=|\lambda|_{\ell^\infty}\,.
	\end{align*}
	The Lipschitz estimate is analogus.
\end{proof}

{\bf Notation}.
Lemma \ref{lemma contro} shows 
that, for every $r,s,\eta$ the space $\Heta^{0,\cK}$
is isometrically identified by  \eqref{silvestre} with $\ell^\infty,$ namely
\begin{equation}\label{grevity}
\mbox{for } H \mbox{as in \eqref{silvestre} we write }
H\in\ell^\infty\ \ \mbox{and}\ \  
\| H\|_\infty:=\|\lambda\|_\infty\,.
\end{equation} 
	Such Hamiltonians   will be referred to as counterterms.
	
	\begin{lemma}\label{bea}
	If $H\in \Heta$ then $\|\Pi^{-2} H\|_{r,s,\eta}\leq    \|H\|_{r,s,\eta}.$
	Moreover
if $\kappa\leq 1/\sqrt 2$  we have
\begin{equation}\label{ritornello}
\|\Pi^0 H\|_{r,s,\eta}\,,\ \| \Pi^{0,\cK} H\|_\infty \leq  3  \|H\|_{r,s,\eta}
\,, \qquad
	 \|\Pi^{\leq 2} H\|_{r,s,\eta}  \leq  4  \|H\|_{r,s,\eta}\,,\qquad
	 \|\Pi^{\geq 2} H\|_{r,s,\eta}  \leq  5  \|H\|_{r,s,\eta}\,.
\end{equation}
	\end{lemma}
	\begin{proof}
By 	\eqref{cacioepepe}
with $\kappa_*=1$, 
\eqref{tegolino} and \eqref{numeroni}.
\end{proof}

We shall consider also the Hamiltonian
\begin{equation}\label{didone}
D_\omega:= \sum_{j\in \Z} \omega_j |u_j|^2\qquad {\rm with}\qquad\omega\in\dgp\,.
\end{equation}
Note that such Hamiltonian belongs to $ \cA^0_r(\tw_s)$ for all $r,s$, see \eqref{ao}, but   is not in $\Heta$.  We define
\[
\Pi^{-2,\cK} D_\omega=\sum_{j\in \Z} \omega_j I_j\,,\qquad
\pon  D_\omega = \sum_{j\in \Z} \omega_j (|u_j|^2 -I_j)\,,
\qquad  \Pi^{-2,\cR} D_\omega= \por  D_\omega=\Pi^{\ge 2}  D_\omega=0\,.
\]
thus extending the projections to the affine space $D_\omega+\Heta$.
We conclude this section with the following Lemma which can be duduced directly from Proposition \ref{proiettotutto}.
\begin{lemma}\label{gradi}
	Given $F,G\in \Heta$ such that $\Pi^{\leq 0}G=0$ then
	\begin{equation}\label{mina}
	\pd \{F,G\}=0\,,
	\end{equation}
	similarly if also $\pd F=0$  then
	\begin{equation}\label{scondo}
	\Pi^0 \{F,G\}=0\,,
	\end{equation}
finally 
	\begin{equation}\label{homo-gradi}
	\Pi^{d}\set{D_\omega,F}= \{D_\omega,F^{(d)}\}\,.
	\end{equation}
\end{lemma}


\section{Normal forms, invariant tori and nearby dynamics}
Here we discuss the relation between the projections of the previous section and the dynamics on the torus $\cT_I$.
The first trivial remark is that if
\[
H= D_\omega + P\,,\quad P\in \Heta^{\ge 0}
\]
then $\cT_I$ is an invariant manifold for the dynamics.
Indeed, by Proposition \ref{burger} one has
\[
H= \sum_j (|u_j|^2 -I_j) \hat H_{j}(u)\,,
\]
and a direct computation ensures that $d_t(|u_j|^2)=0 $ on $\cT_I$.
Recall the denitition \ref{NCO},
of course if  $N$ is in normal form at $\cT_I$ then the torus is invariant and the $N$-flow
is linear on $\cT_I$,  with frequency $\omega.$
Indeed, for all $j\in\Z$ one has
\[
u_j(t) = u_j(0) e^{\im \omega_j t}\,,\quad |u_j(0)|=\sqrt{I_j}\,.
\] 
is an almost-periodic solution for the $N$-flow.
\\
In the appendix we show that if $$\rho:=\inf_{j\in\Z}I_j  \jap{j}^{2p}e^{2a\abs{j}+2s \jap{j}^\teta}>0$$
then one can  introduce action-angle variables around the torus
$\cT_I$.  
The action-angle map is a diffeomorphism from a neighborhood 
of $\cT_I$ into $\{|J-I|_{2p,2s,2a}<\rho/2\}\times \T^\Z,$
where  $\T^\Z$ is a differential manifold modelled on 
$\ell^\infty$. We shall call such an object a {\it maximal torus}.

\smallskip

Now  for $N\in \cN_{r, s ,\eta}(\omega,I)$ we discuss conditions which entail   orbital stability at $\cT_I$, i.e. that the trajectory of  initial data  $\delta$- close to the torus stays $\delta$-close to the torus for times of order $\delta^{-d}$.
To make this statement precise let us fix $r,s$ and $\sqrt{I}\in {\bar B}_{\kappa r}(\tw_s)$. 
Then  we define the {\it annulus}\footnote{Where 
$\sqrt{|u|^2-I}$ is defined componentwise $(\sqrt{||u|^2-I|})_j:=\sqrt{||u_j|^2-I_j|}$.}
\begin{equation}\label{piccoloano}
{\bf A}_\delta:=\Big\{ u\in {\bar B}_r(\tw_s):\quad
\Big|\sqrt{\big||u|^2-I\big|}\Big|_s<\delta r
\Big \}
\qquad
{\rm for} \quad 0<\delta < \sqrt{1-\kappa^2}\,.
\end{equation}
Now in order to measure the variation of the actions we can, in the case of a maximal torus, pass to action angle variables and compute  the action component of the Hamiltonian vector field on $A_\delta$. In order to give a result which holds also in the case of non-maximal tori we 
use the following 
\begin{lemma}\label{suzy}
Fix $r,s>0$, $0<\kappa<1,$ $d\geq 2$, $C\geq 1$  and $\sqrt{I}\in {\bar B}_{\kappa r}(\tw_s)$.	
Then there exist $\delta_0>0, T_0>1$ such that 
if
$$
N= D_\omega + P\,, \quad {\sup_{u\in  \fA_{2\delta}}|X_P|_{s}<\infty\,,}
\qquad
	\frac{1}{\delta^2 r^2} \sup_{u\in  \fA_{2\delta}} \left|\Big(\{P,|u_j|^2\}\Big)_{j\in\Z}\right|_{2p,2s,2a} 
\le  C\delta^d\,,\quad
\forall\, 0<\delta\leq \delta_0\,,
$$ 
 then for any initial datum $u(0)\in \fA_\delta$ the  $N$-flow is well defined and remains in $\fA_{2\delta}$ for all times $|t|\le  T_0 \delta^{-d}.$ 
\end{lemma}
\begin{proof} 
The first two conditions ensure that the Hamilton equations of $N$ are al least locally well-posed in $\tw_s$. The last condition is an {\it energy estimate}. Indeed
$
\frac{d}{dt }|u_j|^2 = \{N, |u_j|^2\} = \{P,|u_j|^2\}
$
hence as long as the flow stays in $\fA_{2\delta}$
\[
\frac{d}{dt } \Big|{\big||u|^2-I\big|}\Big|_{{2p,2s,2a}} \le \sup_{u\in  \fA_{2\delta}} 
 \left|\Big(\{P,|u_j|^2\}\Big)_{j\in\Z}\right|_{2p,2s,2a} \le  r^2 \delta^{d+2}  C
\]
so the proof follows by standard contraction arguments.
\end{proof}
We now prove  that  a Hamiltonian in $\cH^{\ge d}$  when restricted to  $\fA_{\delta}$  leaves the actions approximately invariant.
\begin{lemma}\label{corinna}
	Given $H\in \cH^{\ge d}_{r,s,\eta}$ and 
	$\kappa<\kappa_\sharp<1$, if $0<\delta<\sqrt{\kappa_\sharp^2-\kappa^2}$
	one has
	\[
	\frac{1}{\delta^2 r^2} \sup_{u\in  \fA_\delta} \left|\Big(\{H,|u_j|^2\}\Big)_{j\in\Z}\right|_{2p,2s,2a} 
	\leq 
	2\delta^{d }
 c_{\kappa_\sharp}^{\frac d2 +1}
\kappa_\sharp^{-d+2}|H|_{r,s,0}\,,
	\]
	moreover
		\[
	\frac{1}{\delta r} \sup_{u\in  \fA_\delta} |X_H|_{s} 
	\leq 
	\delta^{d -1}
	c_{\kappa_\sharp}^{\frac d2 +1}
	\kappa_\sharp^{-d+1}|H|_{r,s,0} 
	\]
\end{lemma}
\begin{proof}
First we note that by \eqref{giancarlo} and \eqref{polenta2}
$\sqrt{I_j}\leq	\kappa u_{0j}(r)=\kappa u_{0j}$ and 
\begin{equation}\label{sileant}
u\in  \fA_\delta\qquad
\Longrightarrow\qquad
\sqrt{\big| |u_j|^2-I_j\big|}
\leq \delta u_{0j}
\,,\qquad
|u_j|\leq \kappa_\sharp u_{0j}\,.
\end{equation}
	Setting $2q -2= d$,   and using the representation formula \eqref{rappre}
	and the Leibniz rule, we get for every $j\in\Z$
	$$
	\{H(u),|u_j|^2\}=\sum_{|h|=q} (|u|^2-I)^h \{\check H_{h}(u), |u_j|^2\}\,.
	$$
	Then for $u\in  \fA_\delta$ we have
	\begin{align*}
&\frac{1}{\delta^2 r^2} \sup_{u\in  \fA_\delta} \left|\Big(\{H,|u_j|^2\}\Big)_{j\in\Z}\right|_{2p,2s,2a} 
\\
&= \sup_j \frac{1}{\delta^2 u_{0j}^2} \left|\sum_{|h|=q} (|u|^2-I)^h \{\check H_{h}(u), |u_j|^2\}\right|
\\
&= \sup_j \frac{1}{\delta^2 u_{0j}^2} 
\left|
 \sum_{|h|=q} \sum^\ast_{k,\al,\bt}\check H_{h ,k,\al,\bt} (|u|^2-I)^{h}   |u|^{2k} u^\al \bar u^{\bt} (2k_j +\al_j +\beta_j )
 \right|
  \\ & 
  \stackrel{\eqref{sileant}}\le 
  \delta^{d }
\sup_j  \sum_{|h|=q} \sum^\ast_{k,\al,\bt}|\check H_{h ,k,\al,\bt}|  u_0^{2h + 2 k +\al + \bt -2e_j}    
\kappa_\sharp^{2|k|+|\al|+|\bt|}
(2k_j +\al_j +\beta_j )
 \\ & 
   \stackrel{\eqref{nefertiti}}\le 
 q \delta^{d }
\sup_j  \sum_{|h|=q} \sum^\ast_{k,\al,\bt}
\sum_{m\succeq h+k }
\binom{m}{h}  \binom{m-h}{k}
|H_{m,\al,\bt}| \kappa^{2(|m|-|h|-|k|)} \frac{|k|! (|m|-|k|-1)!}{|m|!}
\\
&\qquad\qquad\qquad\qquad
 u_0^{2m +\al + \bt -2e_j}    
\kappa_\sharp^{2|k|+|\al|+|\bt|}
(2k_j + \beta_j+ \alpha_j )
\\
& 
\le 
 q \delta^{d }
 \kappa^{-2q}
\sup_j  \sum^\ast_{m,\al,\bt}
|H_{m,\al,\bt}|
 \kappa^{2|m|}
  u_0^{2m +\al + \bt -2e_j} 
  (2m_j + \beta_j+ \alpha_j )
\\
&\qquad\qquad\qquad\qquad   
\sum_{|h|=q}
\sum_{k \preceq m -h} 
\binom{m}{h}  \binom{m-h}{k}
\left(\frac{\kappa_\sharp}{\kappa}\right)^{2|k|} \frac{|k|! (|m|-|k|-1)!}{|m|!}   \ .
\end{align*}
By \eqref{serve} and \eqref{fornaio}
\[
\sum_{|h|=q}
\sum_{k \preceq m -h} 
\binom{m}{h}  \binom{m-h}{k}
\left(\frac{\kappa_\sharp}{\kappa}\right)^{2|k|} \frac{|k|! (|m|-|k|-1)!}{|m|!}\le 
\frac1q\sum_{|h|=q, h\preceq m} \binom{m}{h} \pa{\frac{\kappa_\sharp}{\kappa}}^{2|m|-2q}  
\le 
\frac1q   c_{\kappa_\sharp}^q
\kappa_\sharp^{-2q}\kappa^{2q-2|m|}
\]
we get in conclusion 
\begin{eqnarray*}
\frac{1}{\delta^2 r^2} \sup_{u\in  \fA_\delta} \left|\Big(\{H,|u_j|^2\}\Big)_{j\in\Z}\right|_{2p,2s,2a} 
& \le &
\delta^{d }
 c_{\kappa_\sharp}^q
\kappa_\sharp^{-2q}
\sup_j  \sum^\ast_{m,\al,\bt}
|H_{m,\al,\bt}|
  u_0^{2m +\al + \bt -2e_j} 
  (2m_j + \beta_j+ \alpha_j )
  \\
&\stackrel{\eqref{normatris2}}=&
2\delta^{d }
 c_{\kappa_\sharp}^q
\kappa_\sharp^{-2q}|H|_{r,s,0} \,,
\end{eqnarray*}
proving the  first bound in the lemma. The second bound follows exactly in the same way
\end{proof}

\subsection{Stability close to invariant tori.}\label{seziostabile}
In this subsection we rederive  the stability results of \cite{Yuan_et_al:2017} in our setting.
Actually our result is a bit stronger since we do not need any smallness assumption on the non-linear part of our Hamiltonian.
\begin{prop}
	Fix $r,s,\eta>0$, $0<\kappa<1,$ $d\geq 2$,  and $\sqrt{I}\in {\bar B}_{\kappa r}(\tw_s)$.	For any $N\in \cN_{r, s ,\eta}(\omega,I)$, there exists 
 $\delta_0>0, T_0>1$ such that  for $0<\delta<\delta_0$, calling $\Phi_N^t(u)$ the flow of $N$ at time $t$ and with initial datum $u$, one has
 \[
 u\in \fA_\delta \quad \Rightarrow\quad  \Phi_N^t(u)\in \fA_{2\delta}\,,\quad \forall |t|\le T_0 \delta^{-d}\,.
 \]
	\end{prop}
	
\begin{proof}
 The point is that, for any $d\ge 2$ and  given any normal form $N\in \cN_{s,r,\eta}$,  one can prove the existsence of a symplectic change of variables $\Phi_d:\fA_\delta \to \fA_{2\delta}$  (for all sufficiently thin annular domains $\fA_{\delta}$, namely
$\delta$ small)   that conjugates $N$ to $N_d = D_\omega +  P_d$, where $P_d$ satisfies the hypotheses of Lemma \ref{suzy}.
Note that $\Phi_d$ is NOT defined  on the whole ball ${\bar B}_r(\tw_s)$ unless we assume further restriction on  $\kappa$ (rapidly decreasing to 0 as $d$ increases).

Let us briefly  discuss how to prove such a result.
\\
We start with a normal form Hamitonian 
\[
N= D_\omega + P \,,\quad P=P^{\ge 2}
\]
such that $P\in \cH_{r, s, \eta}$.  
Assume that $0<\s<\eta$ and fix
\begin{equation}\label{ringo}
0<\rho\leq \frac14(1-\kappa) r\,.
\end{equation}
We claim that there exists  $S\in \cH^{\ge 2,\cR}_{r-\rho,s+\s,\eta-\s}$   such that
\[
\Pi^{\le d,\cR}  \sum_{k=0}^{d/2} \frac{\ad_S^k}{k!} N =0
\,.
\]
To prove our claim we  decompose $S= \sum_{h=1}^{d/2} S^{2h}$ where  we fix $\s_h =\frac{ 2 h}{ d}\s $, $\rho_h = \frac{ 2 h}{ d} \rho$ and assume $S^{2h} \in \cH^{2h,\cR }_{r-\rho_h,s+\s_h,\eta-\s_h}$. 
The functions $S^{2h}$ are defined recursively as 
\[
\{D_\omega,S^{2h} \}=  \Pi^{2h,\cR} (\sum_{k=2}^{h} \frac{\ad_{S^{< 2h }}^{k-1}}{k!} \{S^{<2h},D_\omega\}  + \sum_{k=0}^{h-1} \frac{\ad_{S^{< 2h }}^k}{k!} P) \,,\quad 
S^{< 2h }:= \sum_{j=1}^{h-1} S^{2j}\,.
\]
By Lemma \ref{Lieder} and Proposition \ref{fan} we get the recursive estimate 
\[
|\{D_\omega, S^{2h}\}|_{r-\rho_h,s+\s_{h-1},\eta-\s_{h-1} }\le   C_h |P|_{r,s,\eta}^h\,,\quad\quad |S^{2h}|_{r-\rho_h,s+\s_h,\eta-\s_h} \le   C_h |P|_{r,s,\eta}^h \,,
\]
which implies
\[
|\{D_\omega,S\}|_{r-\rho,s+\s,\eta-\s}, |S|_{r-\rho,s+\s,\eta-\s} \le   
C \max\{|P|_{r,s,\eta},\ |P|_{r,s,\eta}^d\}\,,
\]
where the constants $C_h, C$  depend on $r/\rho,\s,\theta$ and $d$;
moreover the dependence on $d$ is superexponentially large.
Note that $S$ belongs to $\cR^{\ge 2}_{r-\rho,s+\s,\eta-\s}$ 
but it is not necessarily  small in ${\bar B}_{r-\rho}(\tw_s)$. 
Of course, by Lemma \ref{corinna} and since $S\in \cH^{\ge 2}$,  if we take a sufficienly thin annulus   $\fA_\delta$, we fall under the hypotheses of Lemma \ref{suzy}.
In particular, recalling \eqref{ringo}, we apply Lemma \ref{corinna} with
$$
r\rightsquigarrow r-2\rho\,,\qquad
\kappa\rightsquigarrow \frac{\kappa r}{r-2\rho}\,,\qquad
\delta \rightsquigarrow \frac{2\delta r}{r-2\rho}\,,\qquad
\kappa_\sharp \rightsquigarrow \frac{r-2\rho}{r}\,.
$$
Then assuming $\delta$ small enough
not only the time one flow of $S$ is well defined and generates a symplectic change of variables  $\Phi_d:\fA_\delta \to \fA_{2\delta}$ but 
the Lie series expansion $e^{\{S,\cdot\}} N$ is totally convergent. Indeed one has
\[
\frac{\ad_S^k}{k!} N = \frac{\ad_S^{k-1}}{k!} \{S,D_\omega\}+\frac{\ad_S^k}{k!} P
\]
and by standard computations (see for instance Lemma 2.1 of \cite{BMP1:2018}) one has
\[
|L_k|_{r-2\rho,s+\s,\eta-\s}:= |\frac{\ad_S^{k-1}}{k!} \{S,D_\omega\}|_{r-2\rho,s+\s,\eta-\s} \le \pa{ \frac{8 e |S|_{r-\rho,s+\s,\eta-\s}}\rho }^{k-1} |\{S,D_\omega\}|_{r-\rho,s+\s,\eta-\s}
\]
hence by Lemma \ref{corinna} 
\[
\sup_{u\in  \fA_{2\delta}} |X_{L_k}| \le   r 
c_{\kappa_\sharp}^2 \kappa_\sharp^3 \delta^2 \pa{  \frac{8 e c_{\kappa_\sharp} \kappa_\sharp^2 \delta^2 |S|_{r-\rho,s+\s,\eta-\s}}\rho }^{k-1} |\{S,D_\omega\}|_{r-\rho,s+\s,\eta-\s}\,.
\]
The same kind of bound holds  for $\frac{\ad_S^k}{k!} P$.
 Then for all $u\in \fA_{2\delta}$
\[ 
N\circ \Phi_d = e^{\{S,\cdot\}} N =   \Pi^{\cK} (\sum_{k=0}^{d/2} \frac{\ad_S^k}{k!} N)+ \Pi^{\cR,> 2d} (\sum_{k=0}^{d/2} \frac{\ad_S^k}{k!} N) + \sum_{k=d/2+1}^\infty \frac{\ad_S^k}{k!} N
\]
where the first two terms are finite sums and hence analytic in ${\bar B}_{r-\rho}(\tw_s)$ while the series on the right hand side is totally convergent.
Then one may apply Lemmata \ref{corinna} and \ref{suzy} and the stability estimates follow.
\end{proof}


\section{Proof of Theorems \ref{torello} and \ref{allaMoserbis}}

\begin{rmk}

In this Section  we shall fix $\g$ and always assume that $\omega\in \dgp$ and that Hamiltonian functions depend on $\omega$ in a Lipschitz way. 
Hence, for ease of notations, we shall denote the norm \eqref{normag} with
$\mu=\gamma$ and
$\cO=\dgp$ as
\begin{equation}
\label{doppia barretta}
\norma{\cdot}_{r,s,\eta} :=  \norma{\cdot}^{\g,\,\dgp}_{r,s,\eta}.
\end{equation}
\end{rmk}

We start this section by proving Theorem \ref{torello} through a direct application to the NLS Hamiltonian of the  normal form theorem \ref{allaMoserbis}, and the elimitation of the counter terms via the $(V_j)$.

\begin{proof}[Proof of Theorem \ref{torello}]
%
Recalling the notations of Theorem \ref{torello}
we fix
\begin{equation}\label{polpettedibollito}
r_0:=2\sqrt 2 r\,,\quad
\rho:=r_0-2 r\,,\quad
\eta_0:=\frac{\ta-a}{2}\,,\quad
\s:=\frac12\min\{s,\eta_0,2\}\,,\quad
s_0:=s-\s\,,
\end{equation}
 and choose $I = (I_j)_{j\in\Z}$ such that $\sqrt{I}\in {\bar B}_{r}(\tw_{s})$. For $\omega\in \dg$, let us  write the Hamiltonian $H_{NLS}$  as 
\begin{equation}\label{NLS herman}
H_{NLS} = \sum_{j\in \Z} (j^2+V_j) |u_j|^2 +P = D_\omega +\Lambda + P'\,,
\end{equation}
where we set
\begin{equation}\label{counter NLS}
\Lambda = \sum_{j\in \Z} \lambda_j (|u_j|^2 - I_j), \qquad \lambda_j = j^2-\omega_j+V_j,\qquad P' = P + \sum_{j\in\Z} \lambda_j I_j.
\end{equation}
Of course $P'$ satisfies \eqref{stimazero}.
By the definition of $V_j$ and the fact that  $\dg\subset\pan$ (recall definitions \eqref{vu} and \eqref{Omega}) we have that $\Lambda\in \ell^\infty$.
 \\ In order to apply Theorem \ref{allaMoserbis} to  $H = N_0 + P'$ where $N_0 = D_\omega$ we fix 
\begin{equation}
\label{epstar}
\e_*:= \frac{\bar{\epsilon}}{8 C_{\rm alg}^2(p) C(p,s_0,\eta_0)}
\end{equation}
where $s_0,\eta_0$ are defined in \eqref{polpettedibollito} and $C_{\rm alg}(p),C(p,s_0,\eta_0)$ are the ones in Proposition \ref{neminchione}. With this choice the smallness conditions \eqref{maipendanti} are satisfied, since in this case $\Theta=0$.
Finally, by setting
\[
V_j(\omega)= \Lambda_j(\omega) +\omega_j-j^2\,,
\]
we obtain \[H_{NLS}\circ\Psi = N.\]
Finally note that $V_j(\omega)$ is Lipschitz in $\omega$ since $\Lambda(\omega)$ is so.
\end{proof}
The proof of Theorem \ref{allaMoserbis} follows a  quadratic KAM scheme, which consists in constructing the solutions of equation \eqref{coniugio} iteratively, by linearizing the problem at each step and solving the equation as a linear one (homological equation) plus a reminder, which is proved to converge to $0$. 
More precisely we apply the following iterative procedure.

 Fix
 $r_0,s_0,\eta_0, \rho, r,\s$ as in \eqref{newton}
  and let $\{\rho_n\}_{n\in\N}, \{\s_n\}_{n\in\N}$ be  the  summable  sequences:
 	\begin{equation}\label{amaroni}
 	\rho_n= \frac{\rho}{{4}} 2^{-n}\,,\qquad \sigma_0 = \frac{\s}{8}, \quad \s_n = {\frac{9\s}{4\pi^2 n^2}}\quad \forall n\ge 1\,.
 \end{equation} 
   Let us define recursively
   \begin{eqnarray}\label{pesto}
 &&r_{n+1} = r_n - \red{2}\rho_n\ \to \ r_\infty:=r_0-\rho\qquad   {\rm (decr
 easing)} \nonumber\\
 &&s_{n+1} = s_n + \red{2}\sigma_n\ \to \ s_\infty:=
 s_0+\sigma\qquad  {\rm (increasing)} \nonumber \\
  &&\eta_{n+1} = \eta_n - \red{2}\sigma_n\ \to\ \eta_\infty:=\eta_0 - \s\qquad {\rm (decreasing)} 
\end{eqnarray}
Note that for every $ r'\geq r_\infty\,, s'\leq s_\infty$
\begin{equation}\label{cumdederit}
\sqrt{I}\in {\bar B}_{r}(\tw_{s_\infty})
\qquad
\stackrel{\eqref{newton}}\Longrightarrow\qquad
\sqrt{I}\in 
{\bar B}_{\frac{1}{\sqrt 2}\frac{r_0}{2}}(\tw_{s_\infty})
\subset
{\bar B}_{\frac{1}{\sqrt 2}r_\infty}(\tw_{s_\infty})
\subset {\bar B}_{\frac{1}{\sqrt 2}r'}(\tw_{s'})
\end{equation}
and thatl the projections $\Pi^d$ are well defined  
on every space
$\cH_{r',s',\eta'}$.
Moreover by \eqref{ritornello}
\begin{equation}\label{fiesta}
H\in \cH_{r',s',\eta'}\quad
\mbox{with}\quad
r'\geq r_\infty\,,\  s'\leq s_\infty\,,\ \eta'\geq 0\,,\qquad
\Longrightarrow
\qquad
\|\Pi^{0,\cK}H\|_\infty
\leq 
3 \|H\|_{r',s',\eta'}\,.
\end{equation}

Let  
\begin{equation}\label{bisanzio}
H_0 := D_\omega + G_0 + \Lambda_0\,,
\qquad
G_0\in \cH_{r_0,s_0,\eta_0}\,,\qquad
\Lambda_0\in\ell^\infty\,,
\end{equation}
(recall \eqref{grevity})
 where 
the counterterms $\Lambda_0$ 
 are free parameters. 
 We define 
\begin{equation}\label{vigili}
 \e_0:=\gamma^{-1}\pa{\norma{\zeroK{G_0}}_\infty + \norma{\zeroR{G_0}}_{r_0, s_0, \eta_0} + \norma{\due{G_0}}_{r_0, s_0, \eta_0}} ,  \quad \Theta_0:= \gamma^{-1}\norma{\buon{G_0}}_{r_0, s_0, \eta_0} +\e_0 
 \end{equation}

\begin{lemma}[Iterative step]\label{iterativo}
Let\footnote{Recall also that $0<\theta<1$ was fixed one and for all.} $r_0,s_0,\eta_0, \rho, r,\s$ be  as in \eqref{newton},
  $\rho_n, \s_n, r_n, s_n, \eta_n$  as in \eqref{amaroni}-\eqref{pesto},
  $H_0,G_0,\Lambda_0$ as in \eqref{bisanzio}
  and
$\eps_0,\Theta_0$  as in \eqref{vigili}.
Let $\sqrt{I}\in {\bar B}_{r}(\tw_{s_\infty})$.
 There exists  a constant $\frak C>1$ large enough
 such that
if 
\begin{equation}\label{gianna}
\eps_0 	\leq \pa{1 + \Theta_0}^{-3} \tK^{-2}\,,
\qquad 
\tK
:= \frak C  \pa{\frac{r_0}{\rho}}^4\sup_n 2^{4n}e^{2\crac n^{6/\theta}} e^{-\chi^n (2-\chi) } \,,
\qquad
\crac := 2\pa{\frac{4\pi^2}{9\s}}^{\frac3\teta} \croc
\end{equation}
($\croc$ defined in Lemma \ref{Lieder})
then we can iteratively construct a sequence of generating functions 
$S_i = \due{S_i} + \zero{S_i}\in\cH_{r_{i}-\rho_i, s_{i+1}, \eta_{i+1}}$ 
with\footnote{Namely $\due{S_i}$ is more regular since
$\cH_{r_{i}, s_{i}+\s_i, \eta_{i}-\s_i}\subset\cH_{r_{i}-\rho_i, s_{i+1}, \eta_{i+1}}$}
$\due{S_i}\in\cH_{r_{i}, s_{i}+\s_i, \eta_{i}-\s_i}$
and a sequence of  counterterms  $\bar\Lambda_i\in\ell^\infty$ such that the following holds, for $n\ge 0$.

$(1)$ For all $ i = 0,\ldots,  n -1 $ and any $s_0+\eta_0\ge s'\ge s_{i+1}$ the 
time-1 hamiltonian flow
 $\Phi_{S_i}$ generated by $S_i$   satisfies
	\begin{equation}
\sup_{u\in  {\bar B}_{r_{i+1}}(\tw_{s'})} \norm{\Phi_{S_i}(u)- u}_{s'} \le \rho 2^{-2i-7} \,\label{ln}
 \end{equation}
 Moreover
	\begin{equation}\label{ucazzo}
	\Psi_n := \Phi_{S_0}\circ\cdots \circ \Phi_{S_{n-1}} 
	\end{equation}
	is a well defined, analytic map ${\bar B}_{r_n }(\tw_{s'}) \to {\bar B}_{r_0}(\tw_{s'})$ for all $s_n\le s' \le s_0+\eta_0 $ with the bound
	\begin{equation}
	\label{cosi}
	 \sup_{u\in {\bar B}_{r_n}(\tw_{s'})}\abs{\Psi_{n}(u) - \Psi_{n-1}(u)}  \le  \rho 2^{-2n + 2}.
	\end{equation}

	$(2)$ We set $\cL_0:=0$  and for $i=1,\dots,n$ 
	 $$ \mathcal{L}_{i} + \id := e^{\set{S_{i-1},\cdot}}\pa{\mathcal{L}_{i-1} + \id},\quad \Lambda_{i} := \Lambda_{i-1} - \bar{\Lambda}_{i-1}\,,\quad H_i= e^{\{S_{i-1},\cdot\}}H_{i-1}$$ where $\Lambda_{i-1}$ are free parameters and $\mathcal{L}_{i} : \ell^\infty\to\cH_{r_{i}, s_{i}, \eta_{i}}$ 
are linear operators.  
	 We have
	\begin{equation}
\label{cioccolato}
H_{i} = D_\omega + G_{i} + \pa{\id + \mathcal{L}_{i}}\Lambda_{i},\qquad G_{i}, \in\cH_{r_{i}, s_{i}, \eta_{i}}.
\end{equation}
 Setting for $ i = 0,\ldots, n  $
	\begin{equation}\label{xhx-i}
	 \eps_i:=\gamma^{-1}\pa{\norma{\zeroK{G_i}}_\infty + \norma{\zeroR{G_i}}_{r_i, s_i, \eta_i} + \norma{\due{G_i}}_{r_i, s_i, \eta_i}},  \quad \Theta_i:=\gamma^{-1}\norma{\buon{G_i}}_{r_i,s_i,\eta_i} +\e_i \,,
	\end{equation}
we have
\begin{eqnarray}
& \e_i \leq   \e_0  e^{- \chi^{i}+1} \,, 
\qquad
\chi:=3/2\,,\qquad
\qquad  \Theta_i \leq   \Theta_0 \sum_{j=0}^i 2^{-j}\, \label{en} \\
& 
 \label{fringe}
\norma{\pa{\call_{i} -\call_{i-1}} h}_{r_i,s_i,\eta_i}\le \tK \eps_0 \pa{1 + \Theta_0} 2^{-i} \norma{h}_\infty,\qquad  \norma{\call_i h}_{r_i,s_i,\eta_i} \le \tK (1+  \Theta_0)\e_0\sum_{j=1}^i 2^{-j}\norma{h}_\infty,
\end{eqnarray}
for all $h\in\ell^\infty.$
Finally the counter-terms satisfy the bound
 \begin{equation}
 \norma{\bar{\Lambda}_{i-1}}_\infty \le  \g \tK \e_{i-1}(1+\Theta_0)\,,\quad i=1,\dots,n  \label{lambdonebarra iterstima}
	\end{equation}
\end{lemma}
By \eqref{cosi}, \eqref{lambdonebarra iterstima} and \eqref{fringe} we get 
\begin{cor}\label{convergo}  $\Psi:=\lim_{n\to \infty}\pa{\Psi_n}_{n}$
is well defined as a map from ${\bar B}_{r_\infty}\pa{\tw_{s_\infty}}$
to $ {\bar B}_{r_0}\pa{\tw_{s_\infty}}$.  Moreover the sequence $\bar\Lambda_n$ is summable. Finally the sequence $\call_n$ converges to an operator $\call:\ell^\infty \in
 \cH_{r_\infty,s_\infty,\eta_\infty}.$ 
  \end{cor}

Let us now prove Theorem \ref{allaMoserbis} by applying  Lemma \ref{iterativo} .
\begin{proof}[Proof of Theorem \ref{allaMoserbis}]
Take
\begin{equation}\label{zampadecane}
\bar\epsilon:= 2^{-15}\tK^{-2}\,,\qquad
\bar C:=2^7 \tK\,,
\end{equation}
with $\tK$ defined in \eqref{gianna}.
Recalling  that $H-D_\omega\in \cH_{r_0,s_0,\eta_0}$ let us set 
$$
H_0 := D_\omega + G_0 + \Lambda_0\,,
\qquad
G_0 := H - D_\omega := \pa{N_0 - D_\omega} + G\,,
$$
	where $\Lambda_0\in\ell^\infty$ are free parameters and $\norma{G}_{r_0,s_0,\eta_0}= \norma{H-N_0}_{r_0,s_0,\eta_0}= \gamma \epsilon $  is small.
	Since \[G_0^{\le 0} =  \pa{ N_0-D_\omega +G}^{\le 0} = G^{\le 0}\,, 
	\] 
	 by \eqref{newton} and Lemma \ref{bea}  we have that
	 $$
	 \|\due{G_0}\|_{r_0, s_0, \eta_0} \le \norma{G}_{r_0,s_0,\eta_0},\qquad
	  \|\zeroK{G_0}\|_\infty\,,\ 
	  \|\zeroR{G_0}\|_{r_0,s_0,\eta_0} \le 3 \norma{G}_{r_0,s_0,\eta_0}
	 $$
	hence 
	\begin{equation}\label{epsilonzero}
	\e_0:=\gamma^{-1}\pa{\norma{\zeroK{G_0}}_\infty + \norma{\zeroR{G_0}}_{r_0, s_0, \eta_0} + \norma{\due{G_0}}_{r_0, s_0, \eta_0}}   \le 7\epsilon.
	\end{equation}
	Moreover, since 
	$\buon{G_0} = \buon{\pa{N_0 - D_\omega}} + \buon{G}
		, $
	 we also have that
	$$
	\Theta_0 :=  \gamma^{-1} \|\buon{G_0}\|_{r_0,s_0,\eta_0}+\e_0 
	\le 
	4\gamma^{-1}\pa{ \|N_0 - D_\omega\|_{r_0,s_0,\eta_0}+\|G\|_{r_0,s_0,\eta_0} }+\e_0 \\
	\leq  4\Theta+ 11\epsilon. 
	$$
By \eqref{zampadecane} the hypothesis \eqref{gianna}
of Lemma \ref{iterativo} is satisfied. 
Then by Lemma \ref{iterativo} and
  by corollary \ref{convergo} we pass to the limit in \eqref{cioccolato}  and obtain
	\[
	H_0\circ \Phi  = D_\omega + G_\infty + (\id +\call)(\Lambda_0-\sum_{i=0}^\infty \bar\Lambda_i)=:N\,.
	\]
	Now by formula \eqref{en} we have $N^{\le 2} = D_\omega$ provided that we fix 
	\[
	\Lambda_0=\Lambda=\sum_{i=0}^\infty \bar\Lambda_i\,,\]
	This concludes the proof.
\end{proof}

	\begin{proof}[Proof of the iterative Lemma \ref{iterativo}]
Throughout the Lemma we do not keep track of constant terms in the Hamiltonians $H_i$. Of course such terms do not contribute to the seminorm $\|\cdot \|_{r_i,s_i,\eta_i}$, moreover, by Proposition \ref{ham flow}, $H_i(0) =H_0(0)$ for all $i$.

In the following by $a\lesssim b$ we mean that there exists a positive constant $c$ (depending only on $\theta$, which is fixed)
such that $a\leq c b.$
\\
Moreover, we will repeatedly make use of Proposition \ref{fan}, Lemma \ref{Lieder} and Proposition \ref{proiettotutto}, this last one always and tacitly with $\kappa_* = 1$.

\subsection*{Inizialization}	Proving that Lemma \ref{iterativo} holds at $n=0$ is essentially tautological. Indeed item $(1)$ is empty, while item $(2)$ follows directly from the definitions: \eqref{cioccolato}  and \eqref{xhx-i}  coincide resp. with \eqref{bisanzio} and \eqref{vigili} while the bound \eqref{en} is trivial. Now, assuming that Lemma \ref{iterativo} holds up to $n \ge 0$, we verify that it holds also for $n+1$.

\subsection*{Proving the $n+1$ step.}

We first observe that, since \[
 (\sqrt{I_j})_{j\in\Z}\in {\bar B}_{r}(\tw_{s_0+\s}) \subset  {\bar B}_{r_n}(\tw_{s_n})\,, \quad \forall n
 \]
  by  Proposition \ref{proiettotutto}, the projections $\Pi^{0,\cK}$ and $\Pi^{0,\cR}$ are well defined and continuous throughout the iteration.	
Assume that
\[
H_n:= D_\omega +(\id +\call_n)\Lambda_n + G_n \,
\]
satisfies \eqref{en} and \eqref{fringe} with $i=n$.
We  fix the generating function $S_n = \due{S_n} + \zero{S_n}$ 
and the  { \sl counter-term} $\bar{\Lambda}_n\in \ell^\infty$ as the unique solutions of the {\sl homological equation}
\begin{equation}\label{homo enne}
\Pi^{\le 0} \pa{(\id +\call_n)\bar\Lambda_n + G_n + \set{{S_n},D_\omega+G_n^{\geq 2}}} = \dueK{G_n}\,.
\end{equation}
Solving this equation amounts to canceling the non-quadratic terms which prevent the torus $\cT_I$ to be invariant for the Hamiltonian $e^{\{S_n, \cdot\}} H_n $ (recall lemma \ref{gradi}).\\
Let us project \eqref{homo enne} on the three subspaces $\ell^\infty,\cH^{0,\cR}, \cH^{-2}$; by Lemma \ref{gradi} the equation \eqref{homo enne} splits into the following triangular system
\begin{align}
\label{coda-2}
& \dueR{G_n}+ \Pi^{-2,\cR} \mathcal{L}_n{\bar\Lambda_n} + \set{\due{S_n},D_\omega} = 0 \\
\label{codaN}
& \zeroK{G_n} + \pa{\id + \pon \mathcal{L}_n} \bar\Lambda_n + \pon \set{\due{S_n},G_n^{\geq 2}}=0\,, \\
\label{codaR}
& \zeroR{G_n}+ \por \mathcal{L}_n \bar{\Lambda}_n + \set{\zero{S_n},D_\omega}+\por \set{\due{S_n},G_n^{\geq 2}}
=0\,.
\end{align}
Let us first solve equation \eqref{coda-2} "modulo $\bar{\Lambda}_n$"; secondly we determine the counter term $\bar{\Lambda}_n$ in equation \eqref{codaN} and, eventually, solve equation \eqref{codaR}.
In what follows we repeatedly apply Lemma \ref{Lieder}, Proposition \ref{fan} and Proposition \ref{proiettotutto} in order to solve the (system of) homological equation \eqref{homo enne} and bound appropriately the solutions to obtain  \eqref{ln}-\eqref{lambdonebarra iterstima} for $i=n$.
 
\subsection*{Existence of $S_n$, $\bar\Lambda_n$ and corresponding bounds }
Let us start with \eqref{coda-2}, which gives
\begin{equation}\label{fisher}
\due{S_n} =L_\omega^{-1}\pa{\dueR{G_n}+ \Pi^{-2,\cR} \mathcal{L}_n{\bar\Lambda_n} }
\end{equation}
By substituting $\due{S_n}$ in \eqref{codaN}, we get 
\begin{equation}\label{inverto lambdone}
\pa{\id + \pon \mathcal{L}_n} \bar\Lambda_n + \pon \set{L^{-1}_\omega \pa{\due{G_n}+ \pd \mathcal{L}_n{\bar\Lambda_n}},G_n^{\geq 2}}= -  \zeroK{G_n}, 
\end{equation}
namely 
\begin{equation}\label{inverto operatorone}
\pa{\id + M_n}\bar\Lambda_n = -\zeroK{G_n} - \pon\set{L_\omega^{-1}\due{G_n}, \buon{G_n}},
\end{equation}
where  $M_n:\ell^\infty\to \ell^\infty$ is the operator defined by:
\begin{equation}\label{che schifo}
h \mapsto M_n \, h:={\pon \mathcal{L}_n} h + \pon \set{ \pd L^{-1}_\omega \mathcal{L}_n{h},G_n^{\geq 2}}\,.
\end{equation}
The following Lemma will be proved in Appendix \ref{chitelida}.
\begin{lemma}\label{camel}
$
\norma{M_n \, h }_{\infty} \le \frac12 \|h\|_{\infty}\,.
$
\end{lemma}

\noindent
Since the operator norm of $M_n$ is smaller than $\frac12$, then $\id + M_n$ 
is invertible  by Neumann series and its inverse is bounded by $2$, in operator norm.
We thus conclude that\footnote{Note that at the first step we have $M_0\equiv 0$, since $\call_0 = 0$ and equation \eqref{codaN} determines $\bar\Lambda_0$ trivially. } 
\begin{equation}\label{lambdonebarra enne}
\bar\Lambda_n = - \pa{\id + M_n}^{-1} \pa{ \zeroK{G_n} + \pon\set{L_\omega^{-1}\due{G_n}, \buon{G_n}}}
\,.
\end{equation}
By \eqref{xhx-i}, \eqref{ritornello}, \eqref{commXHK} and Lemma \eqref{Lieder}
we get 
 \begin{equation}\label{nutria}
 \begin{aligned}
 \norma{\bar\Lambda_n}_\infty & \le 2 \norma{\zeroK{G_n}}_\infty + 2\norma{\pon\set{L_\omega^{-1}\due{G_n}, \buon{G_n}} }_\infty
  \le 2\gamma\eps_n + 6\norma{\set{L_\omega^{-1}\due{G_n}, \buon{G_n}}}_{ \frac{r_n}{2},s_n+\eta_n,0} \\ &\le 2\gamma\eps_n +  2^7 \norma{L_\omega^{-1}\due{G_n}}_{  r_n,s_n+\eta_n,0}\norma{\buon{G_n}}_{ r_n,s_n+\eta_n,0}\\
 &\le 2\gamma\eps_n +  2^7\gamma e^{\croc {{{\eta_n}}}^{-3/\teta}} \e_n \Theta_n\le  2^8 e^{\croc {{{\eta_n}}}^{-3/\teta}} (1+\Theta_0)\g \eps_n \le \tK(1+\Theta_0) \g \eps_n
 \,,
 \end{aligned}
  \end{equation}
  where the last inequality follows by \eqref{gianna} and noting that $\eta_n\geq \eta_0-\s\geq \s=8\s_0.$
This proves \eqref{lambdonebarra iterstima}  for $i=n+1$.
We can thus bound the solution $\due{S_n}$ determined in \eqref{fisher} as 
\begin{align}
\norma{\due{S_n} }_{r_n,s_n+\s_n,\eta_n-\s_n}=&\g^{-1} e^{\croc{{\sigma_n}}^{-3/\teta}}\pa{\norma{\due{G_n}}_{r_n,s_n,\eta_n}+ \norma{\pd \mathcal{L}_n{\bar\Lambda_n} }_{r_n,s_n,\eta_n}
}
\nonumber
\\ & 
\leq
  \e_n e^{\croc{{\sigma_n}}^{-3/\teta}}(1+\tK^2 (1+\Theta_0)^2\e_0)\le  2\e_n e^{\croc{{\sigma_n}}^{-3/\teta}}
  \label{pantegana}
\end{align}
by \eqref{xhx-i},\eqref{fringe},\eqref{nutria},\eqref{gianna}.
Equation \eqref{codaR} determines $\zero{S_n}$
\[
\zero{S_n}=- L_\omega^{-1}\pa{\zeroR{G_n}+ \por \mathcal{L}_n \bar{\Lambda}_n +\por \set{\due{S_n},G_n^{\geq 2}}}
\]
which, by Lemma \ref{Lieder},\eqref{xhx-i}, \eqref{ritornello}, \eqref{fringe},
  \eqref{nutria}, \eqref{pantegana}, \eqref{commXHK} and \eqref{gianna} satisfies
\begin{align}
\norma{\zero{S_n}}_{r_n-\rho_n,s_{n+1},\eta_{n+1}} 
&\lesssim   e^{\croc \s_n^{-3/\teta}}\e_n
\Big(1+ \tK^2(1+\Theta_0)^2 \e_0 +  \frac{r_n}{\rho_n}e^{\croc{{\sigma_n}}^{-3/\teta}} \Theta_n \Big)
\nonumber
\\
&\lesssim 
\frac{r_n}{\rho_n} e^{2\croc \s_n^{-3/\teta}}\e_n(1+\Theta_0) \nonumber
\end{align}
Recalling \eqref{pantegana}, \eqref{amaroni}, \eqref{pesto} and using 
$\eps_n \le \eps_0 e^{-\chi^n + 1}$ (by the inductive hypothesis \eqref{en}) we get
\begin{equation}\label{gattasciocca}
\norma{S_n}_{r_n-\rho_n,s_{n+1},\eta_{n+1}} \leq
C \frac{r_0}{\rho}2^{n-10} e^{\crac{n}^{6/\teta}} e^{- \chi^{n}}\e_0( 1+ \Theta_0)
\leq 
C 2^{-2n-10} \e_0( 1+ \Theta_0) \sqrt\tK
\,,
\end{equation}
where  $C>1$ is a universal constant (only depending on the fix quantity $\theta$).

\subsection*{The maps $\Phi_{S_n}$ and $\Psi_n$}
 
For any $s_0+\eta_0\ge s' \ge s_{n+1} $, fixing $\eta' = s_0 + \eta_0 - s'$,  by the monotonicity entailed in \eqref{emiliaparanoica}  and by \eqref{gattasciocca},
we have that 
\begin{equation}
\label{pierino}
\norma{S_n}_{r_{n}- \rho_n, s', \eta'}\le \norma{S_n}_{r_n-\rho_n,s_{n+1},\eta_{n+1}} 
\le 
C 2^{-2n-10} \e_0( 1+ \Theta_0) \sqrt\tK\,,
\end{equation}
where $C$ is the constant in \eqref{gattasciocca}.
We wish to apply Proposition \ref{ham flow} with $\rho \rightsquigarrow \rho_n$ and $r \rightsquigarrow r_n - \rho_n$. Indeed
\begin{equation}\label{vomito}
\norma{S_n}_{r_{n}- \rho_n, s', \eta'} \le 
\frac{\rho}{r_0}
2^{-2n-10}
\stackrel{\eqref{amaroni},\eqref{pesto}}\leq
 \frac{\rho_n}{16 e (r_n - \rho_n)} 2^{-n-1}
\end{equation}
by \eqref{gianna}.
In turn, the bound \eqref{vomito} implies \eqref{stima generatrice}; then 
\eqref{ln}  for $i=n$ follows by \eqref{pollon} and \eqref{amaroni}.\eqref{newton}
\\
By \eqref{ucazzo} and the estimate \eqref{vomito} we have
\begin{equation}\label{omino pescatore}
\Psi_{n+1} := \Psi_n \circ \Phi_{S_n} = \Phi_{S_0} \circ \cdots \circ \Phi_{S_n} : {\bar B}_{r_{n+1} + \rho_n} (\tw_{s'}) \to {\bar B}_{r_0 - {\rho_0}}(\tw_{s'})
\end{equation}
(recall that $r_{n+1} +\rho_n= r_n - \rho_n$) for all $s_{n+1} \le s' \le s_0 + \eta_0$. Note that $\Phi^1_{S_{i-1}} \circ \Phi^1_{S_i}$ is well defined for all $i = 1,\ldots, n$ since ${\bar B}_{{r_i - \rho_i}}(\tw_{s'})\subset {\bar B}_{{r_{i-1} - 2\rho_{i-1}}}(\tw_{s'})$ .\\
Again Proposition \ref{ham flow}  with $s'= s_{n+1}$ implies that $H_{n+1} = e^{\set{S_n,\cdot}} H_n $ is well defined and $\eta_{n+1}$-majorant analytic (where $\eta_{n+1} = \eta_n - 2\sigma_n$, recall definition \eqref{amaroni}).\\
We now prove \eqref{cosi}.   In formula \eqref{ln} in Lemma \ref{iterativo}, we have proved that for any $n\ge 0$ 
  $$
  \Psi_{n} = \Phi_{S_0}\circ\cdots\circ\Phi_{S_{n-1}},\quad \Psi_n: {\bar B}_{r_n + \rho_{n-1}}(\tw_{s' })\to {\bar B}_{r_0}(\tw_{s' }).
  $$
Setting $u_t = (1-t) u + t \Phi_{S_n}(u), t\in [0,1]$ for  $u\in {\bar B}_{r_n-\rho_n}(\tw_{s'})$, we have
   $$ 
  \Psi_{n+1}(u) - \Psi_n(u) = \Psi_n\pa{\Phi_{S_n}(u)} - \Psi_n(u) = \int_0^1 \Psi'_n(u_t)
 \sq{\Phi_{S_n}(u) - u}\, dt\,. $$
  In the following, in order not to burden the notations we will ony indicate indexes of norms which undergo some variation as
$$
\abs{\cdot}_{r_n} := \sup_{u\in {\bar B}_{r_{n}}\pa{\tw_{s'}}} \abs{\cdot}_{s'}\,, \qquad \abs{\cdot}_{\operatorname{op},r_n} := \sup_{u\in {\bar B}_{r_{n}}\pa{\tw_{s'}}} \abs{\cdot}_{\operatorname{op}}\,, \qquad  \Phi_n:=\Phi_{S_n}\,.
$$  
  So, 
  \begin{align*}
 \abs{\Psi_{n+1} - \Psi_n}_{r_{n+1} } & \le \int_0^1  \sup_{u\in {\bar B}_{r_{n+1}}(\tw_{s'})}\abs{\Psi_n'(u_t)}_{\operatorname{op}}\, dt\, \abs{\Phi_{n} - \id}_{r_{n+1}}\\
 &\le \int_0^1  \abs{\Psi_n'}_{\operatorname{op}, r_{n}-\rho_n}\, dt\, \abs{\Phi_{n} - \id}_{r_{n+1}}\\
 \end{align*}
 By  the chain rule and Cauchy estimates, we have that
\begin{align*}
\abs{\Psi_n'}_{\operatorname{op}, r_{n}-\rho_n}&=\sup_{u\in {\bar B}_{r_n-\rho_n}(\tw_{s'})} \abs{\Phi'_0\circ\cdots\circ\Phi_{n-1}(u)\cdots \Phi_{n-2}'\circ\Phi_{n-1}(u)\cdot\Phi'_{n-1}(u)}_{\operatorname{op}}\\
&\le \Pi_{j=0}^{n-1} \abs{\Phi'_{j}}_{\operatorname{op},r_{j+1}-\rho_{j+1}} \le \Pi_{j=0}^{n-1}\pa{\frac{1}{\rho_{j+1}}\abs{\Phi_j - \id}_{r_{j+1}} + 1}
\stackrel{\eqref{ln}}\le \Pi_{j=0}^{n-1}(2^{-j-5}+1)
\le 2
\end{align*}
So, 
\begin{equation*}
 \abs{\Psi_{n+1}(u) - \Psi_n(u)}_{r_{n+1} } \le 2\abs{\Phi_n - \id}_{r_{n+1}} 
 \stackrel{\eqref{ln}}\le  \rho 2^{-2n - 6}.
\end{equation*}

\subsection*{Bound on $\call_{n+1} - \call_{n}$}
By  \eqref{fringe} for $i=n$ and \eqref{gianna}  we have
	$$
{\norma{\pa{\call_{{n}} + \id}h}_{r_{n+1},s_{n+1},\eta_{n+1}}}\le	\norma{\pa{\call_{{n}} + \id}h}_{r_n, s_n, \eta_n} \le 2\norma{h}_\infty\,.
	$$
Since by construction
	$$
\pa{\call_{n+1} - \call_n } = \pa{e^\set{S_{n},\cdot} - \id} \circ \pa{\call_n + \id}\,,
	$$
by \eqref{caio}, \eqref{gattasciocca} and \eqref{amaroni} we get
\begin{eqnarray*}
&&\norma{\pa{\call_{n+1} - \call_n }h}_{r_{n+1},s_{n+1},\eta_{n+1}}
\le {16 e}\frac{r_n - \rho_n}{\rho_n} \norma{S_{{n}}}_{r_{n}-\rho_n,s_{n+1},\eta_{n+1}}{\norma{\pa{\call_{{n}} + \id}h}_{r_{n+1},s_{n+1},\eta_{n+1}}} \nonumber\\
& &\leq 
C\frac{r_0}{\rho} 2^{-n} \e_0( 1+ \Theta_0) \sqrt\tK
\norma{h}_\infty
\leq 
2^{-n-1}\e_0( 1+ \Theta_0) \tK
\norma{h}_\infty
\end{eqnarray*}
proving 
the first bound in \eqref{fringe} for $i=n+1$, taking $\frak C$
large enough in \eqref{gianna}. 
\\
The second bound in \eqref{fringe}   follows directly from the first one.

	\subsection*{Bounds on $G^{\le 2}_{n+1}$ and $\buon{G_{n+1}}$}
By construction 
\[G_{n+1} = e^{\set{S_n,\cdot}} H_n - \sq{D_\omega + \pa{\id + \call_{n+1}}\Lambda_{n+1}}\,.
\]
Since $S_n$ solves the Homological equation \eqref{homo enne}, 
we have that
\begin{align}\label{schifo al cazzo}
G_{n+1} &= \dueK{G_n} + G_n^{\ge 2}+\Pi^{\ge 2}\pa{\call_{n+1} \bar{\Lambda}_n+\set{S_n,G_n^{\ge 2}}}  + G_{n+1,*}
	\\ G_{n+1,*} &=  \{S_n,G_n^{\le 0}\} + \Pi^{\le 0}\pa{\call_{n+1}-\call_n}\bar{\Lambda}_n+
	\pa{e^{\{S_n,\cdot\}} - \id - \set{S_n,\cdot}}
	G_n  \nonumber \\
	& - \sum_{h=2}^\infty \frac{ (\ad S_n)^{h-1}}{h!} \pa{\Pi^{\le 0}\pa{\id + \call_n}\bar\Lambda_n + G_n^{\le 0} +\Pi^{\le 0} \{\due{S_n},G_n^{\ge 2}\}} \nonumber.
\end{align}
Note  that $G_{n+1,*}$ is  quadratic in $S_n \sim G_n^{\le 0}$.\\
 \comment{
\begin{align*}
\due{G_{n+1}} &=  \Pi^{-2}  \pa{\{S_n,G_n^{\le 0}\} 
+ \pa{\call_{n+1}-\call_n}\bar{\Lambda}_n+
\pa{e^{\{S_n,\cdot\}} - \id - \set{S_n,\cdot}}
G_n} \\
&- \Pi^{-2}  \pa{\sum_{h=2}^\infty \frac{ (\ad{S_n})^{h-1}}{h!} \pa{\Pi^{\le 0}\pa{\id + \call_n}\bar\Lambda_n + G_n^{\le 0} +\Pi^{\le 0} \{S_n,G_n^{\ge 2}\}}} 
\end{align*}

\begin{align*}
\zeroR{G_{n+1}}&=   \Pi^{0,\cR}\pa{\{S_n,G_n^{\le 0}\} 
+ \pa{\call_{n+1}-\call_n}\bar{\Lambda}_n+
\pa{e^{\{S_n,\cdot\}} - \id - \set{S_n,\cdot}}
G_n } \\
&- \Pi^{0,\cR}\sum_{h=2}^\infty \frac{ (\ad{S_n})^{h-1}}{h!} \pa{\Pi^{\le 0}\pa{\id + \call_n}\bar\Lambda_n + G_n^{\le 0} +\Pi^{\le 0} \{S_n,G_n^{\ge 2}\}} 
\end{align*}

\begin{align*}
\zeroK{G_{n+1}}&=   \Pi^{0,\cK}\pa{\{S_n,G_n^{\le 0}\} 
	+ \pa{\call_{n+1}-\call_n}\bar{\Lambda}_n+
	\pa{e^{\{S_n,\cdot\}} - \id - \set{S_n,\cdot}}
	G_n } \\
&- \Pi^{0,\cK}\sum_{h=2}^\infty \frac{ (\ad{S_n})^{h-1}}{h!} \pa{\Pi^{\le 0}\pa{\id + \call_n}\bar\Lambda_n + G_n^{\le 0} +\Pi^{\le 0} \{S_n,G_n^{\ge 2}\}} 
\end{align*}
}
In order to prove \eqref{en} for $i=n+1$ we just need to  apply Proposition \ref{ham flow} and repeatedly use Lemmas \ref{fan}-\ref{proiettotutto}.
 In the following formula only the radius of analyticity changes, hence, for brevity, we omit to write
 $s_{n+1},\eta_{n+1}$ in  the 
 indexes of the norms. We have
\begin{align*}
\norma{{G_{n+1,*}}}_{r_{n+1}} 
& \lesssim 
 \frac{r_0}{\rho_n} \Big[\norma{S_n}_{r_n - \rho_n} \norma{G_n^{\leq 0}}_{r_n - \rho_n} + \norma{S_n}_{r_n - \rho_n} \norma{\bar\Lambda_n}_\infty
 +  {\frac{ r_0}{\rho_n} \norma{S_n}^2_{r_n - \rho_n}}\norma{G_n}_{r_n - \rho_n} + \\
& + \norma{S_n}_{r_n - \rho_n}\pa{\norma{\bar\Lambda_n}_\infty+ \norma{G_n^{\leq 0}}_{r_n - \rho_n} +\norma{S_n^{-2}}_{r_n } \norma{G_n^{\ge 2}}_{r_n}}\Big]
\end{align*}
Hence by \eqref{xhx-i}, \eqref{en}, \eqref{nutria}, \eqref{pantegana} and \eqref{gattasciocca}  we have
\begin{align*}
\g^{-1}
\norma{{G_{n+1,*}}}_{r_{n+1}} & 
\lesssim 
  \pa{\frac{r_0}{\rho}}^3 \tK(1+\Theta_0)^2\e_0 \e_n  2^{2n} e^{2\crac{n^{6/\theta}}} e^{-\chi^n}\,.
\end{align*}
Since the same estimate holds for $\|\due{G_{n+1}}\|_{r_{n+1}},
\| \zeroK{G_{n+1}}\|_{r_{n+1}}$ and $\|\zeroR{G_{n+1}}\|_{r_{n+1}}$, we get  
\[
\e_{n+1} 
\lesssim 
  \pa{\frac{r_0}{\rho}}^3 \tK(1+\Theta_0)^2\e_0 \e_n  2^{2n} e^{2\crac{n^{6/\theta}}} e^{-\chi^n}
\le
\tK^2(1+\Theta_0)^2\e_0^2   e^{-\chi^{n+1}}
\,. 
\]
Then the first estimate in \eqref{en} for 
$i=n+1$ follows by \eqref{gianna}.
With similar calculations we estimate $\norma{\buon{G_{n+1}}}_{r_{n+1}}$, obtaining
\[
\Theta_{n+1} -\Theta_n 
\lesssim \eps_0^2 (1 + \Theta_0)^2 e^{-\chi^n} +  \eps_0 (1 + \Theta_0) \Theta_0 2^{2n} e^{-\chi^n}e^{\crac n^{6/\teta} } +  2^{4n} e^{2\crac{n^{6/\theta}}} e^{-2\chi^n}\eps_0^2  (1 + \Theta_0)^3 .	
\]
Then, taking $\frak C$ large enough in \eqref{gianna} we get
$$
\Theta_{n+1} -\Theta_n 
\leq \Theta_0 2^{-n}\,,
$$
proving the second estimate in \eqref{en} for 
$i=n+1$.
\end{proof}




  \section{ Lower dimensional tori}\label{elliptic}
  As discussed in the previous section an advantage of our counterterm method is that it is uniform in the {\it dimension of the torus}, namely in the number of non-zero actions. Of course one should expect that in constructing a non-maximal torus one only needs to modulate the frequencies relative to the non-zero actions. In this section we make this statement precise and we discuss an application to the NLS.
  
  Let $\cS$ be {\sl any} subset of $\Z$ and
 denote $\ub:= \pa{\ub_j}_{j\in \cS}:=\pa{u_j}_{j\in \cS}$
and 
$\zb:=\pa{\zb_j}_{j\in \cS^c}:=\pa{u_j}_{j\in \cS^c}$.
With abuse of notation\footnote{Consisting in a reordering of the indexes $j$.}
we will write $u=(\ub,\zb).$
Analogously we will write 
$V=(\Vb,\Wb)$, $\lambda=(\bl,\bmu)$, $\omega=(\bo,\bO)$. In particular we consider
$I=(\Ib,0)$ with $I_j=\Ib_j>0$ for $j\in\cS.$

Similarly to  section \ref{provola}, given $H\in \Heta$, we expand  in Taylor series as
	\[
	H= \sum_{\substack{m,\al,\bt\in \N^\cS\\ \al\cap \bt= \emptyset\\ a,b\in \N^{\cS^c}}} H_{m,\al,\bt,a,b}|\ub|^{2m}\ub^\al\bar \ub^\bt \zb^a\bar \zb^b
	\]
	and define appropriate projection operators depending on the variables we are considering. More precisely:
	\begin{itemize}
		\item on the variables $\ub$ supported in $\cS$, we  use the projections defined in section \ref{provola}
		\item on the variables $\zb$ supported in $\cS^c$, we  use the projections on the terms of fixed homogeneous degree at $\zb=0$.
	\end{itemize}
	This gives rise to two degree indices; nevertheless we will not distinguish them but instead define just one degree as follows
	\begin{equation}
	H^{(d)} = \sum_{\substack{m,\al,\bt,\delta\in \N^\cS, a,b\in\N^{\cS^c}\\\al\cap \bt= \emptyset\,,\; \delta\preceq m \\ 2|\delta|+|a|+|b|= d+2} } {H}_{m,\al,\bt,a,b}
	\binom{m}{\delta} \Ib^{m-\delta} (|\ub|^2-\Ib)^\delta \ub^\al \bar \ub^\bt \zb^a \bar \zb^b\,.
	\end{equation}
	In this way, if $\cS=\Z$, projections coincide with the ones of section \ref{provola}, while if $\cS= \emptyset$, $H^{(d)}$ represents the usual homogeneous degree at $\zb=0$.\\ Note that in this definition the degree $d$ is always $\ge -2$ but now it is no more necessarily even.
	\smallskip
	In particular, for the projections onto $\Heta^{0,\cK}$, following the previous definitions, we shall set for any $H\in\Heta$
	\begin{equation}\label{pi kaduti}
	\Pi^{0,\cK} H = \Pi^{0,\cK_\cS} H+ \Pi^{0,\cK_{\cS^c}} H:= \sum_{m\neq 0} H_{m,0,0,0,0}\sum_{i\in\cS} \Ib^{m - e_i}m_i\pa{\abs{\ub_i}^2 - \Ib_i}+ \sum_{j\in \cS^c} {H}_{0,0,0,e_j,e_j}  |\zb_j|^{2}
	\end{equation}
	Of course, all the properties enjoyed by $\Pi^{0,\cK}$ and, in general by $\Pi^{(d)}$ entailed in Propositions \ref{proiettotutto}  and \ref{burger} hold also in this case.\\
	We point out that in this frame, the space $\ell^\infty$ of counter-terms (i.e. Hamiltonians $\Lambda$ such that $\Pi^{0,\cK} \Lambda = \Lambda$), now consists of elements of the form
	\begin{equation}\label{counter Kaduti}
	\Lambda =\sum_{j\in\cS}\bl_j (|\ub_j|^2-\Ib_j) + \sum_{j\in\cS^c} \bmu_j |\zb_j|^2 \,, \qquad \bl=\pa{\bl_j}_{j\in\cS}, \,\bmu=\pa{\bmu_j}_{j\in\cS^c}\in \ell^\infty\,.
	\end{equation}

	\smallskip

	\begin{defn}[Diophantine condition] \label{dioellittico} We say that a vector $\betta\in\pan$ belongs to ${\mathtt D}_{\g,\cS}$ if it satisfies
		\begin{equation}\label{diofantino nu}
		|\betta\cdot \ell| \ge  \gamma \prod_{n\in \Z}\frac{1}{(1+|\ell_n|^2 \jap{n}^{2})}\,,\quad \forall \ell\in \Z^\Z: \ell\ne 0\,,\; |\ell|:=\sum_i|\ell_i|<\infty \,,\quad \sum_{j\in\cS^c} |\ell_j|\le 2\,.
		\end{equation}
		 Note that $\dg\subset {\mathtt D}_{\g,\cS}$ (recall \eqref{diofantinozero}).
		 We also call $\mathtt D_{\g,\cS}^0$ the set of $\betta\in\pan$
		 satisfying \eqref{diofantino nu} only for $\ell$ with zero momentum,
		 namely $\pi(\ell)=0.$
	Clearly $\mathtt D_{\g,\cS}\subset \mathtt D_{\g,\cS}^0.$
	\end{defn}

	\begin{thm}[à la Herman--F\'ejoz]
		\label{alla Herman bis} 
		Given any $\cS\subset \Z$,  if $I=(\Ib,0)$ then
		Theorem \ref{allaMoserbis}  holds word by word with
		 $\omega\in \mathtt D_{\g,\cS},$ 
		 $\bar{\epsilon}\rightsquigarrow (1+\Theta)^{-1}\bar{\epsilon}$
		 and $\bar{C}\rightsquigarrow (1+\Theta)\bar{C}$.
\end{thm}
\begin{rmk}
	 The main point in this result is that if some of the actions are zero (say all those supported on $\cS^c$) then one may impose {\it weaker diophantine conditions}, namely 
	 $\omega\in \mathtt D_{\g,\cS}$, instead of $\omega\in \mathtt D_{\g}$.
\end{rmk}

\begin{proof}
	We start with a Hamiltonian
	$H_0= D_\betta + \Lambda +G_0$
with $\Lambda\in \ell^\infty$. 
We need the corresponding of the iterative Lemma \ref{iterativo}, to construct  a change of variables and a counterterm such that
\[
\Pi^{\le 0} e^{\set{S,\cdot}} \pa{D_{\betta} +\Lambda+G }= D_{\betta}
\]
As in Lemma \ref{iterativo}, at the $n$'th step we have an expression of the form
\[
H_n=D_{\betta} +\pa{\id+\call_n}\Lambda_n+G_n
\]
with $G_n\in \cH_{r_n,s_n,\eta_n}$,
\\
The generating function $S_n$ and the counterterm $\Lambda_n$ are fixed as the unique solutions of the Homological equation
\[
\Pi^{\le 0}  \pa{\set{S_n,D_{\betta}+G_n^{\ge 1}} +\pa{\id+\call_n}\bar\Lambda_n+G_n }=0
\]
As before this equation can be written componentwise as a triangular system and solved consequently. We have
\begin{align}
	& {\set{S_n^{(-2)},D_{\betta}} +\Pi^{-2,\cR}\call_n\bar\Lambda_n+G^{(-2,\cR)}_n }=0\\
	&\set{S_n^{(-1)},D_{\betta}}+\Pi^{-1}\set{S_n^{(-2)},G_n^{\ge 1}} +\Pi^{-1}\call_n\bar\Lambda_n+G^{(-1)}_n =0\\
	&\Pi^{0,\cK}\set{S_n^{(-2)}+ S_n^{(-1)},G_n^{\ge 1}} +\bar{\Lambda}_n+\Pi^{0,\cK}\call_n\bar\Lambda_n+G^{(0,\cK)}_n=0\\
	&\set{S_n^{(0,\cR)},D_{\betta}}+\Pi^{0,\cR}\set{S_n^{(-2)}+ S_n^{(-1)},G_n^{\ge 1}} +\Pi^{0,\cR}\call_n\bar\Lambda_n+G^{(0,\cR)}_n = 0
\end{align}
Now we solve
\begin{align}
	& S_n^{(-2)}= L_{\betta}^{-1} \pa{\Pi^{-2}\call_n\bar\Lambda_n+G^{(-2)}_n }\label{ostrica}\\
	&S_n^{(-1)}= L_{\betta}^{-1}\pa{\Pi^{-1}\set{L_{\betta}^{-1} \pa{\Pi^{-2}\call_n\bar\Lambda_n+G^{(-2)}_n },G_n^{\ge 1}} +\Pi^{-1}\call_n\bar\Lambda_n+G^{(-1)}_n}\nonumber 
\end{align}
Then we solve for $\bar\Lambda_n$
\begin{align*}
	&\Pi^{0,\cK}\set{L_{\betta}^{-1} \pa{\Pi^{-2}\call_n\bar\Lambda_n  +\Pi^{-1}\set{L_{\betta}^{-1} {\Pi^{-2}\call_n\bar\Lambda_n },G_n^{\ge 1}}},G_n^{\ge 1}} +\bar{\Lambda}_n+\Pi^{0,\cK}\call_n\bar\Lambda_n=\\
	&\Pi^{0,\cK}\set{L_{\betta}^{-1} \pa{G^{(-2)}_n  +\Pi^{-1}\set{L_{\betta}^{-1} {G^{(-2)}_n },G_n^{\ge 1}}},G_n^{\ge 1}}\,.
\end{align*}
As in the previous case this amounts to showing that the operator $M_n: \ell^\infty\to \ell^\infty$ defined as
\[
M_n h = \Pi^{0,\cK}\set{L_{\betta}^{-1} \pa{\Pi^{-2}\call_n h  +\Pi^{-1}\set{L_{\betta}^{-1} {\Pi^{-2}\call_n h },G_n^{\ge 1}}},G_n^{\ge 1}} +\Pi^{0,\cK}\call_n h
\]
satisfies an estimate of the type
\[
\norma{M_n h}_{\infty} \le \frac12\norma{ h}_{\infty}.
\]
The proof of this last bound follows just like the corresponding Lemma \ref{camel}.
Then
 \[
\bar{\Lambda}_n =(\id +M_n)^{-1}(\Pi^{0,\cK}\set{L_{\betta}^{-1} \pa{G^{(-2)}_n  +\Pi^{-1}\set{L_{\betta}^{-1} {G^{(-2)}_n },G_n^{\ge 1}}},G_n^{\ge 1}})
\]
 is fixed. Now we substitute in the equations \eqref{ostrica}, compute $S^{(-2)}_{n}$ and $S^{(-1)}_n$ and finally $S^{0,\cR}_n$. 
The estimates follow exactly as in Section 6.
\end{proof}

Let us discuss the consequences of Theorem \ref{alla Herman bis}  on the NLS equation \eqref{NLS}. 
Rewrite the NLS Hamiltonian \eqref{hamNLS}
in the form
\begin{equation}\label{epiro}
H_{\rm NLS}= \sum_{j\in \cS} \pa{j^2+ \Vb_j}|\ub_j|^2 + \sum_{j\in \cS^c} (j^2+ \Wb_j) |\zb_j|^2 + P\,
\end{equation}
where 
 $\Vb=\pa{\Vb_j}_{j\in \cS}=\pa{V_j}_{j\in \cS}$ are free parameters while 
$\Wb=\pa{\Wb_j}_{j\in \cS^c}=\pa{V_j}_{j\in \cS^c}$ are fixed.

In order to keep the proof as simple as possible, we shall avoid technical issues related to double eigenvalues by assuming that $P$ preserves momentum,  namely that $f(x,|u|^2)$ in \eqref{NLS} does not depend directly on $x$.
Regarding the parameters $\Wb_j$ the only assumption is that if $0\in\cS^c$ then one has $\Wb_0\ne 0$. We reformulate Theorem \ref{torelloellitticointro} in a more precise way as follows.

\begin{thm}\label{torelloellittico}
Fix $\g>0$ and consider the momentum preserving Hamiltonian $H_{\rm NLS}$ in
 \eqref{hamNLS}-\eqref{epiro}. 
Assume that $r $ satisfies \eqref{cornettone} 
 and take
 $\Ib = (\Ib_j)_{j\in\cS},$
$\Ib_j> 0,$
 such that $(\sqrt{\Ib},0)\in {\bar B}_{r}(\tw_{s+\s})$.
 There exists a Lipschitz map 
\begin{equation}
\label{omegone}
\bO: \cQ_\cS \to \cQ_{\cS^c}\,,\qquad \|\pa{\bO_j-j^2}_{j\in \cS^c}\|^{\g,\cQ_{\cS}} \le C \g \epsilon\,,
\end{equation}
for some $C>1$
such that the following holds.
\\
For all $\bo\in \cQ_{\cS}$ such that $(\bo,\bO(\bo))\in {\mathtt D}^0_{\g,\cS}$ there exist
 $\Vb=\Vb(\bo,\Ib)\in \ell^\infty$ and a symplectic change of variables  $\Psi:{B}_{r}(\tw_s)\to {B}_{2r}(\tw_s)$   such that
\[
H_{\rm NLS}\circ \Psi =  \sum_{j\in \cS} \bo_j |\ub_j|^2 + \sum_{j\in \cS^c} \bO_j(\bo) |\zb_j|^2 + R\,,\quad R\in \cH^{\ge 2}
\]
\end{thm}

\begin{proof}
	In order to apply Theorem \ref{alla Herman bis}, we  recall the estimates \eqref{stimazero} on $P$. For $\bo\in \mathtt D_{\g,\cS}$, let us  write the Hamiltonian $H_{\rm NLS}$  as 
	\begin{equation*}\label{NLS hermanlow}
	H_{\rm NLS} = \sum_{j\in \cS} \pa{j^2+ \Vb_j}|\ub_j|^2 + \sum_{j\in \cS^c} (j^2+ \Wb_j) |\zb_j|^2 + P = D_\bo +\Lambda + P'\,,
	\end{equation*}
	where 
	\begin{equation*}\label{counter NLSlow}
\Lambda=  \sum_{j\in\cS}\bl_j \pa{\modi{j} - \Ib_j} + \sum_{i\in\cS^c} \bmu_j\abs{\zb_j}, \quad \bl_j = j^2-\bo_j+\Vb_j,\;
 \bmu_j=  j^2-\bO_j +\Wb_j\,,\qquad P' = P + \sum_{j\in\cS} \bl_j \Ib_j.
	\end{equation*}
	Of course $P'$ satisfies \eqref{stimazero} by definition of the norm and,
by construction, $\Lambda\in \ell^\infty$. By the definition of $r_0$ one can apply Theorem \ref{alla Herman bis} to  $H = N_0 + P'$ where $N_0 = D_{\bo,\bO}$. In order to prove that $H_{\rm NLS}\circ\Phi= N$ we require that $H_{\rm NLS}= H+\Lambda(\Ib,\bo,\bO)$ fixed in Theorem \ref{alla Herman bis}. This amounts to requiring
\begin{equation}
\label{equaOme}
\begin{cases}
&\bO_j+ \bmu_j(\Ib,\bo, \bO) = j^2+\Wb_j\\
&\bo_j +\bl_j(\Ib,\bo, \bO)= j^2+\Vb_j \,.
\end{cases}
\end{equation}
In order to solve the equations above, the only key point is to extend the map $\bmu: \,\mathtt D^0_{\g,\cS} \to \ell_\infty$   to the whole square $\pan$ preserving the  wheighted Lipschitz norm $\|\cdot\|$ .  This in fact is guaranteed  by  Kirtzbraun theorem on metric spaces (see for instance \cite{MP}) 
\\
At this point by direct application of the contraction Lemma, we solve $\bO= \bO(\bo)$ 
which by construction is a Lipschitz map  and satisfies \ref{omegone}.
Finally we set
$
\Vb_j=\bo_j +\bl_j(\bo, \bO(\bo))- j^2\,.
$
\end{proof}
Of course in order for  Theorem \ref{torelloellittico} to be non empty we need {\it measure estimates}, namely we need to show the following result, whose proof
is postponed to Appendix \ref{righello}. 
\begin{lemma}[Measure estimates]\label{erpupone}
	Let $\bO$ be a map satisfying \eqref{omegone}. Then the set
	\[
	\cC:= \set{ \bo\in \cQ_{\cS}\;:\quad 
	\omega(\bo)=(\bo,\bO(\bo))\in \mathtt D^0_{\g,\cS}}
	\]
	has positive relative measure in $\cQ_{\cS}$.
\end{lemma}

The following result states  that  the frequencies  $(\bo,\bO(\bo))$ with $\bo\in \cC$
satisfied the so called Melnikov non-resonance conditions in the following way
	\begin{lemma}\label{sanguinaccio}
	Let $ \mathtt M_\g$ the set of $\bo\in\cQ_\cS$ such that
	\[
|\bo\cdot  h +\s\bO_j(\bo) + \s' \bO_k(\bo)|> \gamma \prod_{n\in \cS}\frac{1}{(1+|h_n|^6 \jap{n}^{6})} \,,
	\]
	for every
	$ h\in \Z^\cS$, $j,k\in\cS^c$, $\s,\s'=\pm 1,0$ satisfying  $\pi(h) +\s j+\s' k =0$.
	Then $\cC\subseteq \mathtt M_\g$.
	\end{lemma}
	The proof is postponed to  Appendix \ref{righello}.

\appendix
\section{Technicalities}
\label{appendice tecnica}

\noindent
\subsection{ Proof of Lemma \ref{elisabetta}.}
Recalling \eqref{burrata} we get
\begin{eqnarray*}
&&
\frac{1}{r^2}
 \sum_{|\bal|=|\bbt|> 0}
 \sup_{|u|_s\leq r}
  \abs{{H}_{\bal,\bbt}}e^{\eta|\pi(\bal-\bbt)|}\abs{\buu}
 \stackrel{\eqref{burrata}}= 
 \frac{1}{r^2}
 \sum_{|\bal|=|\bbt|> 0} \abs{{H}_{\bal,\bbt}}e^{\eta|\pi(\bal-\bbt)|}
 u_0^{\bal+\bbt}
 \\
 &&
 \leq
  \frac{1}{r^2}
   \sup_{j\in\Z}
 \sum_{|\bal|=|\bbt|> 0}
 \bbt_j \abs{{H}_{\bal,\bbt}}e^{\eta|\pi(\bal-\bbt)|}
 u_0^{\bal+\bbt}
 \\
 &&
 =
  \sup_{j\in\Z}
  \jap{j}^{-2p} e^{-2 a\abs{j}- 2s\jap{j}^\theta}
 \sum_{\bal,\bbt} \abs{{H}_{\bal,\bbt}}e^{\eta|\pi(\bal-\bbt)|}
\bbt_j u_0^{\bal+\bbt-2 e_j}
\\
 &&
 \leq
  \sup_{j\in\Z}
 \sum_{\bal,\bbt} \abs{{H}_{\bal,\bbt}}e^{\eta|\pi(\bal-\bbt)|}
\bbt_j u_0^{\bal+\bbt-2 e_j}
 \,.
\end{eqnarray*}
Then \eqref{abacab} follows by the mass conservation.
As a consequence $\und H_\eta(u)$ and, a fortiori, 
$H(u)$ are analytic functions on the open ball
$\{|u|_s<r\}$ 
 and continuous on the closed ball $\{|u|_s\leq r\}.$
Analogously for the hamiltonian vector fields
$X_{\und H_\eta}$  and $X_H.$
Indeed it is easily seen that
	\[
	X_{\und H_\eta}^{(j)}(u) = \im \sum_{\bal,\bbt\in\N^\Z} \abs{H_{\bal,\bbt}}\bbt_j e^{\eta|\pi(\bal-\bbt)|}u^\bal \bar u^{\bbt-e_j}
	\]
	and, therefore, for every $|u|_{p,s,a}\leq r$
\begin{eqnarray*}
|X_{\und H_\eta}^{(j)}(u)|
&\leq&
\sum_{\bal,\bbt\in\N^\Z} \abs{H_{\bal,\bbt}}\bbt_j e^{\eta|\pi(\bal-\bbt)|}
|u^\bal| |\bar u^{\bbt-e_j}|
\\
&\stackrel{\eqref{burrata}}\leq&
\sum_{\bal,\bbt\in\N^\Z} \abs{H_{\bal,\bbt}}\bbt_j e^{\eta|\pi(\bal-\bbt)|}
u_0^{\bal+\bbt-e_j}
=|X_{\und H_\eta}^{(j)}(u_0)|\,,
	\end{eqnarray*}
	proving the first equality in \eqref{normatris}.
Then
$$
\frac1r
\norm{{X}_{{\underline H}_\eta}(u_0(r))}_{s,a,p}
=
  \sup_j  
\sum_{\bal,\bbt\in\N^\Z} \abs{H_{\bal,\bbt}}\bbt_j u_0^{\bal + \bbt - 2e_j}e^{\eta|\pi(\bal-\bbt)|}\,,
$$	
concluding the proof of the lemma. \qed

\medskip

\medskip

\noindent
\subsection{Proof of Lemma \ref{appendicite}}
We have\footnote{
		Use that $\binom{a}{b}\leq \frac{a^b}{b!}.$
		Moreover
		$|x|^q=\displaystyle\Big(\sum_{1\leq i\leq n}x_i\Big)^q=\sum_{\nu\in\N^n, \, |\nu|=q}\frac{q!}{\nu!} x^\nu.$
		Note that, being $\delta\preceq m$, the support of $\delta$ is contained in the support of $m$. Finally we use that
		$q^q\leq e^q q!$
	}
	$$
	\sum_{\substack{ |\delta|={q}\\ \delta\preceq m}}
	\binom{m}{\delta}
	=
	\sum_{\substack{ |\delta|={q}\\ \delta\preceq m}}
	\prod_{i\in\Z} 
	\binom{m_i}{\delta_i}
	\leq
	\sum_{\substack{ |\delta|={q}\\ \delta\preceq m}}
	\prod_{i\in\Z} 
	\frac{m_i^{\delta_i}}{\delta_i!}
	=
\sum_{\substack{ |\delta|={q}\\ \delta\preceq m}}
	\frac{m^\delta}{\delta!}
	\leq 
	\frac{|m|^{{q}}}{q!}
	\leq 
	\pa{\frac{e|m|}{q}
	}^{{q}}
	$$
and also
$$
\sum_{\substack{ |\delta|={q}\\ \delta\preceq m}}
	\prod_{i\in\Z} 
	\binom{m_i}{\delta_i}
	\leq
	\sum_{\delta\preceq m}
	\prod_{i\in\Z} 
	\binom{m_i}{\delta_i}=2^{|m|}\,,
$$
so that
\[
\sum_{\substack{ |\delta|={q}\\ \delta\preceq m}} 
 \binom{m}{\delta}  
 \le \min\left\{\pa{\frac{e|m|}{q}
}^{{q}},2^{|m|}\right\}\,.
\]
Then, for $|m|\geq q$, we get
$$
\kappa^{2|m|}
\sum_{\substack{ |\delta|={q}\\ \delta\preceq m}} 
 \binom{m}{\delta} 
  \le 
  \min\left\{
  \max_{|m|\geq q}\pa{\frac{e|m|}{q}
}^{{q}} \kappa^{2|m|},\ 
 \max_{|m|\geq q}\pa{2\kappa^2}^{|m|}\right\}
\leq 
c_\kappa^q\,.
$$
\qed



\noindent
\subsection{Proof of Lemma \ref{camel}}\label{chitelida}
	We treat the two summands of $M_n$ separately, we recall
	that by \eqref{ritornello} 
		\[
	\|{\pon \mathcal{L}_n} h\|_{\infty}\le 3 \|{ \mathcal{L}_n} h\|_{r_n,s_n,\eta_n}\le 3 \tK (1+\Theta)\e_0\sum_{j=1}^n 2^{-j} \|h\|_{\infty}< \frac14 \|h\|_{\infty}.
	\]
	provided that \[ 3 \tK (1+\Theta)\e_0<1/4\,.\] As for the second summand we have,
	again by \eqref{ritornello}, Lemma \ref{Lieder} and \eqref{emiliaparanoica}
	\begin{align*}
	&	\| \pon \set{ \pd L^{-1}_\omega \mathcal{L}_n{h},G_n^{\geq 2}}\|_{\infty} \le 3
	\|\set{ \pd L^{-1}_\omega \mathcal{L}_n{h},G_n^{\geq 2}}\|_{r_n/2,s_n+\eta_n,0} \\
	&
  \stackrel{\eqref{commXHK}}\le 
  48\norma{\pd L^{-1}_\omega \mathcal{L}_n{h}}_{r_n,s_n+\eta_n,0}
  \norma{G_n^{\geq 2}}_{r_n,s_n+\eta_n,0}\le 48 \g \norma{ L^{-1}_\omega \mathcal{L}_n{h}}_{r_n,s_n+\eta_n,0}\Theta_n\\
  &
 \le  48 
 e^{\croc{{\eta_n}}^{-3/\teta}}\norma{  \mathcal{L}_n{h}}_{r_n,s_n,\eta_n} \Theta_n\leq 96 \Theta_0 \tK \e_0(1+\Theta_0) e^{\croc{{\eta_n}}^{-3/\teta}}\|h\|_{\infty}
	\end{align*}
	noting that $\eta_n\geq \eta_0/2$ by \eqref{pesto} and \eqref{newton}.
\qed

\section{On action-angle coordinates}\label{aa}
\begin{lemma}
	The set 
	\begin{equation}
	\label{perluno}
	U := \set{ u\in\tw^{\infty}_{p,s,a}\, : \, \inf_j \jap{j}^{ p}e^{a\abs{j}+s\jap{j}^\theta}| u_j|>0 }\,
	\end{equation}
	is open and dense in $\tw^{\infty}_{p,s,a}$.\\	
	Moreover, the map to action angles
	\[
	\Phi:\;\T^\Z\times\mathtt{w}^\infty_{2s, 2a, 2p}(\R_+) \to \mathtt{w}^\infty_{p,s,a}(\C), \quad (\theta,J)  \mapsto  u = \Phi(J,\theta):= \pa{\sqrt{J_j}e^{\im \theta_j}}_{j\in\Z}
	\]
	where $\T^\Z=(\R/2\pi\Z)^\Z$ is endowed with the norm:
	\[
	|\theta-\theta'|_\infty= \sup_j |\theta_j-\theta'_j|_{\rm{mod} 2\pi}\,
	\]
	is locally well defined in some ball ${\bar B}_r(\hat u,\tw_{p,s,a}^\infty )$ for all $\hat u\in U$.
	
\end{lemma}
\begin{proof}
	Let us fix a $\hat u\in U$, then by definition 
	\[
	\inf_j \jap{j}^{ p}e^{a\abs{j}+s\jap{j}^\theta}| \hat u_j|= \mathtt r(\hat u)>0\,,
	\]
moreover for any $r< \mathtt r(\hat u)$ the ball ${\bar B}_\mathtt r(\hat u)\subset U$, indeed 
\[
\inf_j \jap{j}^{ p}e^{a\abs{j}+s\jap{j}^\theta}| u_j| \ge \inf_j \jap{j}^{ p}e^{a\abs{j}+s\jap{j}^\theta}| \hat u_j| - \sup_j \jap{j}^{ p}e^{a\abs{j}+s\jap{j}^\theta}| \hat u_j- u_j|\ge \mathtt r(\hat u) -r>0\,.
\]	
	In order to prove the continuity of the action/angle map we start by recalling that we may identify  isometrically $\tw_{p,s,a}^\infty(\C)$ with $\ell^\infty(\C)$ via the map
\[
i_{p,s,a}:\; \tw_{p,s,a}^\infty(\C)\mapsto \ell^\infty(\C)\,,\quad i_{p,s,a} u = \jap{j}^{ p}e^{a\abs{j}+s\jap{j}^\theta} u\,.
\]
and the same map $i_{2s,2a,2p}$  identifies $\tw_{2s,2a,2p}^\infty(\R_+)$ with $\ell^\infty(\R_+)$. Finally the diagram

\centering{	\begin{tikzpicture}
	\matrix (m) [matrix of math nodes,row sep=3em,column sep=4em,minimum width=2em]
	{
		\T^\Z\times\mathtt{w}^\infty_{2a,2s, 2p}(\R_+) & \mathtt{w}^\infty_{a,s, p}(\C) \\
		\T^\Z\times\ell^\infty(\R_+) & \ell^\infty(\C) \\};
	\path[-stealth]
	(m-1-1) edge node [right] {$i_{2a,2s,2p}$} (m-2-1)
	edge  node [below] {$\Phi$} (m-1-2)
	(m-2-1.east|-m-2-2) edge node [below] {$\Phi$}(m-2-2)
	(m-1-2) edge node [right] {$i_{a,s,p}$} (m-2-2);
	\end{tikzpicture}}
\\
is commutative so we just need to prove the statement for $\ell^\infty(\C)$ where it is trivial.
Indeed fix $\hat u= \Phi(\hat J,\hat \theta)\in \ell^\infty$ and consider the preimage through $\Phi$ of the ball ${\bar B}_r(\hat u, \ell^\infty(\C))$, i.e. 
\[
\set{(J,\theta)\in  \T^\Z\times\ell^\infty(\R_+): \sup_j|\sqrt{J}e^{\im \teta}-\sqrt{\hat J}e^{\im \hat{\teta}}|\le r }
\]
then provided that 
we  assume that $3 r |\hat u|^\infty < \mathtt r(\hat u)$ the ellipses
\[
{\bar B}_1(r):= \set{(J,\theta):\quad |J-\hat{J}|_\infty < 3 r |\hat u|_\infty \,,\quad  |\theta- \hat\teta|_\infty < \arctan(\frac{r}{\mathtt r(\hat u)})}
\]
\[
{\bar B}_2(r):= \set{(I,\theta):\quad |J-\hat{J}|_\infty <  r\frac{ \mathtt r(\hat u)}{2} \,,\quad  |\theta- \hat\teta|_\infty < \arctan(\frac{r}{|\hat u|_\infty})}
\]
satisfy
\[
\Phi({\bar B}_2(r))\subseteq {\bar B}_r(\hat u,\ell^\infty) \subseteq \Phi({\bar B}_1(r))
\]
This implies that $\Phi$ is a homeomorphism. 
\end{proof}

\section{Measure estimates}\label{righello}

\noindent
\subsection{Proof of Lemma \ref{erpupone}}

For $\ell\ne 0$ with $|\ell|<\infty$, $\sum_{j\in \cS^c} |\ell_j|\le 2$ and $\pi(\ell)=0$,
we define the 
 resonant set 
	\[
\cR_\ell:=\set{\bo\in \cQ_{\cS}:	|\betta(\bo)\cdot \ell|\le \gamma \prod_{n\in \Z}\frac{1}{(1+|\ell_n|^2 \jap{n}^{2})}}
	\]
	We first note that 
	  if $\ell$ is  supported only on $\cS^c$  (i.e. $\ell_j=0$ for all $j\in \cS$) then 
	 $\cR_\ell$ is empty. Indeed in this case $\ell=\s \mathbf e_j+\s' \mathbf e_k$
	 for some $j,k\in\cS^c$, $|j|\geq |k|$, $\s=\pm 1,$ $\s'=\pm1,0,$
	 $\s j+\s' k=0.$ 
	 Then $\s\s'\neq -1,$ otherwise, since the momentum is zero, we get
	 $j=k$ and $\ell=0$.
	 Consider first the case $\s\s'=1;$ then
	 $k=-j$ and
	   $$
	|\omega(\bo)\cdot\ell|=   |\bO_j +\bO_k  | \stackrel{\eqref{equaOme}}{\ge} |j^2+ k^2|-(|\Wb_j|+|\Wb_k|)- O(\epsilon)  \ge 2j^2-1-O(\epsilon) \,,
	   $$
	   which, if 
	   $j\neq 0$, is bigger than $1/2$ and, therefore,  $\cR_\ell$ is empty whenever $\g<1/2$.
	 Otherwise, when $j=k=0$ we get
	  $$
	|\omega(\bo)\cdot\ell|\geq 2|\Wb_0|- O(\epsilon)  \ge |\Wb_0| \,,
	   $$
	   then $\cR_\ell$ is empty provided $\g$ is small enough
	   with respect to $|\Wb_0|,$ that we are assuming to be 
	   different from zero (since we are in the case $0\in\cS^c$).
	   \\
	   It remains only the case $\s\s'=0,$ namely $\s'=0.$
	   Then $|\omega(\bo)\cdot\ell|=   |\bO_j|$ and we conclude as above.

	 On the other and, if $\ell_s\ne0$ for some $s\in \cS$, then we can bound from below the Lipschitz variation in the direction $s$. Indeed  (recalling that $|\bO_j(\bo)|^{\rm lip} \sim \epsilon$) one has
\[
|\Delta_{\bo_s} \betta(\bo)\cdot \ell |\ge |\ell_s|- 2\sup_{j\in \cS^c}|\bO_j(\bo)|^{\rm lip}  \ge  |\ell_s|-O(\epsilon) \ge \frac12
\]
if $\epsilon$ is small enough. Then following the proof of  Lemma 4.1 of \cite{BMP1:2018} verbatim one gets 
\[
\meas(\cQ_\cS\setminus \cC)\le \sum_{\ell\ne 0}\meas(\cR_\ell) \le \gamma \sum_{\ell\ne 0}\prod_{n\in \Z}\frac{1}{(1+|\ell_n|^2 \jap{n}^{2})} \sim O(\g)\,.
\]
\qed

\noindent
\subsection{Proof of Lemma \ref{sanguinaccio}}

Let us consider $\bo\in \cC$  and estimate
	\[
	|\bo\cdot  h +\s\bO_j(\bo) + \s' \bO_k(\bo)|= |\betta(\bo)\cdot\ell| \,,\quad \ell= (h, \s \be_j+\s'\be_k)
	\]
	We start by remarking that, since $\bo\in\cQ_\cS$ and $\bO_j(\bo)-\Wb_j -j^2 \sim \g\epsilon$,  one has
		\[
	|\bo\cdot  h +\s\bO_j(\bo) + \s' \bO_k(\bo)|\ge  j^2 +\s\s' k^2 -2 \sum_{i\in \cS} i^2 |h_i| 
	\]
thus unless $\s\s'=-1$  and $j=-k$ we can deduce that if $|j|,|k| \ge C \sum_{i\in \cS} i^2 |h_i|$ then  the left hand side above is 
	$\ge \frac12$ and  hence the conditions defining $\mathtt M_\g$ are trivially met.
	On the other hand if $|j|,|k| \le C \sum_{i\in \cS} i^2 |h_i|$ then
\begin{align*}
\label{meme}
|\bo\cdot  h +\s\bO_j(\bo) + \s' \bO_k(\bo)|& = |\betta(\bo)\cdot\ell| \ge \gamma \frac{1}{(1+j^2)(1+k^2)} \prod_{n\in \cS}\frac{1}{(1+|h_n|^2 \jap{n}^{2})} \\  & \ge \gamma \frac{\gamma}{(1+\sum_{i\in \cS} i^2 |h_i|)^2} \prod_{n\in \cS}\frac{1}{(1+|h_n|^2 \jap{n}^{2})} \ge \gamma  \prod_{n\in \cS}\frac{1}{(1+|h_n|^6 \jap{n}^{6})}\,.
\end{align*}	
Finally if $\s\s'=-1$ $j=-k$  we use momentum conservation to deduce
\[
2|j| \le |\pi(h)|\le \sum_{i\in\cS} |i| |h_i|.  
\]
This conclude the proof. \qed


\begin{thebibliography}{dlLGJV05}

\bibitem[BHTB90]{Broer-Huitema-Takens:1990}
H.~W. Broer, G.~B. Huitema, F.~Takens, and B.~L.~J. Braaksma.
\newblock Unfoldings and bifurcations of quasi-periodic tori.
\newblock {\em Mem. Amer. Math. Soc.}, 83(421):viii+175, 1990.

\bibitem[BMP19]{BMP1:2018}
L.~Biasco, J.E. Massetti, and M.~Procesi.
\newblock An {A}bstract {B}rkhoff {N}ormal {F}orm {T}heorem and exponential
  type stability of the 1d {NLS}.
\newblock {\em Communications in Mathematical Physics}, 2019.
\newblock doi: 10.1007/s00220-019-03618-x.

\bibitem[Bou05]{Bourgain:2005}
J.~Bourgain.
\newblock On invariant tori of full dimension for 1{D} periodic {NLS}.
\newblock {\em J. Funct. Anal.}, 229(1):62--94, 2005.

\bibitem[CGP11]{CGP}
L.~Corsi, G.~Gentile, and M.~Procesi.
\newblock K{AM} theory in configuration space and cancellations in the
  {L}indstedt series.
\newblock {\em Comm. Math. Phys.}, 302(2):359--402, 2011.

\bibitem[Che85]{Chenciner:1985}
A.~Chenciner.
\newblock Bifurcations de points fixes elliptiques. {I}. {C}ourbes invariantes.
\newblock {\em Inst. Hautes {\'E}tudes Sci. Publ. Math.}, 61:67--127, 1985.

\bibitem[CLSY]{Yuan_et_al:2017}
H.~Cong, J.~Liu, Y.~Shi, and X.~Yuan.
\newblock The stability of full dimensional kam tori for nonlinear schrödinger
  equation.
\newblock preprint 2017, arXiv:1705.01658.

\bibitem[dlLGJV05]{Llave}
R.~de~la Llave, A.~Gonzalez, A.~Jorba, and J.~Villanueva.
\newblock {KAM} theory without action-angles.
\newblock {\em Nonlinearity}, 18(2):855--895, 2005.

\bibitem[EFK13]{EFK:2013}
L.~H. Eliasson, B.~Fayad, and R.~Krikorian.
\newblock K{AM}-tori near an analytic elliptic fixed point.
\newblock {\em Regul. Chaotic Dyn.}, 18(6):801--831, 2013.

\bibitem[EK10]{EK10}
L.~H. Eliasson and S.~B. Kuksin.
\newblock K{AM} for the nonlinear {S}chr\"odinger equation.
\newblock {\em Ann. of Math. (2)}, 172(1):371--435, 2010.

\bibitem[F{\'e}j04]{Fejoz:2004}
J.~F{\'e}joz.
\newblock D\'emonstration du "th\'eor\`eme d'\upp{A}rnold" sur la stabilit\'e
  du syst\`eme plan\'etaire (d'apr\`es \upp{M}ichael \upp{H}erman).
\newblock {\em Michael Herman Memorial Issue, Ergodic Theory Dyn. Syst},
  (24:5):1521--1582, 2004.

\bibitem[F{\'e}j05]{Fejoz:oberwolfach}
J.~F{\'e}joz.
\newblock About {M}. {H}erman’s proof of ‘{A}rnold’s theorem’ in
  celestial mechanics.
\newblock In {\em Mathematisches Forschungsinstitut Oberwolfach}, pages
  1767--1770. 2005.
\newblock Report No. 31/2005.

\bibitem[Feo15]{Feo15}
R.~Feola.
\newblock {KAM} for a quasi-linear forced {H}amiltonian {NLS}.
\newblock arXiv:1602.01341, 2015.

\bibitem[FK09]{Fayad-Krikorian:2009}
B.~Fayad and R.~Krikorian.
\newblock Herman's last geometric theorem.
\newblock {\em Ann. Sci. \'Ec. Norm. Sup\'er. (4)}, 42(2):193--219, 2009.

\bibitem[GX13]{GX13}
Jiansheng Geng and Xindong Xu.
\newblock Almost periodic solutions of one dimensional {S}chr{\"o}dinger
  equation with the external parameters.
\newblock {\em J. Dynam. Differential Equations}, 25(2):435--450, 2013.

\bibitem[HS71]{Herman:1971}
M.~Herman and F.~Sergeraert.
\newblock Sur un th\'{e}or\`eme d'{A}rnold et {K}olmogorov.
\newblock {\em C. R. Acad. Sci. Paris S\'{e}r. A-B}, 273:A409--A411, 1971.

\bibitem[KP96]{Kuksin-Poschel:1996}
S.~Kuksin and J.~P{\"o}schel.
\newblock Invariant {C}antor manifolds of quasi-periodic oscillations for a
  nonlinear {S}chr{\"o}dinger equation.
\newblock {\em Ann. of Math. (2)}, 143(1):149--179, 1996.

\bibitem[Mas18]{Massetti:APDE}
J.~E. Massetti.
\newblock A normal form \`a la {M}oser for diffeomorphisms and a generalization
  of {R}\"{u}ssmann's translated curve theorem to higher dimensions.
\newblock {\em Anal. PDE}, 11(1):149--170, 2018.

\bibitem[Mas19]{Massetti:ETDS}
J.~E. Massetti.
\newblock Normal forms for perturbations of systems possessing a {D}iophantine
  invariant torus.
\newblock {\em Ergodic Theory Dynam. Systems}, 39(8):2176--2222, 2019.

\bibitem[MP]{MP19}
R.~Montalto and M.~Procesi.
\newblock {L}inear {S}chr\"odinger equation with an almost periodic potential.
\newblock preprint, arXiv: 1910.12300.

\bibitem[MP18]{MP}
A.~Maspero and M.~Procesi.
\newblock Long time stability of small finite gap solutions of the cubic
  nonlinear {S}chr\"odinger equation on {$\Bbb T^2$}.
\newblock {\em J. Differential Equations}, 265(7):3212--3309, 2018.

\bibitem[P{\"o}s89]{Poschel:1989}
J.~P{\"o}schel.
\newblock On elliptic lower-dimensional tori in {H}amiltonian systems.
\newblock {\em Math. Z.}, 202(4):559--608, 1989.

\bibitem[P{\"o}s96]{Poschel:1996}
J.~P{\"o}schel.
\newblock A {KAM}-theorem for some nonlinear partial differential equations.
\newblock {\em Ann. Scuola Norm. Sup. Pisa Cl. Sci. (4)}, 23(1):119--148, 1996.

\bibitem[P{\"o}s01]{Poschel:2001}
J.~P{\"o}schel.
\newblock A lecture on the classical {KAM} theorem.
\newblock In {\em Smooth ergodic theory and its applications ({S}eattle, {WA},
  1999)}, volume~69 of {\em Proc. Sympos. Pure Math.}, pages 707--732. Amer.
  Math. Soc., Providence, RI, 2001.

\bibitem[P{\"o}s02]{Poschel:2002}
J.~P{\"o}schel.
\newblock On the construction of almost periodic solutions for a nonlinear
  {S}chr{\"o}dinger equation.
\newblock {\em Ergodic Theory Dynam. Systems}, 22(5):1537--1549, 2002.

\bibitem[Sev99]{Sevryuk:1999}
M.~B. Sevryuk.
\newblock The lack-of-parameters problem in the {KAM} theory revisited.
\newblock In {\em Hamiltonian systems with three or more degrees of freedom
  ({S}'{A}gar\'o, 1995)}, volume 533 of {\em NATO Adv. Sci. Inst. Ser. C Math.
  Phys. Sci.}, pages 568--572. Kluwer Acad. Publ., Dordrecht, 1999.

\bibitem[Way90]{W}
C.~E. Wayne.
\newblock Periodic and quasi-periodic solutions of nonlinear wave equations via
  {KAM} theory.
\newblock {\em Comm. Math. Phys.}, 127(3):479--528, 1990.

\end{thebibliography}


	\footnotesize

	\def\cprime{$'$}

\end{document}